\documentclass[a4paper,11pt]{article}
\usepackage{amsmath,amssymb,amsthm,mathtools}
\usepackage[hmargin={26mm,26mm},vmargin={30mm,35mm}]{geometry}
\usepackage[colorlinks,allcolors=blue]{hyperref}
\usepackage[compact]{titlesec}
\usepackage{enumitem}
\usepackage{newtxtext,newtxmath}

\usepackage{authblk}
\usepackage{thmtools,thm-restate}

\newtheorem{theorem}{Theorem}
\newtheorem{proposition}[theorem]{Proposition}
\newtheorem{lemma}[theorem]{Lemma}
\newtheorem{corollary}[theorem]{Corollary}
\theoremstyle{definition}
\newtheorem{definition}[theorem]{Definition}
\newtheorem{remark}[theorem]{Remark}

\newcommand{\Real}{\mathbb{R}}
\newcommand{\DIFF}{\mathrm{d}}
\newcommand{\KOSZUL}{\kappa}
\newcommand{\Mh}{\mathcal{M}_h}
\newcommand{\Sh}{\mathcal{S}_h}
\newcommand{\card}{\operatorname{card}}
\newcommand{\FM}[1]{\Delta_{#1}}
\newcommand{\cof}[1]{\operatorname{cof}_{#1}}
\newcommand{\PL}[2]{\mathcal{P}_{#1}\Lambda^{#2}}
\newcommand{\PLtrim}[2]{\mathcal{P}_{#1}^{-}\Lambda^{#2}}
\newcommand{\norm}[2]{\left\lVert #2\right\rVert_{#1}}
\newcommand{\uX}[2]{\underline{\boldsymbol{X}}_{\ell,#2}^{#1}}
\newcommand{\uI}[2]{\underline{\boldsymbol{I}}_{\ell,#2}^{#1}}
\newcommand{\uXh}[1]{\underline{\boldsymbol{X}}_{\ell,h}^{#1}}
\newcommand{\Qtrace}[2]{\mathcal Q_{\ell,#1}^{#2}}
\newcommand{\piQ}[2]{\pi_{\mathcal Q,#1}^{#2,\ell}}
\newcommand{\Dpoly}[2]{\mathcal D_{\ell,#1}^{#2}}
\newcommand{\piD}[2]{\pi_{\mathcal D,#1}^{#2,\ell}}
\newcommand{\trc}{\operatorname{tr}}
\newcommand{\Ker}{\operatorname{Ker}}
\newcommand{\epsTF}{\epsilon_{T,F}}

\newcommand{\PLtrimz}[2]{\mathring{\mathcal P}_{#1}^{-}\Lambda^{#2}}

\newcommand{\email}[1]{\href{mailto:#1}{#1}}

\allowdisplaybreaks

\newcommand{\vvvert}{\vert\kern-0.25ex\vert\kern-0.25ex\vert}

\begin{document}

\title{Uniform discrete Poincar\'e inequalities for Hybrid High-Order differential forms on polyhedral meshes}
\author[1]{Silvano Pitassi}
\affil[1]{IMAG, Univ Montpellier, CNRS, Montpellier, France\\ \email{silvano.pitassi@umontpellier.fr}}

\maketitle

\begin{abstract}
We introduce a Hybrid High-Order framework for differential forms on general polyhedral meshes and characterise stability-admissible face spaces in terms of the traces of the polynomial kernel of the exterior derivative.
The analysis is based on a cellular-to-hybrid transfer mechanism in which the global part of the stability problem is reduced to a single cellular cochain Poincar\'e problem, and the resulting control is propagated to the hybrid level using stable polynomial skeletons and local polynomial completions.

This mechanism yields a uniform stable conforming lifting and a uniform discrete Poincar\'e inequality for every form degree in arbitrary space dimension.
The lifting preserves the prescribed reconstructed exterior derivative after projection, whereas the Poincar\'e inequality controls the broken cell polynomial form modulo the continuous conforming kernel of the exterior derivative.
Both results hold on domains with arbitrary topology and uniformly over stability-admissible choices of the face spaces.
In three dimensions, the conforming-kernel estimate recovers the hybrid Poincar\'e--Wirtinger inequality for the gradient, the second hybrid Weber estimate under the standard gauge, and the corresponding divergence-kernel estimate.
\medskip\\
\textbf{Key words.} Hybrid High-Order methods, differential forms, conforming lifting, discrete Poincar\'e inequalities, Weber inequalities, polyhedral meshes, functional double complex, cellular cochains.
\medskip\\
\textbf{MSC2020.} 65N30, 65N12, 58A10, 58J10.
\end{abstract}

\section{Introduction}

This work brings together finite element differential forms and arbitrary-order hybrid methods on polyhedral meshes.
Finite element differential forms provide a geometric language for differential operators, Hilbert complexes, and topology within finite element exterior calculus \cite{Arnold.Falk.ea:06,Arnold.Falk.ea:10,Arnold:18}, while Hybrid High-Order methods offer flexible high-order discretisations on general polyhedra through polynomial unknowns attached to mesh cells and faces.
By combining these two viewpoints, we develop a stability theory that follows the common structure of the de Rham sequence and applies uniformly to all form degrees while retaining the geometric flexibility and static-condensation properties of hybrid discretisations.

\paragraph*{Background and motivation.}

Hybrid High-Order (HHO) methods were introduced as arbitrary-order polyhedral schemes combining polynomial unknowns attached to mesh cells and faces \cite{Di-Pietro.Ern.Lemaire:14,Di-Pietro.Ern:15}.
Their applicability to general meshes and arbitrary polynomial degrees, together with the possibility of eliminating cell unknowns locally by static condensation, has made HHO a broad framework for the approximation of partial differential equations.
For monographic and pedagogical presentations of the HHO framework, see \cite{Di-Pietro.Droniou:20,Cicuttin.Ern.Pignet:21,Di-Pietro.Droniou:25b}.

Much of the corresponding stability theory has nevertheless been developed operator by operator in the language of vector calculus.
For scalar diffusion, the basic HHO construction uses a local potential reconstruction whose composition with the natural interpolator is an elliptic projector, while stability rests on a discrete Poincar\'e inequality for hybrid unknowns \cite{Di-Pietro.Droniou:20}.
For incompressible flows, the reconstructed divergence and its commuting property provide a Fortin mechanism for the discrete inf--sup condition \cite{Di-Pietro.Droniou:20}.
For curl-type problems, discrete Weber inequalities on three-dimensional hybrid spaces were first established under trivial-topology assumptions in \cite{Chave.Di-Pietro.Lemaire:22} and subsequently extended to domains with general, possibly non-trivial topology in \cite{Lemaire.Pitassi:25}.
These estimates have since been used in arbitrary-order HHO approximations of first- and second-order div--curl systems on domains with general topology \cite{Dalphin.Ducreux.Lemaire.Pitassi:25}.

Taken together, these developments provide stability mechanisms for the classical vector-calculus operators, but through operator-dependent spaces, reconstructions, norms, and proofs.
A differential-form formulation offers a natural way to expose the geometric structure shared by these constructions.
In the fully cellular Discrete de Rham setting, uniform discrete Poincar\'e inequalities have been established for the original three-dimensional complex and subsequently for differential forms in arbitrary dimension \cite{Di-Pietro.Hanot:24,Di-Pietro.Droniou.ea:25}.
The challenge addressed here is to identify a common local-to-global stability mechanism for genuinely hybrid spaces, in which only top-dimensional cell unknowns and face traces are available, and to derive from it both a stable conforming lifting and the conforming-kernel estimates used in applications.

\paragraph*{A hybrid exterior-calculus framework.}

For every form degree $k\in\{0,\ldots,n-1\}$, we introduce an HHO space $\uXh{k}$ consisting of polynomial $k$-forms on mesh cells and polynomial trace unknowns on mesh faces.
The exterior derivative is reconstructed cellwise through a discrete Stokes formula, with the cell unknown entering the volume term and the face unknowns entering the boundary term.
In three space dimensions, the vector proxies of these operators recover the usual hybrid gradient, curl, and divergence reconstructions associated with the de Rham sequence, while the same construction yields a unified family of hybrid exterior derivatives in arbitrary dimension.
This correspondence is made explicit in Subsection~\ref{sec:vector.proxies}.

The exterior-calculus viewpoint is not merely a reformulation of the usual vector-calculus constructions.
Its main advantage is that it isolates a common local-to-global stability mechanism across all form degrees and, at the local level, identifies the components that need to be represented on the mesh skeleton as the traces of the polynomial kernel of the cellwise exterior derivative, leading naturally to the requirement that the selected face spaces contain these traces.

Once these skeletal spaces have been fixed, richer reconstruction targets may be introduced at the cell level without enlarging the globally coupled skeletal space.
The stabilisation measures the remaining mismatch between the traces of the cell polynomial and the face unknowns and, together with the broken exterior derivative, defines the canonical broken--stabilised seminorm $|\cdot|_{\DIFF,h}$ across the de Rham sequence.
At the global level, the compatibility information missing from the hybrid unknowns is transferred from cellular cochains to polynomial skeletons and then completed in a conforming finite element space on the auxiliary simplicial submesh.
This construction leads to the two uniform stability results described next.

\paragraph*{A uniform conforming lifting and a conforming-kernel Poincar\'e inequality.}

The first analytical question is whether the canonical seminorm controls a conforming form whose projected exterior derivative equals the prescribed reconstructed exterior derivative.
Our first main result gives an affirmative answer uniformly over the mesh family and over every stability-admissible choice of the face spaces.

\begin{restatable}[Uniform stable conforming lifting]{theorem}{stableconforminglifting}
\label{thm:main.hybrid.poincare}
Let $k\in\{0,\ldots,n-1\}$ and $\ell\geq0$, and assume that the selected face spaces satisfy the stability-admissibility condition \eqref{eq:def.Qtrace.admissible}.
There exists a fixed polynomial degree $m_{\mathcal S}\geq1$ such that every $\underline v_h\in\uXh{k}$ admits a form
\[
  \lambda_h=\mathcal L_h^k\underline v_h
  \in \PLtrim{m_{\mathcal S}}{k}(\Sh),
\]
depending linearly on $\underline v_h$, such that
\begin{equation}\label{eq:main.lifting.projected.derivative}
  \piD{T}{k}\bigl((\DIFF\lambda_h)_{|T}\bigr)
  =\DIFF_{\ell,T}^{k}\underline v_T
  \qquad\forall T\in\mathcal T_h,
\end{equation}
and
\begin{equation}\label{eq:main.lifting.stability}
  \|\lambda_h\|_\Omega+\|\DIFF\lambda_h\|_\Omega
  \leq C_{\mathrm{lift}}^{k,\ell}|\underline v_h|_{\DIFF,h}.
\end{equation}
The degree $m_{\mathcal S}$ depends only on $n$, $k$, $\ell$, and the standing mesh-regularity parameters.
The constant $C_{\mathrm{lift}}^{k,\ell}$ may additionally depend on the fixed domain $\Omega$.
Both are independent of the mesh size and of the admissible choice of the face spaces.
\end{restatable}

The lifting preserves the prescribed reconstructed exterior derivative through \eqref{eq:main.lifting.projected.derivative}.
The construction applies on domains with arbitrary topology without requiring contractibility or an explicit decomposition into exact and harmonic components, and it controls both the conforming form and its exterior derivative in $L^2$.
Interpolating this lifting yields a stable hybrid representative with the same reconstructed exterior derivative and the corresponding hybrid Poincar\'e inequality, as stated in Corollary~\ref{cor:hybrid.poincare}.
For zero-forms, a constant correction on each connected component yields a conforming right inverse of the hybrid interpolator, relating the construction to scalar HHO smoothers for rough loads; see Remark~\ref{rem:scalar.moments.ern.zanotti}.

Our second main result concerns the broken cell polynomial form itself and the continuous conforming kernel of the exterior derivative,
\begin{equation}\label{eq:def.conforming.closed.space}
  \mathcal Z_{\mathrm c}^k(\Omega)
  \coloneqq
  H\Lambda^k(\Omega)
  \cap
  \Ker\DIFF.
\end{equation}
\begin{restatable}[Uniform Poincar\'e inequality modulo conforming closed forms]{theorem}{poincareconformingclosed}
\label{thm:poincare.conforming.closed}

Let $k\in\{0,\ldots,n-1\}$ and $\ell\geq0$, and assume that the selected face spaces satisfy the stability-admissibility condition \eqref{eq:def.Qtrace.admissible}.
There exists an integer $m_{\mathrm c}$, depending only on $n$, $k$, $\ell$, and the mesh-regularity parameters, such that every $\underline v_h\in\uXh{k}$ admits a form
\[
  z_h^{\mathrm c}
  \in\mathcal Z_{\mathrm c}^k(\Omega)\cap\PLtrim{m_{\mathrm c}}{k}(\Sh),
\]
depending linearly on $\underline v_h$, such that
\begin{equation}
\label{eq:poincare.conforming.closed.representative}
  \left(
    \sum_{T\in\mathcal T_h}
    \left\|
      v_T-\left.z_h^{\mathrm c}\right|_T
    \right\|_T^2
  \right)^{1/2}
  \leq
  C_{\mathrm{conf}}^{k,\ell}
  |\underline v_h|_{\DIFF,h}.
\end{equation}
Consequently,
\begin{equation}
\label{eq:poincare.modulo.conforming.closed.space}
  \inf_{z\in\mathcal Z_{\mathrm c}^k(\Omega)}
  \left(
    \sum_{T\in\mathcal T_h}
    \|v_T-z_{|T}\|_T^2
  \right)^{1/2}
  \leq
  C_{\mathrm{conf}}^{k,\ell}
  |\underline v_h|_{\DIFF,h}.
\end{equation}
The constant $C_{\mathrm{conf}}^{k,\ell}$ has the same dependencies as $C_{\mathrm{lift}}^{k,\ell}$ and is likewise independent of the mesh size and of the admissible choice of the face spaces.
\end{restatable}

This estimate controls the broken cell field modulo the continuous conforming kernel that enters the standard gauges of differential problems.
The two results are complementary, and the second does not require any identification between the cell components of the discrete reconstructed kernel and the continuous conforming kernel $\mathcal Z_{\mathrm c}^k(\Omega)$.

The conforming completion underlying this second result is related to recent techniques for constructing conforming representatives of non-conforming piecewise polynomial fields.
In particular, the construction in \cite{Dong.Ern:26} provides an $\boldsymbol H(\operatorname{curl})$-conforming reconstruction on simplicial meshes whose distance from the original broken field is controlled by suitable jump seminorms, with an application to discontinuous Galerkin approximations of Maxwell problems.
The construction developed here is instead driven by the hybrid skeletal information and is formulated uniformly for every differential-form degree.

In three dimensions, the conforming-kernel estimate gives concrete consequences at all three vector-proxy levels.
For $k=0$, it yields the hybrid Poincar\'e--Wirtinger inequality and also covers, after rescaling, the mixed-order $L^2$-scaled setting underlying the HHO Poincar\'e inequality of \cite[Lemma~12]{Di-Pietro.Droniou:25b}, with degree $r-1$ cell unknowns and degree $r$ face unknowns.
For $k=1$, it controls the broken field modulo $\boldsymbol H(\operatorname{curl};\Omega)\cap\Ker(\operatorname{curl})$ and, under the standard gauge, recovers the useful variant of the second hybrid Weber inequality in \cite[Remark~16]{Lemaire.Pitassi:25}.
The same argument in the weighted $L^2$ product gives the $\mu$-weighted form used in the HHO analysis of magnetostatic and div--curl systems \cite{Dalphin.Ducreux.Lemaire.Pitassi:25}.
For $k=2$, the result gives the analogous Poincar\'e estimate modulo conforming divergence-free fields.

The only non-local analytical input in both constructions is the cellular cochain Poincar\'e inequality, whose uniform constant is denoted by $C_{\mathrm{coch}}^k$ in Lemma~\ref{lem:cochain.poincare}.
Through the Whitney--de Rham realisation of \cite[Lemma~6 and Section~3.4]{Di-Pietro.Droniou.ea:25}, this constant is controlled, up to factors depending only on mesh regularity and fixed polynomial degrees, by the Poincar\'e constant of the corresponding conforming Whitney finite element complex.

For the conforming lifting, the additional local factors arise from stable local cochain inverses, geometric extensions, polynomial moment realisations, and local first-order potential problems.
For the conforming-kernel estimate, they arise from the stable local kernel decomposition, coface averaging, and closed local trace extensions.
All these local factors are uniformly controlled and are not tracked individually.
Thus, even for $n=3$ and $k=1$, the present argument not only recovers the gauge-fixed second hybrid Weber estimate, but also identifies $C_{\mathrm{coch}}^1$, itself controlled by the Poincar\'e constant of the corresponding conforming Whitney finite element complex, as the only non-local analytical contribution to $C_{\mathrm{conf}}^{1,\ell}$.

\paragraph*{Structural difficulty and proof strategy.}

The proof of Theorem~\ref{thm:main.hybrid.poincare} is constructive and is based on a stable compatible polynomial face skeleton.
For a given hybrid unknown, this skeleton consists of single-valued polynomial $k$-forms on the mesh faces whose moments reproduce the prescribed reconstructed exterior derivative against closed test forms, while the weighted norms of the face forms and their exterior derivatives are controlled by $|\underline v_h|_{\DIFF,h}$.
Once the skeleton is available, its face components are completed in the mesh cells through local moment completions and local first-order problems with prescribed trace.
This final step produces the conforming lifting appearing in Theorem~\ref{thm:main.hybrid.poincare}.
The argument therefore separates the global topological compatibility problem from the conforming polynomial realisation, which becomes entirely cellwise once the stable face skeleton has been constructed.

Constructing this skeleton is the main mathematical difficulty.
A fully cellular DDR $k$-form carries polynomial components on mesh cells of every dimension $d\geq k$, so that the intermediate-dimensional hierarchy required by the de Rham structure is built directly into the discrete space \cite{Di-Pietro.Droniou.Pitassi:25}.
The genuinely hybrid space $\uXh{k}$ is fundamentally different because it contains only top-dimensional cell polynomials and face traces.
The intermediate cellular hierarchy needed to propagate global compatibility is therefore absent from the discrete space and must instead be generated within the analysis.

For $k\leq n-2$, this missing hierarchy is organised in a functional \emph{double complex} on auxiliary full trimmed test spaces.
Its horizontal differential is the functional trace coboundary, while its vertical differential is the transpose of the local exterior derivative.
A trace-compatible geometric decomposition transforms each horizontal row into a direct sum of local cellular cochain complexes built from the cofaces of individual mesh cells, which we refer to as coface-link complexes.
The topology of these links provides the exactness required by the local horizontal problems, while their uniformly bounded combinatorial complexity yields stable local inverses.
For general reconstructed differential data, our construction uses the full double complex to propagate both horizontal compatibility and vertical differential information through the missing dimensions.

The construction is related to the discrete distributional differential form complexes of \cite{Licht:17,Christiansen.Licht:20} and to the local-to-global homological arguments underlying bounded cochain projections in finite element exterior calculus \cite{Falk.Winther:14,Falk.Winther:15}.
In the former setting, broken finite element spaces on matching simplicial meshes are organised into distributional complexes with polynomial components on cells of all dimensions.
The present argument adapts this type of local-to-global mechanism to a genuinely hybrid polyhedral setting, in which the intermediate levels exist only analytically as functionals over auxiliary full trimmed test spaces.
In these local coordinates, the horizontal equations reduce to uniformly stable cellular problems, and the construction ultimately yields stable polynomial hierarchies on the original polyhedral mesh.

\paragraph*{A cellular-to-hybrid transfer principle.}

The conforming lifting and the conforming-kernel estimate can thus be viewed as complementary realisations of a common cellular-to-hybrid transfer principle.
Global compatibility is first controlled at the cellular level and then realised by stable polynomial data, with the final completion performed in a conforming finite element space on the auxiliary simplicial submesh.
For general reconstructed differential data, our construction uses the full functional double complex, whereas for cellwise closed data only trace compatibility has to be transferred before the conforming completion.

This common structure also suggests a route to Sobolev-type estimates, starting from a suitable estimate for the terminal cellular cochain problem and verifying the local polynomial realisations in the corresponding weighted norms.

The remainder of the paper is organised as follows.
Section~\ref{sec:setting} introduces the geometric, differential-form, and polynomial notation.
Section~\ref{sec:hho.forms} defines the HHO spaces, the reconstructed exterior derivatives, the stability-admissible face spaces, and the canonical broken--stabilised seminorm.
Section~\ref{sec:functional.framework} introduces the functional trace framework and the stable local coface-link solves.
Section~\ref{sec:poincare.proof} constructs the stable skeleton and proves the conforming lifting theorem.
Section~\ref{sec:poincare.conforming.closed} establishes the complementary estimate modulo conforming closed forms.
Subsection~\ref{sec:poincare.vector.proxies} spells out its gradient, curl, and divergence consequences in three space dimensions.
Concluding remarks and perspectives are collected in Section~\ref{sec:conclusions}.

\section{Setting}
\label{sec:setting}
We collect here the geometric, differential-form, and polynomial notation used throughout the paper.
The presentation follows the polytopal exterior-calculus setting proposed in \cite{Bonaldi.Di-Pietro.ea:24}, with the closed-cell convention used in \cite{Di-Pietro.Droniou.ea:25,Di-Pietro.Droniou.Pitassi:25}.

\subsection{Polytopal cells, meshes, and simplicial submeshes}
\label{sec:setting:mesh}

Let $\Omega\subset\Real^n$ be a bounded Lipschitz polytopal domain, not necessarily contractible.
A $d$-cell is a closed polytopal subset of a $d$-dimensional affine subspace of $\Real^n$, homeomorphic to a closed $d$-dimensional ball.
A \emph{polytopal mesh} $\Mh$ of $\Omega$ is a finite collection of cells of dimensions $d\in\{0,\ldots,n\}$ such that the union of the $n$-cells is $\overline\Omega$, the boundary of every $d$-cell is a union of $(d-1)$-cells, and the intersection of two distinct cells of the same dimension is either empty or a union of lower-dimensional cells.
Moreover, if $e$ and $f$ are distinct cells with $\dim e\leq\dim f$ and $e\cap f\neq\emptyset$, then $e\cap f\subset\partial f$.
The cells of $\Mh$, together with their incidence relations, are understood to form a finite regular cellular decomposition of $\overline\Omega$.

The set of $d$-cells is denoted by $\FM{d}(\Mh)$.
If $f$ is a cell, we write $\dim f$ for its dimension.
If $f$ has dimension at least $d$, we set
\[
  \FM{d}(f)
  \coloneqq
  \{e\in\FM{d}(\Mh):e\subset f\},
  \qquad
  \FM{d}(\partial f)
  \coloneqq
  \{e\in\FM{d}(\Mh):e\subset\partial f\}.
\]
For a cell $s\in\Mh$ and $d\geq\dim s$, we denote by
\[
  \cof{d}(s)
  \coloneqq
  \{f\in\FM{d}(\Mh)\mid s\subset f\}
\]
the set of its $d$-dimensional cofaces.
Throughout the combinatorial constructions, cells are understood as closed.
Function spaces and differential forms on a positive-dimensional cell are understood on its relative interior, while traces are taken on its topological boundary.

We reserve the notation $\mathcal T_h\coloneqq\FM{n}(\Mh)$, $\mathcal F_h\coloneqq\FM{n-1}(\Mh)$, for the sets of $n$-dimensional cells and faces, respectively.
Elements of $\mathcal T_h$ are denoted by $T$, whereas elements of $\mathcal F_h$ are denoted by $F$.
For $T\in\mathcal T_h$, we set $\mathcal F_T\coloneqq\FM{n-1}(\partial T)$.

Every cell $f\in\Mh$ is equipped with an orientation, with all vertices positively oriented.
If $\dim f\geq1$, then $h_f$ denotes its diameter.
For a vertex $v\in\FM{0}(\Mh)$, whose geometric diameter vanishes, $h_v$ denotes a positive local mesh size comparable with the diameters of the cells incident to $v$; for instance,
\[
  h_v
  \coloneqq
  \max_{\substack{T\in\mathcal T_h\\v\subset T}}h_T.
\]
For every positive-dimensional cell $f$, we also fix a point $\boldsymbol{x}_f$ in its relative interior; for a vertex $v$, we set $\boldsymbol{x}_v=v$.
These points are used only to define translated Koszul operators.

If $f'\in\FM{d-1}(\partial f)$ is a boundary cell of an oriented $d$-cell $f$, we denote by $\epsilon_{f,f'}\in\{-1,+1\}$ the relative orientation of $f'$ with respect to $f$.
Boundary integrals over $\partial f$ are always understood as oriented sums over the boundary cells with the signs $\epsilon_{f,f'}$.

We assume that $\Mh$ admits a conforming auxiliary simplicial submesh $\Sh$.
Thus every $d$-cell $f\in\FM{d}(\Mh)$ is the union of the $d$-simplices in a simplicial mesh $\Sh(f)$.
We also set
\[
  \Sh(\partial f)
  \coloneqq
  \bigcup_{f'\in\FM{d-1}(\partial f)}\Sh(f'),
\]
so that $\Sh(\partial f)$ is the induced simplicial mesh of $\partial f$.

We work with regular mesh families in the sense of \cite[Definition~1.9]{Di-Pietro.Droniou:20}.
We use the following consequences throughout the paper without further mention.
The simplicial meshes $\Sh(f)$ are uniformly shape-regular and have uniformly bounded cardinality, so $\card(\Sh(f))\lesssim 1$ for every $f\in\Mh$.
The diameter of every top-dimensional simplex of $\Sh(f)$ is uniformly comparable with $h_f$.
Consequently, the number of subcells of any mesh cell and the number of cells incident on any fixed subcell are uniformly bounded.
The local mesh sizes of incident cells are uniformly comparable; in particular, if $f'\subset\partial f$, then $h_f\simeq h_{f'}$.

Throughout the paper, $a\lesssim b$ means that $a\leq Cb$ with a constant $C$ independent of the mesh size $h$, of the particular mesh cell or simplex, and of the discrete unknowns and data.
The constant may depend on $\Omega$, on $n$, on the fixed polynomial degrees and local polynomial spaces, and on the regular mesh family parameter.
We write $a\simeq b$ when both $a\lesssim b$ and $b\lesssim a$ hold.

\subsection{Differential forms and notation conventions}
\label{sec:setting:forms}

Let $f$ be a $d$-cell and let $k\in\{0,\ldots,d\}$.
We denote by $\Lambda^k(f)$ the space of differential $k$-forms on the relative interior of $f$.
Regularity is indicated by placing the functional space before $\Lambda^k(f)$; in particular, $L^2\Lambda^k(f)$ is the space of $k$-forms with square-integrable coefficients.

The exterior derivative on $f$ is denoted by $\DIFF_f$, or simply by $\DIFF$ when no ambiguity is possible.
We set
\[
  H\Lambda^k(f)
  \coloneqq
  \left\{
    \omega\in L^2\Lambda^k(f):
    \DIFF_f\omega\in L^2\Lambda^{k+1}(f)
  \right\},
\]
with the convention $\Lambda^{d+1}(f)=\{0\}$ when $k=d$.

If $e\subset\partial f$, the trace of a sufficiently regular form on $e$ is denoted by $\trc_e\omega$, or simply by $\trc\omega$ when the cell is clear.
The Hodge star on a $d$-cell $f$ is denoted by $\star_f:\Lambda^k(f)\longrightarrow\Lambda^{d-k}(f)$, and we write simply $\star$ when the cell is clear from the context.
The $L^2$ inner product and norm are
\[
  \langle\omega,\mu\rangle_f
  \coloneqq
  \int_f\omega\wedge\star_f\mu,
  \qquad
  \|\omega\|_f
  \coloneqq
  \langle\omega,\omega\rangle_f^{1/2}.
\]
When the form degree is useful for clarity, the same norm is written $\|\omega\|_{L^2\Lambda^k(f)}$.
For a collection $\theta_{\partial f}=(\theta_e)_{e\in\FM{d-1}(\partial f)}$ on the boundary of a $d$-cell, we use
\[
  \|\theta_{\partial f}\|_{\partial f}^2
  \coloneqq
  \sum_{e\in\FM{d-1}(\partial f)}\|\theta_e\|_e^2.
\]

\paragraph*{Notation conventions.}
Whenever an operator or norm indexed by a mesh entity is applied to a global or piecewise form, restriction to that entity is understood whenever it is unambiguous.
Thus, for instance, expressions such as $\DIFF_Tw$, $\trc_Fw$, or a local projector applied to $w$ do not explicitly display the preceding restriction to the relevant entity.
For composable operators, juxtaposition denotes composition, with the rightmost operator applied first.
Finally, whenever a finite-dimensional space $\mathcal X\subset L^2\Lambda^k(f)$ is introduced, the symbol $\pi_{\mathcal X}$, with the mesh-entity, form-degree, and polynomial-degree indices carried by the notation of $\mathcal X$, denotes the $L^2$-orthogonal projector onto $\mathcal X$.
We therefore only specify the corresponding projector notation when a new finite-dimensional space is introduced.

\subsection{Polynomial and finite element spaces of differential forms}
\label{sec:setting:polynomial-forms}

Let $f\in\FM{d}(\Mh)$ be a $d$-cell.
For any integer $\ell$ and any form degree $k$, we denote by $\PL{\ell}{k}(f)$ the space of polynomial $k$-forms on $f$ whose coefficients have total degree at most $\ell$.
We use the convention $\PL{\ell}{k}(f)\coloneqq\{0\}$ if $\ell<0$, $k<0$, or $k>d$.
The $L^2$ projector onto $\PL{\ell}{k}(f)$ is denoted by $\pi_f^{k,\ell}\colon L^2\Lambda^k(f)\longrightarrow\PL{\ell}{k}(f)$.
In particular, $\pi_T^{k,\ell}$ denotes the full polynomial projector on a top-dimensional cell.

\paragraph{Koszul complements.}

We denote by $\KOSZUL_f$ the Koszul contraction associated with the shifted identity field $x\mapsto x-\boldsymbol{x}_f$.
For every polynomial degree $\ell$ and form degree $k$, define
\[
  \mathcal K_\ell^k(f)
  \coloneqq
  \KOSZUL_f\PL{\ell-1}{k+1}(f).
\]
Thus $\mathcal K_\ell^k(f)$ consists of polynomial $k$-forms of degree at most $\ell$.
In particular, $\mathcal K_0^k(f)=\{0\}$, $\mathcal K_\ell^d(f)=\{0\}$, $\mathcal K_\ell^{-1}(f)=\{0\}$.

The standard polynomial Koszul decomposition reads
\begin{subequations}\label{eq:full-polynomial-koszul-decomposition}
\begin{align}
  \PL{\ell}{0}(f)
  &=
  \PL{0}{0}(f)
  \oplus
  \mathcal K_\ell^0(f),
  \label{eq:full-polynomial-koszul-decomposition-zero}
  \\
  \PL{\ell}{k}(f)
  &=
  \DIFF_f\PL{\ell+1}{k-1}(f)
  \oplus
  \mathcal K_\ell^k(f),
  \qquad
  k\in\{1,\ldots,d\}.
  \label{eq:full-polynomial-koszul-decomposition-positive}
\end{align}
\end{subequations}
Moreover, $\DIFF_f$ is injective on $\mathcal K_{\ell+1}^{k-1}(f)$ and
\[
  \mathcal K_{\ell+1}^{k-1}(f)
  \times
  \mathcal K_\ell^k(f)
  \ni
  (\mu,\nu)
  \longmapsto
  \DIFF_f\mu+\nu
  \in
  \PL{\ell}{k}(f)
\]
is an isomorphism for every $k\geq1$.

\paragraph{Trimmed polynomial spaces.}

The trimmed polynomial spaces are defined by
\begin{subequations}\label{eq:def-trimmed-polynomial-forms}
\begin{align*}
  \PLtrim{\ell}{0}(f)
  &\coloneqq
  \PL{\ell}{0}(f),
  \\
  \PLtrim{\ell}{k}(f)
  &\coloneqq
  \DIFF_f\PL{\ell}{k-1}(f)
  \oplus
  \mathcal K_\ell^k(f),
  \qquad
  k\in\{1,\ldots,d\}.
\end{align*}
\end{subequations}
They satisfy $\PL{\ell-1}{k}(f) \subset \PLtrim{\ell}{k}(f) \subset \PL{\ell}{k}(f)$, and, in top form degree, $\PLtrim{\ell}{d}(f) = \PL{\ell-1}{d}(f)$.

\paragraph{Conforming spaces on the auxiliary simplicial submesh.}

We denote by $\PL{\ell}{k}(\Sh(f))$ the conforming finite element space of $k$-forms on $f$ whose restriction to every $S\in\Sh(f)$ belongs to $\PL{\ell}{k}(S)$.
Similarly,
\[
  \PLtrim{\ell}{k}(\Sh(f))
  \coloneqq
  \left\{
    \omega\in H\Lambda^k(f):
    \omega_{|S}\in\PLtrim{\ell}{k}(S)
    \quad
    \forall S\in\Sh(f)
  \right\},
\]
and
\[
  \PLtrimz{\ell}{k}(\Sh(f))
  \coloneqq
  \left\{
    \omega\in\PLtrim{\ell}{k}(\Sh(f)):
    \trc_{\partial f}\omega=0
  \right\}.
\]
On a simplex $S$, the notation $\PLtrimz{\ell}{k}(S)$ has the same meaning, with zero trace on $\partial S$.

On the whole auxiliary simplicial submesh, we set
\[
  \PLtrim{\ell}{k}(\Sh)
  \coloneqq
  \left\{
    \omega\in H\Lambda^k(\Omega):
    \omega_{|S}\in\PLtrim{\ell}{k}(S)
    \quad
    \forall S\in\FM{n}(\Sh)
  \right\}.
\]

The trace preserves trimmed polynomial spaces.
For every subcell $e\subset f$ and every admissible form degree $k$, $\trc_e\PLtrim{\ell}{k}(f) \subset \PLtrim{\ell}{k}(e)$, and the same property holds for the corresponding conforming spaces on the auxiliary simplicial submeshes.

\section{Hybrid High-Order spaces of differential forms}
\label{sec:hho.forms}

We introduce a family of hybrid spaces of differential forms, their interpolation operators, a class of local exterior-derivative reconstructions, and the associated stabilisation.
The presentation separates three ingredients, namely (i) the stability-admissible face spaces and the corresponding hybrid interpolation; (ii) the cell-local reconstruction spaces and commuting exterior-derivative reconstructions; and (iii) the canonical and reconstructed energies.
This separation makes it possible to keep the globally coupled face spaces as small as required by the projected-jump stability mechanism analysed here, while allowing richer cell-local reconstructions whenever needed for consistency or analysis; see, for example, their use in Section~\ref{sec:poincare.proof}.

\subsection{Stability-admissible hybrid spaces and interpolation}
\label{subsec:hybrid.spaces}

Let $k\in\{0,\ldots,n-1\}$ and let $\ell\geq0$ be the polynomial degree of the cell unknowns.
We first identify the smallest face spaces required by the projected-jump stability mechanism used below, and then allow the actual face space to be any polynomial enrichment of these spaces.

\begin{definition}[Stability-minimal and admissible face trace spaces]
\label{def:face.trace.spaces}
For every face $F\in\mathcal F_h$, define the \emph{stability-minimal} face trace space by
\[
  \mathcal S_F^{k,\ell}
  \coloneqq
  \begin{cases}
    \PL{0}{0}(F),
    & k=0,
    \\[1mm]
    \DIFF_F\PL{\ell+1}{k-1}(F),
    & k\in\{1,\ldots,n-2\},
    \\[1mm]
    \PL{\ell}{n-1}(F),
    & k=n-1.
  \end{cases}
\]
An \emph{admissible} face trace space is any polynomial space $\Qtrace{F}{k}$ such that
\begin{equation}\label{eq:def.Qtrace.admissible}
  \mathcal S_F^{k,\ell}
  \subseteq
  \Qtrace{F}{k}
  \subseteq
  \PL{\ell}{k}(F).
\end{equation}
The lower inclusion is the stability requirement, whereas the possible enrichment from $\mathcal S_F^{k,\ell}$ to $\Qtrace{F}{k}$ can be used to support a chosen exterior-derivative reconstruction or additional approximation properties.
The estimates proved below under the sole face-space assumption \eqref{eq:def.Qtrace.admissible} are uniform with respect to the particular admissible choice of $\Qtrace{F}{k}$.
The qualifier \emph{stability-minimal} refers to the projected-jump stability mechanism considered here, for which the face space must contain the traces of the polynomial kernel.
\end{definition}

For every $T\in\mathcal T_h$ and every $F\in\mathcal F_T$, polynomial exactness, surjectivity of polynomial restriction to $F$, and commutation of the trace with the exterior derivative give
\begin{equation}\label{eq:stability.minimal.kernel.trace}
  \mathcal S_F^{k,\ell}
  =
  \trc_F
  \left(
    \PL{\ell}{k}(T)
    \cap
    \Ker\DIFF_T
  \right).
\end{equation}
Indeed, the polynomial kernel consists of the constants when $k=0$ and of $\DIFF_T\PL{\ell+1}{k-1}(T)$ when $k\geq1$; in top degree on the face, polynomial exactness yields $\DIFF_F\PL{\ell+1}{n-2}(F)=\PL{\ell}{n-1}(F)$.
Thus, $\mathcal S_F^{k,\ell}$ is precisely the trace on $F$ of the polynomial kernel, and these kernel traces are the components that must be controlled directly on the skeleton.

\begin{remark}[Independent bulk and skeletal polynomial degrees]
\label{rem:bulk.skeletal.degrees}

The use of the same polynomial degree $\ell$ for cell and face unknowns is a convenient choice for the presentation, rather than an intrinsic feature of the projected-jump stability mechanism.

Indeed, let $V_T^k$ be a polynomial cell space and let $\mathcal Q_F^k$ be a polynomial face space, with polynomial degrees that need not coincide.
Provided that the cell spaces $V_T^k$ admit a uniformly stable decomposition modulo $V_T^k\cap\Ker\DIFF_T$, the local argument underlying Proposition~\ref{prop:full.jump.control.forms} requires, as far as the face spaces are concerned, only
\begin{equation}\label{eq:general.kernel.trace.inclusion}
  \trc_F
  \left(
    V_T^k\cap\Ker\DIFF_T
  \right)
  \subseteq
  \mathcal Q_F^k
  \qquad
  \forall T\in\cof{n}(F).
\end{equation}
This is precisely the property that allows the kernel component of the cell polynomial to be eliminated from the part of its trace that is not seen by the face projector.

For the full polynomial choice $V_T^k=\PL{r_T}{k}(T)$, condition \eqref{eq:general.kernel.trace.inclusion} becomes
\[
  \trc_F
  \left(
    \PL{r_T}{k}(T)\cap\Ker\DIFF_T
  \right)
  \subseteq
  \mathcal Q_F^k
  \qquad
  \forall T\in\cof{n}(F).
\]
Hence the same local stability mechanism accommodates cell-dependent polynomial degrees, as well as independently chosen face spaces, provided that each face space contains the kernel traces induced by all adjacent cells.

For $k=0$, these traces consist only of constants, independently of the cell degree, so stability alone permits a skeletal degree strictly lower than the bulk degree.
For positive form degrees, the kernel traces generally retain more of the bulk polynomial structure, so a lower-order face space is not automatically admissible and condition \eqref{eq:general.kernel.trace.inclusion} has to be checked explicitly.

The remainder of the paper is written in the uniform-degree setting for simplicity.
Allowing variable degrees would additionally require adapting the interpolation and reconstruction spaces so that the corresponding compatibility and uniformity assumptions remain valid.
\end{remark}

For $k=0$, the usual HHO choice $\Qtrace{F}{0}=\PL{\ell}{0}(F)$ is an admissible enrichment of the stability-minimal constant space.
The corresponding three-dimensional vector proxies are discussed in Subsection~\ref{sec:vector.proxies}.

In accordance with the projection convention of Subsection~\ref{sec:setting:forms}, we denote the corresponding $L^2$ projector by $\piQ{F}{k}$.

We use underlined letters for hybrid unknowns, $\underline v_h$ globally and $\underline v_T$ locally, and reserve $v_h$ for the broken cell polynomial form associated with $\underline v_h$.
The global hybrid space of $k$-forms is
\[
  \uXh{k}
  \coloneqq
  \Bigl\{
    \underline v_h
    =
    \bigl(
      (v_T)_{T\in\mathcal T_h},
      (v_F)_{F\in\mathcal F_h}
    \bigr)
    \;\Bigm|\;
    \begin{array}{l}
      v_T\in\PL{\ell}{k}(T)
      \quad\forall T\in\mathcal T_h,
      \\
      v_F\in\Qtrace{F}{k}
      \quad\forall F\in\mathcal F_h
    \end{array}
  \Bigr\}.
\]
For a cell $T\in\mathcal T_h$, the corresponding local space is $\uX{k}{T} \coloneqq \PL{\ell}{k}(T) \times \bigtimes_{F\in\mathcal F_T} \Qtrace{F}{k}$.
If $\underline v_h\in\uXh{k}$, its restriction to $T$ is denoted by $\underline v_T\coloneqq\bigl(v_T,(v_F)_{F\in\mathcal F_T}\bigr)\in\uX{k}{T}$.
Finally, we associate with $\underline v_h$ the broken polynomial form $v_h \in \bigtimes_{T\in\mathcal T_h} \PL{\ell}{k}(T)$ such that $(v_h)_{|T} \coloneqq v_T$.

\subsubsection{Interpolation operators}
\label{subsec:hybrid.reduction.forms}

For sufficiently regular $v\in H\Lambda^k(\Omega)$, we define the global interpolator $\uI{k}{h}v\in\uXh{k}$ by
\[
  \uI{k}{h}v
  \coloneqq
  \left(
    (\pi_T^{k,\ell}(v_{|T}))_{T\in\mathcal T_h},
    (\piQ{F}{k}\trc_Fv)_{F\in\mathcal F_h}
  \right),
\]
and the local interpolator by
\[
  \uI{k}{T}v
  \coloneqq
  \left(
    \pi_T^{k,\ell}v,
    (\piQ{F}{k}\trc_Fv)_{F\in\mathcal F_T}
  \right).
\]

\begin{remark}[Weak interpretation of the face projections]
\label{rem:regularity.interpolator.forms}
The face component of the interpolator is directly defined whenever the traces of $v$ on the mesh faces are $L^2$ differential forms.
At lower regularity, the same definition is understood weakly.
More precisely, whenever the trace of $v$ defines a duality pairing with the finite-dimensional space $\Qtrace{F}{k}$, denoted by $\langle\cdot,\cdot\rangle_F$, the element $\piQ{F}{k}\trc_Fv\in\Qtrace{F}{k}$ is defined by
\[
  \int_F
  \piQ{F}{k}\trc_Fv
  \wedge
  \star_F\mu_F
  =
  \langle\trc_Fv,\mu_F\rangle_F
  \qquad
  \forall\mu_F\in\Qtrace{F}{k}.
\]
Here the bracket on the right-hand side denotes the duality extension of the $L^2$ pairing introduced in Subsection~\ref{sec:setting:forms}.
This weak interpretation is sufficient for the commuting results below.
\end{remark}

\subsubsection{Stable decomposition modulo the polynomial kernel}

The identity \eqref{eq:stability.minimal.kernel.trace} provides a qualitative characterisation of the stability-minimal face spaces, while the following proposition gives the corresponding quantitative statement.
It decomposes each cell polynomial into a closed component, whose traces lie in the stability-minimal face spaces, and a remainder controlled both in the cell and on its boundary by the cellwise exterior derivative.

This decomposition is the local mechanism used below to recover the full trace jump from the projected stabilisation.
It will also be used in Proposition~\ref{prop:stable.cellwise.closed.reduction} to extract the cellwise closed family entering the Poincar\'e inequality modulo conforming closed forms.

\begin{proposition}[Stable polynomial decomposition modulo the kernel]
\label{prop:stable.polynomial.kernel.decomposition}
For every $T\in\mathcal T_h$ and every $v_T\in\PL{\ell}{k}(T)$, there exists a form $z_T\in\PL{\ell}{k}(T)\cap\Ker\DIFF_T$ such that
\begin{equation}\label{eq:stable.polynomial.kernel.decomposition}
  \|v_T-z_T\|_{L^2\Lambda^k(T)}
  +
  h_T^{1/2}
  \|\trc_{\partial T}(v_T-z_T)\|_{L^2\Lambda^k(\partial T)}
  \lesssim
  h_T
  \|\DIFF_Tv_T\|_{L^2\Lambda^{k+1}(T)}.
\end{equation}
Moreover, $\trc_Fz_T\in\mathcal S_F^{k,\ell}$ for every $F\in\mathcal F_T$.
The selection can be made linear in $v_T$, and the hidden constant is uniform on a regular mesh sequence.
\end{proposition}
\begin{proof}
The polynomial Koszul decompositions \eqref{eq:full-polynomial-koszul-decomposition-zero}-- \eqref{eq:full-polynomial-koszul-decomposition-positive} provide a unique decomposition
\[
v_T=
z_T+w_T,\qquad
z_T\in\PL{\ell}{k}(T)\cap\Ker\DIFF_T,\qquad
w_T\in\mathcal K_\ell^k(T).
\]
The restriction of $\DIFF_T$ to $\mathcal K_\ell^k(T)$ is injective and its image is $\DIFF_T\PL{\ell}{k}(T)$.
The scaled stability of the inverse on the Koszul complement follows from the standard polynomial Poincar\'e estimate; see, in particular, \cite[Lemmas~16--17]{Di-Pietro.Droniou.ea:25}.
Thus, $\|w_T\|_T \lesssim h_T \|\DIFF_Tw_T\|_T = h_T \|\DIFF_Tv_T\|_T$.
The polynomial trace inequality then yields $h_T^{1/2} \|\trc_{\partial T}w_T\|_{\partial T} \lesssim \|w_T\|_T$.
Since $w_T=v_T-z_T$, these estimates prove \eqref{eq:stable.polynomial.kernel.decomposition}.
Finally, \eqref{eq:stability.minimal.kernel.trace} gives $\trc_Fz_T\in\mathcal S_F^{k,\ell}$ for every $F\in\mathcal F_T$.
All the selections involved in the Koszul decomposition are linear.
\end{proof}

\subsection{Local exterior-derivative reconstructions and commuting properties}
\label{subsec:hybrid.exterior.derivative}

For every $T\in\mathcal T_h$, let $\mathcal R_T^{k+1}\subset L^2\Lambda^{k+1}(T)$ be a fixed finite-dimensional polynomial space, possibly defined on the auxiliary simplicial submesh $\Sh(T)$.
We assume that $\star_T\mathcal R_T^{k+1} \subset H\Lambda^{n-k-1}(T)$, and that
\[
  \trc_F(\star_Tq_T)
  \in
  L^2\Lambda^{n-k-1}(F)
  \qquad
  \forall q_T\in\mathcal R_T^{k+1},
  \quad
  \forall F\in\mathcal F_T.
\]
We denote the corresponding $L^2$ projector by $\pi_{\mathcal R,T}^{k+1}$.

For $\underline v_T\in\uX{k}{T}$, define $\DIFF_{\mathcal R,T}^{k}\underline v_T\in\mathcal R_T^{k+1}$ by
\begin{equation}\label{eq:def.local.derivative.general.forms}
\begin{aligned}
\int_T
\DIFF_{\mathcal R,T}^{k}\underline v_T
\wedge
\star_Tq_T
={}
(-1)^{k+1}
\int_T
v_T\wedge\DIFF_T(\star_Tq_T)
+
\sum_{F\in\mathcal F_T}
\epsTF
\int_F
v_F\wedge\trc_F(\star_Tq_T)
\end{aligned}
\end{equation}
for every $q_T\in\mathcal R_T^{k+1}$.
The non-degeneracy of the $L^2$ pairing on $\mathcal R_T^{k+1}$ uniquely determines the reconstruction.
The corresponding global operator is defined cellwise by
\[
  \DIFF_{\mathcal R,h}^{k}\colon
  \uXh{k}
  \longrightarrow
  \bigtimes_{T\in\mathcal T_h}
  \mathcal R_T^{k+1},
  \qquad
  (\DIFF_{\mathcal R,h}^{k}\underline v_h)_{|T}
  \coloneqq
  \DIFF_{\mathcal R,T}^{k}\underline v_T.
\]

Two reconstruction spaces play a distinguished role below.

The prescribed reconstruction space is the largest exact-image space compatible with the selected face spaces and is defined by
\[
  \Dpoly{T}{k}
  \coloneqq
  \left\{
    q_T\in\DIFF_T\PL{\ell+1}{k}(T)
    \;:\;
    \star_F^{-1}
    \trc_F(\star_Tq_T)
    \in
    \Qtrace{F}{k}
    \quad
    \forall F\in\mathcal F_T
  \right\}.
\]
The superscript $k$ records the degree of the hybrid input form, whereas $\Dpoly{T}{k}$ itself is a space of $(k+1)$-forms.
We denote the corresponding $L^2$ projector by $\piD{T}{k}$ and write $\DIFF_{\ell,T}^{k} \coloneqq \DIFF_{\mathcal R,T}^{k}$ and $\DIFF_{\ell,h}^{k} \coloneqq \DIFF_{\mathcal R,h}^{k}$ when $\mathcal R_T^{k+1}=\Dpoly{T}{k}$.
Thus
\[
\begin{aligned}
  \int_T
  \DIFF_{\ell,T}^{k}\underline v_T
  \wedge
  \star_Tq_T
  ={}&
  (-1)^{k+1}
  \int_T
  v_T\wedge\DIFF_T(\star_Tq_T)
  \\
  &+
  \sum_{F\in\mathcal F_T}
  \epsTF
  \int_F
  v_F\wedge\trc_F(\star_Tq_T)
\end{aligned}
\]
for every $q_T\in\Dpoly{T}{k}$, and $\DIFF_{\ell,h}^{k}$ is defined cellwise with values in $\bigtimes_{T\in\mathcal T_h}\Dpoly{T}{k}$.

The proof in Section~\ref{sec:poincare.proof} additionally uses the auxiliary full trimmed reconstruction space
\begin{equation}\label{eq:def.auxiliary.trimmed.derivative.space}
  \widehat{\mathcal R}_T^{k+1}
  \coloneqq
  \star_T^{-1}
  \PLtrim{m}{n-k-1}(\Sh(T)),
\end{equation}
with $m$ fixed independently of $h$.
This auxiliary reconstruction is essentially an analytical tool, as it does not modify the discrete unknowns and need not be assembled in an implementation.

The distinction between the two reconstruction spaces is structural.
The prescribed space $\Dpoly{T}{k}$ determines the discrete operator and the projection appearing in Theorem~\ref{thm:main.hybrid.poincare}.
The auxiliary full trimmed space is used only in the analysis to generate the intermediate functional hierarchy and the stable skeleton.

Other cell-local choices, including full polynomial spaces, are covered by the definition above.
They satisfy the projected commuting property below whenever the stated regularity and interpolation-compatibility conditions hold.
The computational implications of enriching the reconstruction and face spaces are discussed in Remark~\ref{rem:role.broken.reconstructed.energies}.

We say that a reconstruction space $\mathcal R_T^{k+1}$ is \emph{interpolation-compatible} if
\begin{align}
  \star_T^{-1}
  \DIFF_T(\star_Tq_T)
  &\in
  \PL{\ell}{k}(T)
  &&\forall q_T\in\mathcal R_T^{k+1},
  \label{eq:reconstruction.cell.compatibility}
  \\
  \star_F^{-1}
  \trc_F(\star_Tq_T)
  &\in
  \Qtrace{F}{k}
  &&\forall q_T\in\mathcal R_T^{k+1},
  \quad\forall F\in\mathcal F_T.
  \label{eq:reconstruction.face.compatibility}
\end{align}

The following proposition isolates the compatibility assumptions under which a local Stokes reconstruction commutes with interpolation up to the $L^2$ projection onto its target space.
It is the abstract consistency result for the reconstruction family; the commuting properties of the specific reconstruction spaces used below follow as direct consequences.

\begin{proposition}[Projected commuting property]
\label{prop:projected.commuting.forms}
Let $\mathcal R_T^{k+1}$ be interpolation-compatible.
Let $v$ be a $k$-form on $T$.
Assume that the local interpolator is well defined and that Stokes' formula holds against every test form in this reconstruction space.
Then
\[
  \DIFF_{\mathcal R,T}^{k}\uI{k}{T}v
  =
  \pi_{\mathcal R,T}^{k+1}\DIFF_Tv.
\]
\end{proposition}

\begin{proof}
Let $q_T\in\mathcal R_T^{k+1}$.
Using \eqref{eq:def.local.derivative.general.forms} with the interpolant gives
\[
\begin{aligned}
  \int_T
  \DIFF_{\mathcal R,T}^{k}(\uI{k}{T}v)
  \wedge\star_Tq_T
  ={}&
  (-1)^{k+1}
  \int_T
  \pi_T^{k,\ell}v
  \wedge
  \DIFF_T(\star_Tq_T)
  \\
  &+
  \sum_{F\in\mathcal F_T}
  \epsTF
  \int_F
  \piQ{F}{k}\trc_Fv
  \wedge
  \trc_F(\star_Tq_T).
\end{aligned}
\]
Condition \eqref{eq:reconstruction.cell.compatibility} removes the cell projector, while \eqref{eq:reconstruction.face.compatibility} removes the face projectors.
Stokes' formula then yields $\int_T \DIFF_{\mathcal R,T}^{k}(\uI{k}{T}v) \wedge\star_Tq_T = \int_T \DIFF_Tv\wedge\star_Tq_T$.
Since $q_T$ is arbitrary in $\mathcal R_T^{k+1}$, the conclusion follows.
\end{proof}

\begin{corollary}[Projected commuting property of the prescribed reconstruction]
\label{cor:prescribed.projected.commuting.forms}

Let $v$ be a $k$-form on $T$ such that the local interpolator $\uI{k}{T}v$ is well defined and Stokes' formula holds against $\star_Tq_T$ for every $q_T\in\Dpoly{T}{k}$.
Then
\[
  \DIFF_{\ell,T}^{k}\uI{k}{T}v
  =
  \piD{T}{k}\DIFF_Tv.
\]
\end{corollary}
\begin{proof}
The face compatibility condition \eqref{eq:reconstruction.face.compatibility} holds by the definition of $\Dpoly{T}{k}$.
Moreover, for every $q_T\in\Dpoly{T}{k}$, polynomial degree reduction under the exterior derivative gives $\star_T^{-1}\DIFF_T(\star_Tq_T)\in\PL{\ell}{k}(T)$, so that the cell compatibility condition \eqref{eq:reconstruction.cell.compatibility} also holds.
The assertion therefore follows from Proposition~\ref{prop:projected.commuting.forms}.
\end{proof}

\subsection{Stabilisation and canonical hybrid norms}
\label{subsec:stabilization.hybrid.norms}

For $\underline v_T,\underline w_T\in\uX{k}{T}$, define the local stabilisation
\begin{equation}\label{eq:def.stabilization.forms}
  \mathrm s_{\mathcal Q,T}^{k,\ell}
  (\underline v_T,\underline w_T)
  \coloneqq
  \sum_{F\in\mathcal F_T}
  h_F^{-1}
  \int_F
  \left(
    \piQ{F}{k}\trc_Fv_T-v_F
  \right)
  \wedge
  \star_F
  \left(
    \piQ{F}{k}\trc_Fw_T-w_F
  \right).
\end{equation}
When no ambiguity is possible, we write simply $\mathrm s_T^{k,\ell}$.

The hybrid exterior-derivative seminorm is
\[
  |\underline v_h|_{\DIFF,h}^2
  \coloneqq
  \sum_{T\in\mathcal T_h}
  \left(
    \|\DIFF_Tv_T\|_{L^2\Lambda^{k+1}(T)}^2
    +
    \mathrm s_{\mathcal Q,T}^{k,\ell}
    (\underline v_T,\underline v_T)
  \right).
\]
We also define the hybrid $L^2$ norm
\[
  \|\underline v_h\|_{0,h}^2
  \coloneqq
  \sum_{T\in\mathcal T_h}
  \left(
    \|v_T\|_{L^2\Lambda^k(T)}^2
    +
    \sum_{F\in\mathcal F_T}
    h_F
    \|v_F\|_{L^2\Lambda^k(F)}^2
  \right),
\]
and the canonical broken graph norm
\[
  \|\underline v_h\|_{\boldsymbol X,h}^2
  \coloneqq
  \|\underline v_h\|_{0,h}^2
  +
  |\underline v_h|_{\DIFF,h}^2.
\]

\begin{lemma}[Polynomial consistency of the stabilisation]
\label{lem:stabilisation.polynomial.consistency}
Let $p\in\PL{\ell}{k}(T)$ and let $\underline p_T=\uI{k}{T}p$.
Then
\[
  \mathrm s_{\mathcal Q,T}^{k,\ell}
  (\underline p_T,\underline w_T)
  =0
  \qquad
  \forall\underline w_T\in\uX{k}{T}.
\]
\end{lemma}

\begin{proof}
The cell component of $\uI{k}{T}p$ is $p$, while its face component on $F$ is $\piQ{F}{k}\trc_Fp$.
Every residual in \eqref{eq:def.stabilization.forms} therefore vanishes.
\end{proof}

The stabilisation controls only the projected trace residual $  \piQ{F}{k}\trc_Fv_T-v_F$.
The following result shows that the stability-minimal kernel-trace inclusion is precisely what is needed to recover the full cell-to-face mismatch.
This is the local projected-jump stability mechanism underlying both main stability results.

Define the full jump seminorm
\[
  J_T(\underline v_T)^2
  \coloneqq
  \sum_{F\in\mathcal F_T}
  h_F^{-1}
  \|\trc_Fv_T-v_F\|_{L^2\Lambda^k(F)}^2.
\]

\begin{proposition}[Control of the full jump]
\label{prop:full.jump.control.forms}
For every admissible face-space family satisfying \eqref{eq:def.Qtrace.admissible}, every $T\in\mathcal T_h$, and every $\underline v_T\in\uX{k}{T}$, one has
\begin{equation}\label{eq:full.jump.control.forms}
  J_T(\underline v_T)^2
  \lesssim
  \|\DIFF_Tv_T\|_{L^2\Lambda^{k+1}(T)}^2
  +
  \mathrm s_{\mathcal Q,T}^{k,\ell}
  (\underline v_T,\underline v_T).
\end{equation}
\end{proposition}

\begin{proof}
Let $z_T$ be given by Proposition~\ref{prop:stable.polynomial.kernel.decomposition}.
Since $\trc_Fz_T\in\mathcal S_F^{k,\ell}\subseteq\Qtrace{F}{k}$, $(\mathrm{Id}-\piQ{F}{k})\trc_Fv_T = (\mathrm{Id}-\piQ{F}{k})\trc_F(v_T-z_T)$.
Hence
\[
\begin{aligned}
  \|\trc_Fv_T-v_F\|_F
  &\leq
  \|(\mathrm{Id}-\piQ{F}{k})\trc_Fv_T\|_F
  +
  \|\piQ{F}{k}\trc_Fv_T-v_F\|_F
  \\
  &\leq
  \|\trc_F(v_T-z_T)\|_F
  +
  \|\piQ{F}{k}\trc_Fv_T-v_F\|_F.
\end{aligned}
\]
Squaring, multiplying by $h_F^{-1}$, summing over the faces, and using Proposition~\ref{prop:stable.polynomial.kernel.decomposition} together with the uniform comparability of incident mesh sizes prove \eqref{eq:full.jump.control.forms}.
\end{proof}

\subsection{Comparison of canonical and reconstructed energies}
\label{subsec:hybrid.energy.comparison}

The canonical seminorm is defined independently of any particular reconstruction target and serves as the reference quantity throughout the stability analysis.
In practical discretisations, however, the energy is often expressed in terms of a reconstructed exterior derivative associated with a chosen local reconstruction space.
It is therefore useful to compare the resulting reconstruction-based energy with the canonical one.
The following proposition separates the direct and reverse comparison estimates.

For a reconstruction family $\mathcal R=(\mathcal R_T^{k+1})_{T\in\mathcal T_h}$, define the reconstructed seminorm
\[
  |\underline v_h|_{\mathcal R,\DIFF,h}^2
  \coloneqq
  \sum_{T\in\mathcal T_h}
  \left(
    \|\DIFF_{\mathcal R,T}^{k}\underline v_T\|_{L^2\Lambda^{k+1}(T)}^2
    +
    \mathrm s_{\mathcal Q,T}^{k,\ell}
    (\underline v_T,\underline v_T)
  \right).
\]

\begin{proposition}[Comparison of canonical and reconstructed seminorms]
\label{prop:comparison.broken.reconstructed.energies}
The following assertions hold.

\begin{enumerate}

\item Assume that the reconstruction spaces satisfy the uniform inverse trace estimate
\begin{equation}\label{eq:reconstruction.inverse.trace}
  \sum_{F\in\mathcal F_T}
  h_F
  \left\|
    \star_F^{-1}
    \trc_F(\star_Tq_T)
  \right\|_F^2
  \lesssim
  \|q_T\|_T^2
  \qquad
  \forall q_T\in\mathcal R_T^{k+1},
  \quad
  \forall T\in\mathcal T_h.
\end{equation}
Then
\begin{equation}\label{eq:reconstructed.energy.controlled.by.broken}
  |\underline v_h|_{\mathcal R,\DIFF,h}
  \lesssim
  |\underline v_h|_{\DIFF,h}
  \qquad
  \forall\underline v_h\in\uXh{k}.
\end{equation}

\item Assume instead that, for every $T\in\mathcal T_h$,
\begin{align}
  \DIFF_T\PL{\ell}{k}(T)
  &\subseteq
  \mathcal R_T^{k+1},
  \label{eq:energy.reverse.cell.inclusion}
  \\
  \star_F^{-1}
  \trc_F
  \left(
    \star_T\DIFF_T\PL{\ell}{k}(T)
  \right)
  &\subseteq
  \Qtrace{F}{k}
  \qquad
  \forall F\in\mathcal F_T.
  \label{eq:energy.reverse.face.inclusion}
\end{align}
Then
\begin{equation}\label{eq:broken.energy.controlled.by.reconstructed}
  |\underline v_h|_{\DIFF,h}
  \lesssim
  |\underline v_h|_{\mathcal R,\DIFF,h}
  \qquad
  \forall\underline v_h\in\uXh{k}.
\end{equation}

\end{enumerate}

Consequently, if all the assumptions above hold, the two seminorms are uniformly equivalent and have the same kernel.
\end{proposition}
\begin{proof}
For the first assertion, let $q_T\in\mathcal R_T^{k+1}$.
An integration by parts in \eqref{eq:def.local.derivative.general.forms} gives
\[
\begin{aligned}
  \int_T
  \DIFF_{\mathcal R,T}^{k}\underline v_T
  \wedge\star_Tq_T
  ={}&
  \int_T
  \DIFF_Tv_T\wedge\star_Tq_T
  \\
  &+
  \sum_{F\in\mathcal F_T}
  \epsTF
  \int_F
  (v_F-\trc_Fv_T)
  \wedge
  \trc_F(\star_Tq_T).
\end{aligned}
\]
Cauchy--Schwarz, \eqref{eq:reconstruction.inverse.trace}, and the definition of $J_T$ yield $\|\DIFF_{\mathcal R,T}^{k}\underline v_T\|_T \lesssim \|\DIFF_Tv_T\|_T + J_T(\underline v_T)$.
Using Proposition~\ref{prop:full.jump.control.forms}, squaring, and summing over the cells proves \eqref{eq:reconstructed.energy.controlled.by.broken}.

For the second assertion, assume \eqref{eq:energy.reverse.cell.inclusion}-- \eqref{eq:energy.reverse.face.inclusion}.
We may choose $q_T=\DIFF_Tv_T$ in the defining equation of the reconstruction.
Using \eqref{eq:energy.reverse.face.inclusion}, the face residual can be projected onto $\Qtrace{F}{k}$, and we obtain
\[
\begin{aligned}
  \|\DIFF_Tv_T\|_T^2
  ={}&
  \int_T
  \DIFF_{\mathcal R,T}^{k}\underline v_T
  \wedge
  \star_T\DIFF_Tv_T
  \\
  &-
  \sum_{F\in\mathcal F_T}
  \epsTF
  \int_F
  \left(
    v_F-\piQ{F}{k}\trc_Fv_T
  \right)
  \wedge
  \trc_F
  \left(
    \star_T\DIFF_Tv_T
  \right).
\end{aligned}
\]
The polynomial inverse trace inequality therefore gives
\[
  \|\DIFF_Tv_T\|_T
  \lesssim
  \|\DIFF_{\mathcal R,T}^{k}\underline v_T\|_T
  +
  \mathrm s_{\mathcal Q,T}^{k,\ell}
  (\underline v_T,\underline v_T)^{1/2}.
\]
After squaring and summing over the mesh, we obtain \eqref{eq:broken.energy.controlled.by.reconstructed}.
\end{proof}

\begin{remark}[Role of the two energies and computational trade-off]
\label{rem:role.broken.reconstructed.energies}
The seminorm $|\cdot|_{\DIFF,h}$ is intrinsic to the hybrid space and is available for every stability-admissible face family satisfying \eqref{eq:def.Qtrace.admissible}.
It is therefore the reference seminorm for both main stability results.
A reconstructed seminorm may replace it on the right-hand side of the corresponding stability estimates whenever the reverse comparison \eqref{eq:broken.energy.controlled.by.reconstructed} holds, while full equivalence additionally shows that the two seminorms have the same kernel.

This distinction also has a computational consequence.
Enriching a reconstruction space does not add globally coupled unknowns and changes only cell-local, parallelisable problems, whereas enriching the face spaces increases the size of the global system after static condensation.
The canonical broken energy is therefore available with stability-minimal face spaces, while a reconstruction-based formulation may require enrichment of the local reconstruction spaces or of the skeletal spaces in order to verify the reverse comparison.
\end{remark}

\subsection{Vector proxies in three space dimensions}
\label{sec:vector.proxies}

We conclude this section by spelling out the vector proxies in the practical case $n=3$.
This subsection provides a dictionary for the intrinsic construction and illustrates the distinction between stability-minimal face spaces and reconstruction spaces.
For every oriented face $F$, $\boldsymbol n_F$ denotes the unit normal compatible with the chosen orientation.
We use the two-dimensional proxy convention of \cite[Example~2 and Appendix~A]{Bonaldi.Di-Pietro.ea:24}. If a face $1$-form $\alpha_F$ has standard tangential vector proxy $\boldsymbol\alpha_F$, its rotated tangential proxy is the clockwise rotation $\boldsymbol\alpha_{F,\tau} \coloneqq \varrho_{-\pi/2}\boldsymbol\alpha_F$.
Equivalently, this is the vector proxy induced by $\star_F^{-1}$.
If $\alpha_F=\trc_F\alpha_T$ and the cell $1$-form $\alpha_T$ has vector proxy $\boldsymbol\alpha_T$, then
\[
\boldsymbol\alpha_{F,\tau}
=\boldsymbol\alpha_{T|F}\times\boldsymbol n_F.
\]
Under this convention, $\DIFF_F\PL{\ell+1}{0}(F) \,\leftrightarrow\, \operatorname{rot}_F\mathcal P^{\ell+1}(F)$.

With these identifications, the stability-minimal face spaces correspond to
\[
\mathcal S_F^{0,\ell}
\,\leftrightarrow\,
\mathcal P^0(F),
\qquad
\mathcal S_F^{1,\ell}
\,\leftrightarrow\,
\operatorname{rot}_F\mathcal P^{\ell+1}(F),
\qquad
\mathcal S_F^{2,\ell}
\,\leftrightarrow\,
\mathcal P^\ell(F).
\]
Consequently, admissible scalar face spaces for the hybrid gradient satisfy
\begin{equation}\label{eq:vector.proxy.admissible.gradient.face.space}
\mathcal P^0(F)
\subseteq \mathcal Q_F^{0,\ell}
\subseteq \mathcal P^\ell(F),
\end{equation}
whereas admissible rotated tangential face spaces for the hybrid curl satisfy
\begin{equation}\label{eq:vector.proxy.admissible.curl.face.space}
  \operatorname{rot}_F\mathcal P^{\ell+1}(F)
  \subseteq
  \boldsymbol{\mathcal Q}_F^\ell
  \subseteq
  \mathcal P^\ell(F)^2.
\end{equation}
For the hybrid divergence, the stability-minimal normal face space is already the full scalar polynomial space $\mathcal P^\ell(F)$.

\subsubsection{Zero-forms and hybrid gradient reconstructions}
\label{subsec:proxy.gradient}

For $k=0$, the hybrid unknown is $\underline v_T = (v_T,(v_F)_{F\in\mathcal F_T})$, $v_T\in\mathcal P^\ell(T)$, $v_F\in\mathcal Q_F^{0,\ell}$, with $\mathcal Q_F^{0,\ell}$ satisfying \eqref{eq:vector.proxy.admissible.gradient.face.space}.
The standard HHO scalar choice is $\mathcal Q_F^{0,\ell}=\mathcal P^\ell(F)$.
Given a vector polynomial reconstruction space $\boldsymbol{\mathcal R}_T^{\mathrm{grad}}$, the reconstruction $\boldsymbol G_{\mathcal R,T}\underline v_T \in \boldsymbol{\mathcal R}_T^{\mathrm{grad}}$ is characterised by
\[
  (\boldsymbol G_{\mathcal R,T}\underline v_T,
   \boldsymbol q)_T
  =
  -(v_T,\operatorname{div}\boldsymbol q)_T
  +
  \sum_{F\in\mathcal F_T}
  \epsTF
  (v_F,\boldsymbol q_{|F}\cdot\boldsymbol n_F)_F
  \qquad
  \forall\boldsymbol q\in
  \boldsymbol{\mathcal R}_T^{\mathrm{grad}}.
\]
Typical choices are $\nabla\mathcal P^{\ell+1}(T)$ and a full vector polynomial space.
The stabilisation is
\[
  \mathrm s_{\operatorname{grad},T}
  (\underline v_T,\underline w_T)
  \coloneqq
  \sum_{F\in\mathcal F_T}
  h_F^{-1}
  \bigl(
    \piQ{F}{0}(v_{T|F})-v_F,
    \piQ{F}{0}(w_{T|F})-w_F
  \bigr)_F.
\]

\subsubsection{One-forms and hybrid curl reconstructions}
\label{subsec:proxy.curl}

For $k=1$, write
\[
  \underline{\boldsymbol v}_T
  =
  \left(
    \boldsymbol v_T,
    (\boldsymbol v_{F,\tau})_{F\in\mathcal F_T}
  \right),
  \qquad
  \boldsymbol v_T\in\mathcal P^\ell(T)^3,
  \quad
  \boldsymbol v_{F,\tau}\in\boldsymbol{\mathcal Q}_F^\ell,
\]
with $\boldsymbol{\mathcal Q}_F^\ell$ satisfying \eqref{eq:vector.proxy.admissible.curl.face.space}.
Given a vector polynomial reconstruction space $\boldsymbol{\mathcal R}_T^{\mathrm{curl}}$, define $\boldsymbol C_{\mathcal R,T} \underline{\boldsymbol v}_T \in \boldsymbol{\mathcal R}_T^{\mathrm{curl}}$ by
\[
\begin{aligned}
(\boldsymbol C_{\mathcal R,T}
\underline{\boldsymbol v}_T,
\boldsymbol q)_T
={}
(\boldsymbol v_T,
\operatorname{curl}\boldsymbol q)_T
-\sum_{F\in\mathcal F_T}
\epsTF
\bigl(
\boldsymbol v_{F,\tau},
\boldsymbol n_F\times
(\boldsymbol q_{|F}\times\boldsymbol n_F)
\bigr)_F
\end{aligned}
\]
for every $\boldsymbol q\in\boldsymbol{\mathcal R}_T^{\mathrm{curl}}$.
The minus sign is the vector-proxy expression of the two-dimensional wedge product.
With the rotated proxy convention above,
\[
  \int_F
  v_F\wedge\trc_F(\star_Tq_T)
  =
  -
  \bigl(
    \boldsymbol v_{F,\tau},
    \boldsymbol n_F\times
    (\boldsymbol q_{|F}\times\boldsymbol n_F)
  \bigr)_F.
\]
Natural reconstruction spaces include $\operatorname{curl}\mathcal P^{\ell+1}(T)^3$ and $\mathcal P^m(T)^3$.
The first is the ambient exact-image target, whereas the second is a full polynomial reconstruction.
The associated stabilisation is
\[
\begin{aligned}
\mathrm s_{\operatorname{curl},T}
(\underline{\boldsymbol v}_T,
\underline{\boldsymbol w}_T)
\coloneqq
\sum_{F\in\mathcal F_T}
h_F^{-1}
\bigl(
\pi_{\boldsymbol{\mathcal Q},F}^{\ell}
(\boldsymbol v_{T|F}\times\boldsymbol n_F)
-\boldsymbol v_{F,\tau},
\pi_{\boldsymbol{\mathcal Q},F}^{\ell}
(\boldsymbol w_{T|F}\times\boldsymbol n_F)
-\boldsymbol w_{F,\tau}
\bigr)_F.
\end{aligned}
\]
The hybrid exterior-derivative seminorm becomes the $\boldsymbol H(\operatorname{curl})$-like hybrid seminorm used in the Weber analysis developed in \cite{Lemaire.Pitassi:25}, based on the broken curl and the projected rotated tangential jump.
Proposition~\ref{prop:comparison.broken.reconstructed.energies} provides conditions under which this seminorm is uniformly equivalent to the reconstruction-based energy associated with the local curl reconstructions $\boldsymbol C_{\mathcal R,T}$.

\subsubsection{Two-forms and hybrid divergence reconstructions}
\label{subsec:proxy.divergence}

For $k=2$, write
\[
  \underline{\boldsymbol v}_T
  =
  (\boldsymbol v_T,(v_{F,n})_{F\in\mathcal F_T}),
  \qquad
  \boldsymbol v_T\in\mathcal P^\ell(T)^3,
  \quad
  v_{F,n}\in\mathcal P^\ell(F).
\]
Given a scalar polynomial reconstruction space $\mathcal R_T^{\mathrm{div}}$, define $D_{\mathcal R,T}\underline{\boldsymbol v}_T \in \mathcal R_T^{\mathrm{div}}$ by
\[
  (D_{\mathcal R,T}\underline{\boldsymbol v}_T,q)_T
  =
  -(\boldsymbol v_T,\operatorname{grad}q)_T
  +
  \sum_{F\in\mathcal F_T}
  \epsTF
  (v_{F,n},q_{|F})_F
  \qquad
  \forall q\in\mathcal R_T^{\mathrm{div}}.
\]
The stability-minimal normal face space is already the full scalar polynomial space.
The corresponding stabilisation is
\[
  \mathrm s_{\operatorname{div},T}
  (\underline{\boldsymbol v}_T,
   \underline{\boldsymbol w}_T)
  \coloneqq
  \sum_{F\in\mathcal F_T}
  h_F^{-1}
  \bigl(
    \boldsymbol v_{T|F}\cdot\boldsymbol n_F-v_{F,n},
    \boldsymbol w_{T|F}\cdot\boldsymbol n_F-w_{F,n}
  \bigr)_F.
\]

\section{Functional trace framework on polytopal meshes}
\label{sec:functional.framework}

Throughout this section, fix $k\in\{0,\ldots,n-2\}$ and a polynomial degree $m\geq1$.

\paragraph{Scaled graph norms on the auxiliary polynomial spaces.}

For every mesh cell $f$ and every polynomial differential form $\nu_f$ on $\Sh(f)$, we use the scaled graph norm
\begin{equation}\label{eq:polynomial.graph.norm}
  \vvvert\nu_f\vvvert_f^2
  \coloneqq
  \|\nu_f\|_f^2
  +
  h_f^2\|\DIFF_f\nu_f\|_f^2.
\end{equation}
Unless otherwise specified, all dual norms on the auxiliary polynomial test spaces are induced by \eqref{eq:polynomial.graph.norm}.
Since the polynomial degrees are fixed, the polynomial inverse estimate gives $\vvvert\nu_f\vvvert_f \simeq \|\nu_f\|_f$ uniformly on each of the polynomial spaces used below.

\subsection{Stable trace extensions and geometric decomposition}
\label{subsubsec:geometric.decomposition}

\begin{proposition}[Stable polynomial trace extension]
\label{lem:stable.polynomial.trace.extension}
Let $f\in\FM{d}(\Mh)$ be a $d$-cell, let $k\in\{0,\ldots,d-1\}$, and let $r\geq1$.
Then there exists a linear operator $\mathcal E_{r,f}^k:\PLtrim{r}{k}(\Sh(\partial f))\to\PLtrim{r}{k}(\Sh(f))$ such that, for every $\theta_{\partial f}\in\PLtrim{r}{k}(\Sh(\partial f))$,
\begin{equation}\label{eq:trace.extension.trace}
  \trc_{\partial f}\mathcal E_{r,f}^k\theta_{\partial f}
  =
  \theta_{\partial f},
\end{equation}
and
\begin{equation}\label{eq:trace.extension.stability}
  \|\mathcal E_{r,f}^k\theta_{\partial f}\|_{L^2\Lambda^k(f)}
  +
  h_f\|\DIFF\mathcal E_{r,f}^k\theta_{\partial f}\|_{L^2\Lambda^{k+1}(f)}
  \lesssim
  h_f^{1/2}\|\theta_{\partial f}\|_{L^2\Lambda^k(\partial f)}.
\end{equation}
The hidden constant depends only on $d$, $k$, $r$, and the standing mesh-regularity parameters.
\end{proposition}

\begin{proof}
Let $\theta_{\partial f}\in\PLtrim{r}{k}(\Sh(\partial f))$.
We define $\mathcal E_{r,f}^k\theta_{\partial f}$ by prescribing its finite element degrees of freedom.
Recall that, for a simplex $S\in\Sh(f)$ of dimension $a\geq k$, the degrees of freedom attached to $S$ are $\int_S\trc_S\zeta\wedge\mu$ for all $\mu\in\PL{r+k-a-1}{a-k}(S)$.
For every $S\in\Sh(\partial f)$, we prescribe the degrees of freedom of $\mathcal E_{r,f}^k\theta_{\partial f}$ attached to $S$ to be equal to the corresponding degrees of freedom of $\theta_{\partial f}$.
For every $S\in\Sh(f)$ such that $S\not\subset\partial f$, we set all the degrees of freedom attached to $S$ equal to zero.
By unisolvence, these prescriptions define a unique form $\mathcal E_{r,f}^k\theta_{\partial f}\in\PLtrim{r}{k}(\Sh(f))$, and the resulting operator is linear.

For every $(d-1)$-simplex $S\in\Sh(\partial f)$, all the degrees of freedom of $\trc_S\mathcal E_{r,f}^k\theta_{\partial f}$ and $(\theta_{\partial f})_{|S}$, including those attached to the subsimplices of $S$, coincide.
Unisolvence on $S$ therefore yields \eqref{eq:trace.extension.trace}.

Denote by $N_{\mathrm{dof}}(\zeta,S)$ the norm of the degrees of freedom of $\zeta$ attached to $S$.
The scaled norm equivalence associated with these degrees of freedom on the auxiliary simplicial submeshes $\Sh(f)$ and $\Sh(\partial f)$ gives
\[
  \begin{aligned}
  \|\mathcal E_{r,f}^k\theta_{\partial f}\|_{L^2\Lambda^k(f)}^2
  &\lesssim
  \sum_{a=k}^{d-1}
  h_f^{d-a}
  \sum_{\substack{S\in\Sh(\partial f)\\\dim S=a}}
  N_{\mathrm{dof}}(\theta_{\partial f},S)^2
  \\
  &=
  h_f
  \sum_{a=k}^{d-1}
  h_f^{d-1-a}
  \sum_{\substack{S\in\Sh(\partial f)\\\dim S=a}}
  N_{\mathrm{dof}}(\theta_{\partial f},S)^2
  \lesssim
  h_f\|\theta_{\partial f}\|_{L^2\Lambda^k(\partial f)}^2.
  \end{aligned}
\]
Finally, the polynomial inverse inequality on $\Sh(f)$ yields
\[
  h_f\|\DIFF\mathcal E_{r,f}^k\theta_{\partial f}\|_{L^2\Lambda^{k+1}(f)}
  \lesssim
  \|\mathcal E_{r,f}^k\theta_{\partial f}\|_{L^2\Lambda^k(f)}.
\]
Combining the last two estimates proves \eqref{eq:trace.extension.stability}.
\end{proof}

The following result provides a stable realisation, on the auxiliary simplicial submeshes of the polytopal cells, of the standard extension-based geometric decomposition of finite element differential forms on simplices \cite{Arnold.Falk.ea:06,Arnold.Falk.ea:09}.

Once a trace-compatible family of extension operators is available, the direct-sum decomposition follows by a standard recursive argument.
The main point specific to the present setting is the construction of such a family with uniform graph-norm stability across the hierarchy of polytopal cells.
Geometric decompositions of the same type are used in \cite{Licht:17,Christiansen.Licht:20} to decompose vertical distributional complexes into local complexes indexed by mesh cells in a simplicial setting.

\begin{lemma}[Trace-compatible geometric decomposition]
\label{lem:geometric.extensions}
Let $j\in\{k+1,\ldots,n-1\}$, $q_j\coloneqq j-k-1$, and fix a polynomial degree $r\geq 1$.

For every pair of cells $s\subset f$ with $q_j\leq\dim s$, there exists a linear operator
\[
  \mathcal E_{s,f}^j\colon
  \PLtrimz{r}{q_j}(\Sh(s))
  \longrightarrow
  \PLtrim{r}{q_j}(\Sh(f))
\]
such that $\mathcal E_{f,f}^j=\mathrm{Id}$, and, for every $e\in\FM{\dim f-1}(\partial f)$,
\begin{equation}\label{eq:geometric.extension.trace}
  \trc_e
  \mathcal E_{s,f}^j\zeta_s
  =
  \begin{cases}
    \mathcal E_{s,e}^j\zeta_s,
    &s\subset e,
    \\[1mm]
    0,
    &s\not\subset e.
  \end{cases}
\end{equation}

Moreover, one has the direct geometric decomposition
\begin{equation}\label{eq:geometric.direct.sum}
\PLtrim{r}{q_j}(\Sh(f))
=\bigoplus_{\substack{
s\subset f\\
q_j\leq\dim s
}}
\mathcal E_{s,f}^j
\PLtrimz{r}{q_j}(\Sh(s)).
\end{equation}
Equivalently, every $\nu_f\in\PLtrim{r}{q_j}(\Sh(f))$ admits a unique representation
\[
  \nu_f
  =
  \sum_{\substack{
    s\subset f\\
    q_j\leq\dim s
  }}
  \mathcal E_{s,f}^j\zeta_s,
  \qquad
  \zeta_s\in
  \PLtrimz{r}{q_j}(\Sh(s)).
\]

Setting $d\coloneqq\dim f$, the following uniform stability estimate holds.
\begin{equation}\label{eq:geometric.norm.equivalence}
  \vvvert\nu_f\vvvert_f^2
  \simeq
  \sum_{\substack{
    s\subset f\\
    q_j\leq\dim s
  }}
  h_f^{d-\dim s}
  \vvvert\zeta_s\vvvert_s^2,
\end{equation}
In particular,
\begin{equation}\label{eq:geometric.extension.stability}
  \vvvert
  \mathcal E_{s,f}^j\zeta_s
  \vvvert_f
  \lesssim
  h_f^{(d-\dim s)/2}
  \vvvert\zeta_s\vvvert_s.
\end{equation}
\end{lemma}
\begin{proof}
We construct the operators recursively with respect to the codimension $\dim f-\dim s$.

If $s=f$, we set $\mathcal E_{f,f}^j=\mathrm{Id}$.
Now let $s\subsetneq f$, and assume that the operators
\[
  \mathcal E_{s,e}^j\colon
  \PLtrimz{r}{q_j}(\Sh(s))
  \longrightarrow
  \PLtrim{r}{q_j}(\Sh(e))
\]
have already been constructed for every $e\subset\partial f$ containing $s$.

Notice that $q_j\leq\dim s\leq d-1$.
Hence Proposition~\ref{lem:stable.polynomial.trace.extension} applies on $f$, with form degree $q_j$ and polynomial degree $r$.

For $\zeta_s\in\PLtrimz{r}{q_j}(\Sh(s))$, define a boundary family on $\Sh(\partial f)$ by
\[
  \left.
  \Theta_{s,\partial f}^j\zeta_s
  \right|_e
  \coloneqq
  \begin{cases}
    \mathcal E_{s,e}^j\zeta_s,
    &s\subset e,
    \\[1mm]
    0,
    &s\not\subset e.
  \end{cases}
\]

We first verify that this boundary family is compatible.
Let $e_1,e_2\subset\partial f$ be two facets and let $g\subset\partial e_1\cap\partial e_2$ be a common subcell.
If $s\subset g$, the induction hypothesis gives $\trc_g \mathcal E_{s,e_1}^j\zeta_s = \mathcal E_{s,g}^j\zeta_s = \trc_g \mathcal E_{s,e_2}^j\zeta_s$.
If $s\not\subset g$, the induction hypothesis gives zero for both traces, independently of whether either facet contains $s$.
Thus $\Theta_{s,\partial f}^j\zeta_s \in \PLtrim{r}{q_j}(\Sh(\partial f))$.

We define
\begin{equation}\label{eq:geometric.extension.definition}
\mathcal E_{s,f}^j\zeta_s
\coloneqq
\mathcal E_{r,f}^{q_j}
\left(
\Theta_{s,\partial f}^j\zeta_s
\right),
\end{equation}
where $\mathcal E_{r,f}^{q_j}$ is the operator given by Proposition~\ref{lem:stable.polynomial.trace.extension}.
The right-inverse property of $\mathcal E_{r,f}^{q_j}$ gives \eqref{eq:geometric.extension.trace}.

We next prove the extension estimate.
The stability of the trace extension yields
\[
\begin{aligned}
\vvvert
  \mathcal E_{s,f}^j\zeta_s
  \vvvert_f^2
  &\lesssim
  h_f
  \left\|
    \Theta_{s,\partial f}^j\zeta_s
  \right\|_{\partial f}^2
  \\
  &=
  h_f
  \sum_{\substack{
    e\in\FM{d-1}(\partial f)\\
    s\subset e
  }}
  \|
    \mathcal E_{s,e}^j\zeta_s
  \|_e^2
  \\
  &\leq
  h_f
  \sum_{\substack{
    e\in\FM{d-1}(\partial f)\\
    s\subset e
  }}
  \vvvert
    \mathcal E_{s,e}^j\zeta_s
  \vvvert_e^2.
\end{aligned}
\]
By the induction hypothesis, $\vvvert \mathcal E_{s,e}^j\zeta_s \vvvert_e^2 \lesssim h_e^{d-1-\dim s} \vvvert\zeta_s\vvvert_s^2$.
The uniform comparability of incident mesh sizes and the uniformly bounded number of facets therefore give $\vvvert \mathcal E_{s,f}^j\zeta_s \vvvert_f^2 \lesssim h_f^{d-\dim s} \vvvert\zeta_s\vvvert_s^2$.
This proves \eqref{eq:geometric.extension.stability}.

We now prove the direct decomposition by induction on $d=\dim f$.
If $d=q_j$, every $q_j$-form on $f$ has zero trace and hence $\PLtrim{r}{q_j}(\Sh(f)) = \PLtrimz{r}{q_j}(\Sh(f))$.
The assertion then reduces to the summand indexed by the cell $s=f$.

Assume $d>q_j$, and let $\nu_f\in\PLtrim{r}{q_j}(\Sh(f))$.
On each facet, the induction hypothesis provides a unique geometric decomposition of the trace of $\nu_f$.
Compatibility of the traces on the intersections of the facets, together with uniqueness in the lower-dimensional decompositions, identifies a single coefficient $\zeta_s\in\PLtrimz{r}{q_j}(\Sh(s))$ for every proper subcell $s\subsetneq f$ with $q_j\leq\dim s$.
Thus
\[
\trc_{\partial f}\nu_f
=\sum_{\substack{
  s\subsetneq f\\
    q_j\leq\dim s
  }}
  \Theta_{s,\partial f}^j\zeta_s.
\]
By the linearity of the trace extension and \eqref{eq:geometric.extension.definition},
\[
  \mathcal E_{r,f}^{q_j}
  \trc_{\partial f}\nu_f
  =
  \sum_{\substack{
    s\subsetneq f\\
    q_j\leq\dim s
  }}
  \mathcal E_{s,f}^j\zeta_s.
\]
Consequently,
\begin{equation}\label{eq:proper.cell.coefficient}
  \zeta_f
  \coloneqq
  \nu_f
  -
  \mathcal E_{r,f}^{q_j}
  \trc_{\partial f}\nu_f
  \in
  \PLtrimz{r}{q_j}(\Sh(f)),
\end{equation}
which gives the required representation of $\nu_f$.

To prove uniqueness, suppose that
\[
  \sum_{\substack{
    s\subset f\\
    q_j\leq\dim s
  }}
  \mathcal E_{s,f}^j\zeta_s
  =0.
\]
Taking the traces on the facets of $f$ and using uniqueness of the lower-dimensional decompositions gives $\zeta_s=0$ for all $s\subsetneq f$.
The remaining cell coefficient is then also zero.
Hence the sum in \eqref{eq:geometric.direct.sum} is direct.

It remains to prove the norm equivalence.
The upper estimate follows from the triangle inequality, \eqref{eq:geometric.extension.stability}, and the uniformly bounded number of subcells.
\[
  \vvvert\nu_f\vvvert_f^2
  \lesssim
  \sum_{\substack{
    s\subset f\\
    q_j\leq\dim s
  }}
  h_f^{d-\dim s}
  \vvvert\zeta_s\vvvert_s^2.
\]

For the converse estimate, let $e\subset\partial f$ be a facet.
The geometric coefficients of $\trc_e\nu_f$ are precisely the coefficients $\zeta_s$ associated with subcells $s\subset e$.
The induction hypothesis therefore gives
\[
  \sum_{\substack{
    s\subset e\\
    q_j\leq\dim s
  }}
  h_e^{d-1-\dim s}
  \vvvert\zeta_s\vvvert_s^2
  \lesssim
  \vvvert
  \trc_e\nu_f
  \vvvert_e^2.
\]
Multiplying by $h_f$, summing over the facets, and using the uniform comparability of incident mesh sizes and the uniformly bounded number of facets, we obtain
\begin{equation}\label{eq:proper.subcell.lower.bound}
  \sum_{\substack{
    s\subsetneq f\\
    q_j\leq\dim s
  }}
  h_f^{d-\dim s}
  \vvvert\zeta_s\vvvert_s^2
  \lesssim
  h_f
  \sum_{e\in\FM{d-1}(\partial f)}
  \vvvert
  \trc_e\nu_f
  \vvvert_e^2.
\end{equation}

The polynomial inverse trace inequalities, together with $\DIFF_e\trc_e\nu_f=\trc_e\DIFF_f\nu_f$, give
\begin{equation}\label{eq:graph.trace.estimate}
  h_f
  \sum_{e\in\FM{d-1}(\partial f)}
  \vvvert
  \trc_e\nu_f
  \vvvert_e^2
  \lesssim
  \vvvert\nu_f\vvvert_f^2.
\end{equation}
Thus all proper-subcell coefficients are controlled by the graph norm of $\nu_f$.

Finally, from \eqref{eq:proper.cell.coefficient}, the stability of the trace extension and the polynomial inverse trace inequality give
\[
\begin{aligned}
  \vvvert\zeta_f\vvvert_f
  &\leq
  \vvvert\nu_f\vvvert_f
  +
  \vvvert
  \mathcal E_{r,f}^{q_j}
  \trc_{\partial f}\nu_f
  \vvvert_f
  \\
  &\lesssim
  \vvvert\nu_f\vvvert_f
  +
  h_f^{1/2}
  \|\trc_{\partial f}\nu_f\|_{\partial f}
  \\
  &\lesssim
  \vvvert\nu_f\vvvert_f.
\end{aligned}
\]
Combining this estimate with \eqref{eq:proper.subcell.lower.bound} and \eqref{eq:graph.trace.estimate} proves the reverse bound in \eqref{eq:geometric.norm.equivalence}.
\end{proof}

\subsection{Functional double complex and geometric coordinates}
\label{subsubsec:functional.coordinates}
For $j\in\{k+1,\ldots,n-1\}$, let $q_j\coloneqq j-k-1$.
For every $d\in\{q_j,\ldots,n\}$, define
\[
\mathbb X_d^j
\coloneqq
\bigtimes_{f\in\FM{d}(\Mh)}
(\PLtrim{m}{q_j}(\Sh(f)))'.
\]
In what follows, we denote elements of the spaces $\mathbb X_d^j$ by capital Greek letters such as $\Gamma$, $\Psi$, $\Xi$, and $\Upsilon$.
We equip $\mathbb X_d^j$ with the norm
\[
\|\Psi^d\|_{\mathbb X_d^j}^2
\coloneqq
\sum_{f\in\FM{d}(\Mh)}
h_f^{n-d}\|\Psi_f^d\|_{(\PLtrim{m}{q_j}(\Sh(f)))'}^2.
\]

For every $j\in\{k+1,\ldots,n-1\}$ and every $d\in\{q_j+1,\ldots,n\}$, we define the functional trace coboundary
\[
\mathbb B_d^j\colon
\mathbb X_{d-1}^j
\longrightarrow
\mathbb X_d^j
\]
by
\[
\left(\mathbb B_d^j\Psi^{d-1}\right)_f(\nu_f)
\coloneqq
\sum_{e\in\FM{d-1}(\partial f)}
\epsilon_{f,e}
\Psi_e^{d-1}\left(\trc_e\nu_f\right)
\]
for every $f\in\FM{d}(\Mh)$ and every $\nu_f\in\PLtrim{m}{q_j}(\Sh(f))$.
This is well-defined since traces preserve the trimmed spaces on the auxiliary simplicial submeshes.

For every $j\in\{k+1,\ldots,n-2\}$ and every $d\in\{q_{j+1},\ldots,n\}$, we define the transpose of the local exterior derivative by
\[
  \mathbb D_d^j\colon
  \mathbb X_d^{j+1}
  \longrightarrow
  \mathbb X_d^j,
  \qquad
  (\mathbb D_d^j\Phi^d)_f(\nu_f)
  \coloneqq
  \Phi_f^d(\DIFF_f\nu_f).
\]
Notice that $\mathbb B$ changes the cell dimension while preserving the test form degree, whereas $\mathbb D$ lowers the test form degree while preserving the cell dimension.

\begin{proposition}[Functional double complex]
\label{prop:functional.complex}

For every $j\in\{k+1,\ldots,n-1\}$ and $d\in\{q_j+1,\ldots,n-1\}$,
\begin{equation}\label{eq:functional.horizontal.complex}
  \mathbb B_{d+1}^j\mathbb B_d^j
  =
  0.
\end{equation}
For every $j\in\{k+2,\ldots,n-2\}$ and $d\in\{q_{j+1},\ldots,n\}$,
\begin{equation}\label{eq:functional.vertical.complex}
  \mathbb D_d^{j-1}\mathbb D_d^j
  =
  0.
\end{equation}
Finally, for every $j\in\{k+1,\ldots,n-2\}$ and $d\in\{q_{j+1},\ldots,n-1\}$,
\begin{equation}\label{eq:functional.double.complex.commutation}
  \mathbb B_{d+1}^j\mathbb D_d^j
  =
  \mathbb D_{d+1}^j\mathbb B_{d+1}^{j+1}.
\end{equation}

Consequently, for every $j\in\{k+1,\ldots,n-1\}$, the sequence
\[
0
\longrightarrow
\mathbb X_{q_j}^j
\xrightarrow{\mathbb B_{q_j+1}^j}
\mathbb X_{q_j+1}^j
\longrightarrow
\cdots
\longrightarrow
\mathbb X_n^j
\longrightarrow
0
\]
is a complex, and, for every $j\in\{k+1,\ldots,n-2\}$ and $d\in\{q_{j+1}+1,\ldots,n\}$, the diagram
\[
\begin{array}{ccc}
\mathbb X_{d-1}^{j+1}
&\xrightarrow{\ \mathbb B_d^{j+1}\ }&
\mathbb X_d^{j+1}
\\[2mm]
\underset{\mathbb D_{d-1}^{j}}{\big\downarrow}
&&
\underset{\mathbb D_d^{j}}{\big\downarrow}
\\[2mm]
\mathbb X_{d-1}^{j}
&\xrightarrow{\ \mathbb B_d^{j}\ }&
\mathbb X_d^{j}
\end{array}
\]
commutes.
\end{proposition}

\begin{proof}
Let $\Psi^{d-1}\in\mathbb X_{d-1}^j$, $g\in\FM{d+1}(\Mh)$, and $\nu_g\in\PLtrim{m}{q_j}(\Sh(g))$.
By transitivity of traces,
\[
\begin{aligned}
\left(
\mathbb B_{d+1}^j\mathbb B_d^j\Psi^{d-1}
\right)_g(\nu_g)
&=
\sum_{f\in\FM{d}(\partial g)}
\sum_{e\in\FM{d-1}(\partial f)}
\epsilon_{g,f}\epsilon_{f,e}
\Psi_e^{d-1}(\trc_e\nu_g)
\\
&=
\sum_{e\in\FM{d-1}(\partial g)}
\left(
\sum_{\substack{
f\in\FM{d}(\partial g)\\
e\subset\partial f
}}
\epsilon_{g,f}\epsilon_{f,e}
\right)
\Psi_e^{d-1}(\trc_e\nu_g)
\\
&=
0.
\end{aligned}
\]
This proves \eqref{eq:functional.horizontal.complex}.

Identity \eqref{eq:functional.vertical.complex} follows directly from $\DIFF_f^2=0$.

Finally, the commutation of traces with the exterior derivative gives
\[
\begin{aligned}
\left(
\mathbb B_{d+1}^j\mathbb D_d^j\Phi^d
\right)_g(\nu_g)
&=
\sum_{f\in\FM{d}(\partial g)}
\epsilon_{g,f}
\Phi_f^d(\DIFF_f\trc_f\nu_g)
\\
&=
\sum_{f\in\FM{d}(\partial g)}
\epsilon_{g,f}
\Phi_f^d(\trc_f\DIFF_g\nu_g)
\\
&=
\left(
\mathbb D_{d+1}^j
\mathbb B_{d+1}^{j+1}\Phi^d
\right)_g(\nu_g),
\end{aligned}
\]
which proves \eqref{eq:functional.double.complex.commutation}.
\end{proof}

We now use the geometric decomposition in Lemma~\ref{lem:geometric.extensions} to define a cochain complex isomorphic to $\left(\mathbb X_\bullet^j,\mathbb B_\bullet^j\right)$.

Let $\mathfrak S_h^j\coloneqq\left\{s\in\Mh:q_j\leq\dim s\right\}$.
For every cell $s\in\mathfrak S_h^j$ and every $d\geq\dim s$, set
\[
\mathbb C_{d,s}^j
\coloneqq
\bigtimes_{f\in\cof{d}(s)}
\left(\PLtrimz{m}{q_j}(\Sh(s))\right)'.
\]
We equip this space with the norm
\[
\|\lambda\|_{\mathbb C_{d,s}^j}^2
\coloneqq
\sum_{f\in\cof{d}(s)}
h_f^{n-2d+\dim s}
\|\lambda_f\|_{\left(\PLtrimz{m}{q_j}(\Sh(s))\right)'}^2.
\]
The direct sum of the coefficient spaces is equipped with the corresponding $\ell^2$ product norm.

We set $\mathbb C_{d,s}^j=\{0\}$ if $d<\dim s$.
For $d>\dim s$, define $\delta_{d,s}\colon \mathbb C_{d-1,s}^j \longrightarrow \mathbb C_{d,s}^j$ by
\[
\left(\delta_{d,s}\lambda\right)_f(q_s)
\coloneqq
\sum_{\substack{
e\in\FM{d-1}(\partial f)\\
s\subset e
}}
\epsilon_{f,e}\lambda_e(q_s)
\qquad
\forall q_s\in\PLtrimz{m}{q_j}(\Sh(s)).
\]
The incidence relation gives $\delta_{d+1,s} \delta_{d,s}=0$.

For every $d$, define the geometric coordinate map $\mathcal G_d^j\colon \mathbb X_d^j \longrightarrow \bigoplus_{s\in\mathfrak S_h^j} \mathbb C_{d,s}^j$ by
\begin{equation}\label{eq:geometric.coordinate.definition}
\left(\mathcal G_d^j\Psi^d\right)_{s,f}(q_s)
\coloneqq
\Psi_f^d\left(\mathcal E_{s,f}^jq_s\right)
\qquad
\forall q_s\in\PLtrimz{m}{q_j}(\Sh(s)).
\end{equation}

\begin{proposition}[Isomorphism with coface coefficient complexes]
\label{prop:functional.cochain.isomorphism}
Under the assumptions of Lemma~\ref{lem:geometric.extensions}, the map $\mathcal G_d^j$ is an isomorphism and
\begin{equation}\label{eq:cochain.commutativity}
\mathcal G_d^j\mathbb B_d^j
=
\left(
\bigoplus_{s\in\mathfrak S_h^j}\delta_{d,s}
\right)
\mathcal G_{d-1}^j.
\end{equation}
Moreover, $\mathcal G_d^j$ and $(\mathcal G_d^j)^{-1}$ are uniformly bounded so that
\begin{equation}\label{eq:coordinate.norm.equivalence}
\|\mathcal G_d^j\Psi^d\|_{
\bigoplus_{s\in\mathfrak S_h^j}\mathbb C_{d,s}^j
}
\simeq
\|\Psi^d\|_{\mathbb X_d^j}.
\end{equation}
\end{proposition}
\begin{proof}
By \eqref{eq:geometric.direct.sum}, every $\nu_f\in\PLtrim{m}{q_j}(\Sh(f))$ has a unique representation
\[
\nu_f
=
\sum_{\substack{
s\in\mathfrak S_h^j\\
s\subset f
}}
\mathcal E_{s,f}^jq_s,
\qquad
q_s\in\PLtrimz{m}{q_j}(\Sh(s)).
\]
Consequently, a functional on $\PLtrim{m}{q_j}(\Sh(f))$ is uniquely determined by its restrictions to the geometric components, and $\mathcal G_d^j$ in \eqref{eq:geometric.coordinate.definition} is an isomorphism.
Its inverse is characterised by
\[
\Psi_f^d
\left(
\sum_{\substack{
s\in\mathfrak S_h^j\\
s\subset f
}}
\mathcal E_{s,f}^jq_s
\right)
=
\sum_{\substack{
s\in\mathfrak S_h^j\\
s\subset f
}}
\left(
\mathcal G_d^j\Psi^d
\right)_{s,f}(q_s).
\]

To prove the commuting identity \eqref{eq:cochain.commutativity}, let $f\in\cof{d}(s)$ and $q_s\in\PLtrimz{m}{q_j}(\Sh(s))$.
Using \eqref{eq:geometric.extension.trace}, we obtain
\[
\begin{aligned}
\left(
\mathcal G_d^j\mathbb B_d^j\Psi^{d-1}
\right)_{s,f}(q_s)
&=
\sum_{e\in\FM{d-1}(\partial f)}
\epsilon_{f,e}
\Psi_e^{d-1}
\left(
\trc_e\mathcal E_{s,f}^jq_s
\right)
\\
&=
\sum_{\substack{
e\in\FM{d-1}(\partial f)\\
s\subset e
}}
\epsilon_{f,e}
\Psi_e^{d-1}
\left(
\mathcal E_{s,e}^jq_s
\right)
\\
&=
\left(
\delta_{d,s}\mathcal G_{d-1}^j\Psi^{d-1}
\right)_f(q_s).
\end{aligned}
\]
This proves \eqref{eq:cochain.commutativity}.

It remains to prove the norm equivalence in \eqref{eq:coordinate.norm.equivalence}.
Fix $f\in\FM{d}(\Mh)$ and, for every $s\in\mathfrak S_h^j$ such that $s\subset f$, set
\[
\lambda_{s,f}
\coloneqq
\left(
\mathcal G_d^j\Psi^d
\right)_{s,f}
\in
\left(
\PLtrimz{m}{q_j}(\Sh(s))
\right)'.
\]
We first prove the local equivalence
\begin{equation}\label{eq:local.coordinate.norm.equivalence}
\|\Psi_f^d\|_{(\PLtrim{m}{q_j}(\Sh(f)))'}^2
\simeq
\sum_{\substack{
s\in\mathfrak S_h^j\\
s\subset f
}}
h_f^{\dim s-d}
\|\lambda_{s,f}\|_
{\left(\PLtrimz{m}{q_j}(\Sh(s))\right)'}^2.
\end{equation}

Let $\nu_f\in\PLtrim{m}{q_j}(\Sh(f))$ and write its geometric decomposition as
\[
\nu_f
=
\sum_{\substack{
s\in\mathfrak S_h^j\\
s\subset f
}}
\mathcal E_{s,f}^jq_s,
\qquad
q_s\in\PLtrimz{m}{q_j}(\Sh(s)).
\]
By the definition of the coordinate functionals and a weighted Cauchy--Schwarz inequality,
\[
\begin{aligned}
|\Psi_f^d(\nu_f)|
&=
\left|
\sum_{\substack{
s\in\mathfrak S_h^j\\
s\subset f
}}
\lambda_{s,f}(q_s)
\right|
\\
&\le
\left(
\sum_{\substack{
s\in\mathfrak S_h^j\\
s\subset f
}}
h_f^{\dim s-d}
\|\lambda_{s,f}\|_
{\left(\PLtrimz{m}{q_j}(\Sh(s))\right)'}^2
\right)^{1/2}
\left(
\sum_{\substack{
s\in\mathfrak S_h^j\\
s\subset f
}}
h_f^{d-\dim s}
\vvvert q_s\vvvert_s^2
\right)^{1/2}.
\end{aligned}
\]
The stability of the geometric decomposition \eqref{eq:geometric.norm.equivalence} therefore gives
\[
\|\Psi_f^d\|_{(\PLtrim{m}{q_j}(\Sh(f)))'}^2
\lesssim
\sum_{\substack{
s\in\mathfrak S_h^j\\
s\subset f
}}
h_f^{\dim s-d}
\|\lambda_{s,f}\|_
{\left(\PLtrimz{m}{q_j}(\Sh(s))\right)'}^2.
\]

Conversely, for every admissible subcell $s\in\mathfrak S_h^j$ with $s\subset f$ and every $q_s\in\PLtrimz{m}{q_j}(\Sh(s))$,
\[
\begin{aligned}
|\lambda_{s,f}(q_s)|
&=
\left|
\Psi_f^d
\left(
\mathcal E_{s,f}^jq_s
\right)
\right|
\\
&\le
\|\Psi_f^d\|_{(\PLtrim{m}{q_j}(\Sh(f)))'}
\vvvert
\mathcal E_{s,f}^jq_s
\vvvert_f
\\
&\lesssim
h_f^{(d-\dim s)/2}
\|\Psi_f^d\|_{(\PLtrim{m}{q_j}(\Sh(f)))'}
\vvvert q_s\vvvert_s,
\end{aligned}
\]
where the last inequality follows from \eqref{eq:geometric.extension.stability}.
Hence
\[
h_f^{\dim s-d}
\|\lambda_{s,f}\|_
{\left(\PLtrimz{m}{q_j}(\Sh(s))\right)'}^2
\lesssim
\|\Psi_f^d\|_{(\PLtrim{m}{q_j}(\Sh(f)))'}^2.
\]
Summing over the subcells of $f$ and using their uniformly bounded number proves the reverse inequality in \eqref{eq:local.coordinate.norm.equivalence}.

We now sum the local equivalence over all $d$-cells.
By the definition of the product norms and by exchanging the order of summation,
\[
\begin{aligned}
\left\|
  \mathcal G_d^j\Psi^d
\right\|_{
  \bigoplus_{s\in\mathfrak S_h^j}
  \mathbb C_{d,s}^j
}^2
&=
\sum_{\substack{
  s\in\mathfrak S_h^j\\
  \dim s\leq d
}}
\sum_{f\in\cof{d}(s)}
h_f^{n-2d+\dim s}
\|\lambda_{s,f}\|_{
  \left(
    \PLtrimz{m}{q_j}(\Sh(s))
  \right)'
}^2
\\
&=
\sum_{f\in\FM{d}(\Mh)}
h_f^{n-d}
\sum_{\substack{
  s\in\mathfrak S_h^j\\
  s\subset f
}}
h_f^{\dim s-d}
\|\lambda_{s,f}\|_{
  \left(
    \PLtrimz{m}{q_j}(\Sh(s))
  \right)'
}^2
\\
&\simeq
\sum_{f\in\FM{d}(\Mh)}
h_f^{n-d}
\|\Psi_f^d\|_{
  \left(
    \PLtrim{m}{q_j}(\Sh(f))
  \right)'
}^2
\\
&=
\|\Psi^d\|_{\mathbb X_d^j}^2.
\end{aligned}
\]
This proves \eqref{eq:coordinate.norm.equivalence} and completes the proof.
\end{proof}

As an immediate consequence of Proposition~\ref{prop:functional.cochain.isomorphism}, we obtain the following description of the cohomology of the functional trace complex.
\begin{corollary}[Cohomology of the functional trace complex]
\label{cor:functional.cohomology}
There is an isomorphism of complexes
\[
\left(
\mathbb X_\bullet^j,\mathbb B_\bullet^j
\right)
\simeq
\bigoplus_{s\in\mathfrak S_h^j}
\left(
\mathbb C_{\bullet,s}^j,\delta_{\bullet,s}
\right),
\]
and therefore
\[
H^d
\left(
\mathbb X_\bullet^j,\mathbb B_\bullet^j
\right)
\simeq
\bigoplus_{s\in\mathfrak S_h^j}
H^d
\left(
\mathbb C_{\bullet,s}^j,\delta_{\bullet,s}
\right).
\]
After choosing a basis of $\PLtrimz{m}{q_j}(\Sh(s))$, the coefficient complex $\left( \mathbb C_{\bullet,s}^j,\delta_{\bullet,s}\right)$ is a direct sum of $\dim\PLtrimz{m}{q_j}(\Sh(s))$ scalar cochain complexes on the cofaces of $s$.
\end{corollary}

\subsection{Coface--link complexes}
\label{subsubsec:coface.link.complexes}

We now recall the coface--link structure associated with the regular cellular decomposition of $\overline\Omega$.
The topological properties below are standard local link properties of a regular cellular decomposition of a polytopal domain, while the uniform cardinality bound follows from the mesh-regularity assumptions in Subsection~\ref{sec:setting:mesh}.
For background on links in polyhedral and triangulated complexes, we refer to \cite[Sections~2.2.4 and~2.3.2]{Kozlov:08}.

Let $s\in\Mh$ be a cell of dimension $a\coloneqq\dim s<n$.
Its strict cofaces, together with their incidence relations, define a finite regular cell complex denoted by $\operatorname{Lk}_{\Mh}(s)$ and called the \emph{coface link} of $s$.
It has the following properties.

\begin{enumerate}

\item For every $d\in\{a+1,\ldots,n\}$, there is a bijection $\ell_{s,d}:\cof{d}(s)\longrightarrow\FM{d-a-1}(\operatorname{Lk}_{\Mh}(s))$.
For $d=a$, the unique coface $s$ is associated with the formal degree $-1$ term of the augmented cellular cochain complex of the link.

\item The bijections preserve incidence.
If $e\in\FM{d-1}(\partial f)$ and $s\subset e$, then $\ell_{s,d-1}(e)\subset\partial\ell_{s,d}(f)$, and every facet incidence in the link arises in this way.

\item There exist signs $\omega_{s,f}\in\{-1,1\}$, for $f\supset s$, such that, for every incidence $e\subset\partial f$ with $s\subset e$,
\begin{equation}\label{eq:link.incidence.compatibility}
  \epsilon_{f,e}
  =
  \omega_{s,f}\omega_{s,e}
  \epsilon^{\operatorname{Lk}(s)}
  _{\ell_{s,d}(f),\ell_{s,d-1}(e)}.
\end{equation}
The same convention is used for the incidence between the formal degree $-1$ term and the vertices of the link.

\item The link has dimension $\dim\operatorname{Lk}_{\Mh}(s)=n-a-1$.
If $s\not\subset\partial\Omega$, then $\operatorname{Lk}_{\Mh}(s)$ has the reduced cohomology of the sphere $S^{n-a-1}$.
If $s\subset\partial\Omega$, then $\operatorname{Lk}_{\Mh}(s)$ has the reduced cohomology of the ball $B^{n-a-1}$ and is therefore acyclic.

\item The number of cells of $\operatorname{Lk}_{\Mh}(s)$ is uniformly bounded with respect to $s$.
Indeed, its cells are in one-to-one correspondence with the cofaces of $s$, whose number is uniformly bounded by the mesh-regularity assumptions of Subsection~\ref{sec:setting:mesh}.

\end{enumerate}

The exactness properties used in the recursive construction of the stable skeleton are local consequences of this coface--link structure.
In particular, they depend only on the local topology of the mesh around the cell $s$ and require no assumption on the global topology of $\Omega$.

\begin{lemma}[Exactness and cohomology of the coface complexes]
\label{lem:coface.exactness}
Let $s\in\mathfrak S_h^j$ and set $a\coloneqq\dim s$.
For every $a<d<n$,
\begin{equation}\label{eq:coface.exactness}
  \Ker\delta_{d+1,s}
  =
  \operatorname{Im}\delta_{d,s}.
\end{equation}
If $s$ is a boundary cell, the same equality holds for $d=n$ as well, with the convention $\delta_{n+1,s}=0$.
If $s$ is an interior cell, the top-dimensional cohomology is
\begin{equation}\label{eq:coface.top.cohomology}
  H^n
  \left(
    \mathbb C_{\bullet,s}^j,
    \delta_{\bullet,s}
  \right)
  \simeq
  \left(\PLtrimz{m}{q_j}(\Sh(s))\right)'.
\end{equation}

Moreover, for every $d\in\{a+1,\ldots,n\}$ and every $\eta_s^d\in\operatorname{Im}\delta_{d,s}$, let $\lambda_s^{d-1}$ be the unique solution of
\[
  \delta_{d,s}\lambda_s^{d-1}
  =
  \eta_s^d
\]
of minimum norm $\|\cdot\|_{\mathbb C_{d-1,s}^j}$.
Then $\lambda_s^{d-1}$ depends linearly on $\eta_s^d$ and satisfies
\begin{equation}\label{eq:coface.poincare}
  \|\lambda_s^{d-1}\|_{\mathbb C_{d-1,s}^j}^2
  \lesssim
  \sum_{f\in\cof{d}(s)}
  h_f^{n-2d+\dim s+2}
  \|(\eta_s^d)_f\|_{(\PLtrimz{m}{q_j}(\Sh(s)))'}^2,
\end{equation}
uniformly with respect to the mesh cell $s$.
In particular, in every exact degree identified above, this estimate applies to every $\eta_s^d\in\Ker\delta_{d+1,s}$.
\end{lemma}

\begin{proof}
We first identify the coefficient complex with the augmented cellular cochain complex of the link.
Set
\[
  E_s^j
  \coloneqq
  \left(\PLtrimz{m}{q_j}(\Sh(s))\right)'.
\]
If $a=n$, the only nonzero space is $\mathbb C_{n,s}^j=E_s^j$, and the assertions are immediate.
We therefore assume $a<n$.
For a finite regular cell complex $L$ and a coefficient space $E$, we denote by $\widetilde C^\bullet(L;E)$ its augmented cellular cochain complex, with the standard degree $-1$ augmentation, by $\widetilde\delta$ its cellular coboundary, and by $\widetilde H^\bullet(L;E)$ the corresponding reduced cohomology; see, e.g., \cite[Section~3.1, pp.~199 and~203]{Hatcher:02}.
We omit the degree indices on $\widetilde\delta$ whenever they are clear from its source and target spaces.

For every $d\geq a$, define
\[
  \mathcal J_{d,s}^j\colon
  \mathbb C_{d,s}^j
  \longrightarrow
  \widetilde C^{d-a-1}
  \left(
    \operatorname{Lk}_{\Mh}(s);
    E_s^j
  \right)
\]
by
\[
  \left(
    \mathcal J_{d,s}^j\lambda
  \right)_{\ell_{s,d}(f)}
  \coloneqq
  \omega_{s,f}\lambda_f,
  \qquad
  f\in\cof{d}(s).
\]
When $d=a$, the unique component indexed by $s$ is identified with the degree $-1$ term of the augmented complex.
The coface--link bijection shows that $\mathcal J_{d,s}^j$ is an isomorphism.

Let $d>a$, $f\in\cof{d}(s)$, and $q_s\in\PLtrimz{m}{q_j}(\Sh(s))$.
By \eqref{eq:link.incidence.compatibility},
\[
\begin{aligned}
\left(
  \mathcal J_{d,s}^j\delta_{d,s}\lambda
\right)_{\ell_{s,d}(f)}(q_s)
&=
\omega_{s,f}
\sum_{\substack{
  e\in\FM{d-1}(\partial f)\\
  s\subset e
}}
\epsilon_{f,e}\lambda_e(q_s)
\\
&=
\sum_{\substack{
  e\in\FM{d-1}(\partial f)\\
  s\subset e
}}
\epsilon^{\operatorname{Lk}(s)}_
{\ell_{s,d}(f),\ell_{s,d-1}(e)}
\omega_{s,e}\lambda_e(q_s)
\\
&=
\left(
  \widetilde\delta\,
  \mathcal J_{d-1,s}^j\lambda
\right)_{\ell_{s,d}(f)}(q_s).
\end{aligned}
\]
Hence
\begin{equation}\label{eq:coface.link.commuting}
  \mathcal J_{d,s}^j\delta_{d,s}
  =
  \widetilde\delta\,
  \mathcal J_{d-1,s}^j,
\end{equation}
so that
\[
  \left(
    \mathbb C_{\bullet,s}^j,
    \delta_{\bullet,s}
  \right)
  \simeq
  \left(
    \widetilde C^{\bullet-a-1}
    \left(
      \operatorname{Lk}_{\Mh}(s);
      E_s^j
    \right),
    \widetilde\delta
  \right).
\]
Consequently,
\[
  H^d
  \left(
    \mathbb C_{\bullet,s}^j,
    \delta_{\bullet,s}
  \right)
  \simeq
  \widetilde H^{d-a-1}
  \left(
    \operatorname{Lk}_{\Mh}(s);
    E_s^j
  \right).
\]

If $s$ is an interior cell, its link has the reduced cohomology of $S^{n-a-1}$.
It therefore vanishes in degrees below $n-a-1$, which gives \eqref{eq:coface.exactness} for $d<n$, while its top-degree cohomology gives \eqref{eq:coface.top.cohomology}.
If $s$ is a boundary cell, its link has the reduced cohomology of a ball and is therefore acyclic in every degree, yielding exactness also for $d=n$.

We now prove \eqref{eq:coface.poincare}.
Let $\widehat\eta_s^d\coloneqq\mathcal J_{d,s}^j\eta_s^d$.
Since $\eta_s^d\in\operatorname{Im}\delta_{d,s}$, the commuting property \eqref{eq:coface.link.commuting} shows that $\widehat\eta_s^d$ belongs to the image of the corresponding coboundary of the augmented link complex.

Choose an orthonormal basis $(e_\alpha)_\alpha$ of $E_s^j$, and let $B$ be the scalar incidence matrix of the relevant map $\widetilde\delta$ in the corresponding consecutive link cochain spaces.
When $d=a+1$, $B$ is the augmented incidence matrix from degree $-1$ to degree $0$.
Its entries belong to $\{-1,0,1\}$, and its numbers of rows and columns are bounded by the number of cells of the link.
By the uniform link-cardinality bound above, these dimensions and hence the rank $\rho\coloneqq\operatorname{rank}B$ are uniformly bounded.
If $\rho=0$, then $\widehat\eta_s^d=0$ and the required estimate follows by taking the zero solution.
We may therefore assume that $\rho\geq1$ and choose row and column index sets $I$ and $J$, both of cardinality $\rho$, such that $M\coloneqq B_{I,J}$ is invertible.
Writing $\widehat\eta_s^d=\sum_\alpha\widehat\eta_{s,\alpha}^d e_\alpha$, we have $\widehat\eta_{s,\alpha}^d\in\operatorname{Im}B$ for every $\alpha$.
Define
\[
  \left(\overline\lambda_{s,\alpha}^{d-1}\right)_J
  \coloneqq
  M^{-1}\left(\widehat\eta_{s,\alpha}^d\right)_I,
  \qquad
  \left(\overline\lambda_{s,\alpha}^{d-1}\right)_{J^c}
  \coloneqq
  0.
\]
Since the columns indexed by $J$ span $\operatorname{Im}B$ and restriction to the rows indexed by $I$ is injective on this space, $B\overline\lambda_{s,\alpha}^{d-1}=\widehat\eta_{s,\alpha}^d$.
Since $M$ is integer-valued and invertible, Cramer's rule and Hadamard's inequality give $\|\overline\lambda_{s,\alpha}^{d-1}\|_{\ell^2} \leq\rho^{\rho/2}\|\widehat\eta_{s,\alpha}^d\|_{\ell^2}$.
Therefore, setting $\overline\lambda_s^{d-1}\coloneqq \sum_\alpha\overline\lambda_{s,\alpha}^{d-1}e_\alpha$, we obtain
\begin{equation}\label{eq:link.poincare.application}
  \widetilde\delta\,
  \overline\lambda_s^{d-1}
  =
  \widehat\eta_s^d,
  \qquad
  \|\overline\lambda_s^{d-1}\|_{\ell^2(E_s^j)}
  \lesssim
  \|\widehat\eta_s^d\|_{\ell^2(E_s^j)}.
\end{equation}

Define
\[
  \widetilde\lambda_s^{d-1}
  \coloneqq
  \left(\mathcal J_{d-1,s}^j\right)^{-1}
  \overline\lambda_s^{d-1}.
\]
The commuting property gives $\delta_{d,s}\widetilde\lambda_s^{d-1}=\eta_s^d$.

Set $h_{s,*}\coloneqq\max_{T\in\cof{n}(s)}h_T$.
Mesh regularity gives $h_f\simeq h_{s,*}$ for every $(d-1)$- or $d$-dimensional coface $f$ of $s$.
Moreover, the maps $\mathcal J_{d-1,s}^j$ and $\mathcal J_{d,s}^j$ only reorder the coefficients and multiply them by signs.
Therefore, setting $A\coloneqq n-2d+a+2$ and using \eqref{eq:link.poincare.application}, we obtain
\[
\begin{aligned}
\|\widetilde\lambda_s^{d-1}\|_{\mathbb C_{d-1,s}^j}^2
&\simeq
h_{s,*}^{A}
\sum_{\tau\in
\FM{d-a-2}(\operatorname{Lk}_{\Mh}(s))}
\|(\overline\lambda_s^{d-1})_\tau\|_{E_s^j}^2
\\
&\lesssim
h_{s,*}^{A}
\sum_{\tau\in
\FM{d-a-1}(\operatorname{Lk}_{\Mh}(s))}
\|(\widehat\eta_s^d)_\tau\|_{E_s^j}^2
\\
&\simeq
\sum_{f\in\cof{d}(s)}
h_f^{n-2d+a+2}
\|(\eta_s^d)_f\|_{E_s^j}^2.
\end{aligned}
\]
When $d=a+1$, the first sum in this display is understood as the single contribution associated with the formal degree $-1$ term of the augmented complex.

Since $\lambda_s^{d-1}$ has minimum norm among all the solutions of $\delta_{d,s}\lambda=\eta_s^d$, we have $\|\lambda_s^{d-1}\|_{\mathbb C_{d-1,s}^j} \leq \|\widetilde\lambda_s^{d-1}\|_{\mathbb C_{d-1,s}^j}$, and \eqref{eq:coface.poincare} follows.
Finally, $\lambda_s^{d-1}$ is the unique solution lying in the weighted orthogonal complement of $\Ker\delta_{d,s}$, and therefore depends linearly on $\eta_s^d$.
\end{proof}

\begin{corollary}[Horizontal exactness below the top dimension]
\label{cor:functional.horizontal.exactness}
For every $j\in\{k+1,\ldots,n-1\}$, the functional trace complex is exact in every degree below $n$.
More precisely,
\begin{equation}\label{eq:functional.horizontal.exactness}
  \Ker\mathbb B_{d+1}^j
  =
  \operatorname{Im}\mathbb B_d^j
  \qquad
  \forall d\in\{q_j+1,\ldots,n-1\},
\end{equation}
and
\begin{equation}\label{eq:functional.first.injectivity}
  \Ker\mathbb B_{q_j+1}^j
  =
  \{0\}.
\end{equation}
\end{corollary}

\begin{proof}
By Corollary~\ref{cor:functional.cohomology}, the cohomology of the functional trace complex is the direct sum of the cohomologies of the coface coefficient complexes $\left( \mathbb C_{\bullet,s}^j, \delta_{\bullet,s} \right)$.
Lemma~\ref{lem:coface.exactness} and the augmented degree-$-1$ exactness of the coface-link complexes show that all these cohomologies vanish in degrees below $n$.
This proves \eqref{eq:functional.horizontal.exactness} and \eqref{eq:functional.first.injectivity}.
\end{proof}

\subsection{Trace-supported data, relative closedness, and stable link solves}
\label{subsubsec:relative.closedness}

We now use the geometric coordinate representation to impose relative closedness without restricting the coefficient spaces $\mathbb C_{d,s}^j$ before solving the link equations.

\begin{definition}[Trace-supported data and relatively closed functionals]
Fix $j\in\{k+1,\ldots,n-1\}$.
A family $\Gamma^{j+1}\in\mathbb X_{j+1}^j$ is said to be \emph{trace-supported} if
\begin{equation}\label{eq:trace.supported}
  \Gamma_g^{j+1}(\nu_g)=0
  \qquad
  \forall g\in\FM{j+1}(\Mh),\quad
  \forall
  \nu_g\in
  \PLtrimz{m}{q_j}(\Sh(g)).
\end{equation}
Equivalently, each functional $\Gamma_g^{j+1}$ depends only on the trace of its argument.
In geometric coordinates, setting $\eta \coloneqq \mathcal G_{j+1}^j\Gamma^{j+1}$, condition \eqref{eq:trace.supported} means that the coordinate indexed by the cell $s=g$ vanishes.
Thus $(\eta_g)_g=0$ for every $g\in\FM{j+1}(\Mh)$.

When $q_j>0$, we will also use the condition
\begin{equation}\label{eq:proper.exact.datum.vanishing}
  (\eta_f)_g(\DIFF_f\chi_f)=0
\end{equation}
for every
\[
  f\in\FM{j}(\Mh),
  \qquad
  g\in\cof{j+1}(f),
  \qquad
  \chi_f\in
  \PLtrimz{m}{q_j-1}(\Sh(f)).
\]
When $q_j=0$, this condition is void.
Unlike \eqref{eq:trace.supported}, which removes the coordinate indexed by the $(j+1)$-cell itself, condition \eqref{eq:proper.exact.datum.vanishing} acts on the coordinate indexed by a $j$-cell.

A family $\Psi^j\in\mathbb X_j^j$ is said to be \emph{relatively closed} if, whenever $q_j>0$,
\begin{equation}\label{eq:relative.closedness}
  \Psi_f^j(\DIFF_f\chi_f)=0
  \qquad
  \forall f\in\FM{j}(\Mh),\quad
  \forall
  \chi_f\in
  \PLtrimz{m}{q_j-1}(\Sh(f)).
\end{equation}
When $q_j=0$, the condition is void.
When $q_j>0$, since $q_{j-1}=q_j-1$, condition \eqref{eq:relative.closedness} is equivalent to requiring that $\mathbb D_j^{j-1}\Psi^j\in\mathbb X_j^{j-1}$ be trace-supported.
Thus relative closedness ensures that the vertical differential of a diagonal functional has the trace-supported property required for the next horizontal solve.
\end{definition}

\begin{lemma}[Stable link solve with relative closedness]
\label{lem:full.link.solve}
Let $\Gamma^{j+1}\in\mathbb X_{j+1}^j$ be trace-supported and set
\[
  \eta
  \coloneqq
  \mathcal G_{j+1}^j\Gamma^{j+1}.
\]
Assume that every geometric coordinate $\eta_s$ with $\dim s\leq j$ belongs to $\operatorname{Im}\delta_{j+1,s}$ and that \eqref{eq:proper.exact.datum.vanishing} holds.

Then there exists a relatively closed $\Psi^j\in\mathbb X_j^j$ such that
\begin{equation}\label{eq:full.link.equation}
  \mathbb B_{j+1}^j\Psi^j
  =
  \Gamma^{j+1}.
\end{equation}
Moreover, $\Psi^j$ can be selected linearly in $\Gamma^{j+1}$ and satisfies
\begin{equation}\label{eq:full.link.stability}
  \|\Psi^j\|_{\mathbb X_j^j}^2
  \lesssim
  \sum_{g\in\FM{j+1}(\Mh)}
  h_g^{n-j+1}
  \|\Gamma_g^{j+1}\|_{
    \left(
      \PLtrim{m}{q_j}(\Sh(g))
    \right)'
  }^2.
\end{equation}

If $j+1<n$ and
\begin{equation}\label{eq:full.link.cocycle}
  \mathbb B_{j+2}^j\Gamma^{j+1}=0,
\end{equation}
then the componentwise range conditions follow from the cocycle condition.
\end{lemma}

\begin{proof}
The trace-supported condition removes the coordinates indexed by the $(j+1)$-cells themselves.
The geometric-coordinate commuting property reduces \eqref{eq:full.link.equation} to the following independent coface systems.
\begin{equation}\label{eq:full.link.coordinate.system}
  \delta_{j+1,s}\lambda_s=\eta_s,
  \qquad
  s\in\mathfrak S_h^j,\quad
  \dim s\leq j.
\end{equation}
By the range assumptions, for every cell $s$ we let $\lambda_s$ be the unique solution of the corresponding coface system in \eqref{eq:full.link.coordinate.system} of minimum norm $\|\cdot\|_{\mathbb C_{j,s}^j}$.
Set
\[
  \lambda
  \coloneqq
  (\lambda_s)_{\substack{
    s\in\mathfrak S_h^j\\
    \dim s\leq j
  }},
  \qquad
  \Psi^j
  \coloneqq
  (\mathcal G_j^j)^{-1}\lambda.
\]
The commuting property \eqref{eq:cochain.commutativity} and \eqref{eq:full.link.coordinate.system} then give $\mathbb B_{j+1}^j\Psi^j = \Gamma^{j+1}$, which proves \eqref{eq:full.link.equation}.

The coordinate norm equivalence and \eqref{eq:coface.poincare}, applied with $d=j+1$, give
\[
\begin{aligned}
  \|\Psi^j\|_{\mathbb X_j^j}^2
  &\lesssim
  \sum_{\substack{
    s\in\mathfrak S_h^j\\
    \dim s\leq j
  }}
  \|\lambda_s\|_{\mathbb C_{j,s}^j}^2
  \\
  &\lesssim
  \sum_{\substack{
    s\in\mathfrak S_h^j\\
    \dim s\leq j
  }}
  \sum_{g\in\cof{j+1}(s)}
  h_g^{n-2j+\dim s}
  \|(\eta_s)_g\|_{
    \left(
      \PLtrimz{m}{q_j}(\Sh(s))
    \right)'
  }^2.
\end{aligned}
\]
Exchanging the order of summation and using \eqref{eq:local.coordinate.norm.equivalence} on every $(j+1)$-cell $g$, together with the vanishing of the coordinate indexed by $s=g$, we obtain
\[
\begin{aligned}
  \|\Psi^j\|_{\mathbb X_j^j}^2
  &\lesssim
  \sum_{g\in\FM{j+1}(\Mh)}
  h_g^{n-j+1}
  \sum_{\substack{
    s\subset g\\
    \dim s\leq j
  }}
  h_g^{\dim s-j-1}
  \|(\eta_s)_g\|_{
    \left(
      \PLtrimz{m}{q_j}(\Sh(s))
    \right)'
  }^2
  \\
  &\lesssim
  \sum_{g\in\FM{j+1}(\Mh)}
  h_g^{n-j+1}
  \|\Gamma_g^{j+1}\|_{
    \left(
      \PLtrim{m}{q_j}(\Sh(g))
    \right)'
  }^2.
\end{aligned}
\]
This proves \eqref{eq:full.link.stability}.
By Lemma~\ref{lem:coface.exactness}, each minimum-norm solution $\lambda_s$ depends linearly on $\eta_s$.
Since the geometric coordinate maps are linear, the resulting $\Psi^j$ depends linearly on $\Gamma^{j+1}$.

It remains to prove that $\Psi^j$ is relatively closed.
If $q_j=0$, there is nothing to prove.
Assume $q_j>0$, let $f\in\FM{j}(\Mh)$, and, for $\chi_f \in \PLtrimz{m}{q_j-1}(\Sh(f))$, set $\beta_f \coloneqq \DIFF_f\chi_f$.
Since the trace commutes with the exterior derivative, $\trc_{\partial f}\beta_f = \DIFF_{\partial f} \trc_{\partial f}\chi_f = 0$, and therefore $\beta_f \in \PLtrimz{m}{q_j}(\Sh(f))$.

By uniqueness of the geometric decomposition \eqref{eq:geometric.direct.sum}, the decomposition of $\beta_f$ contains only the component indexed by $s=f$.
Since $\mathcal E_{f,f}^j=\mathrm{Id}$, it follows that $\Psi_f^j(\beta_f) = \lambda_{f,f}(\beta_f)$.

The only $j$-dimensional coface of the cell $f$ is $f$ itself.
Consequently, for every $g\in\cof{j+1}(f)$, the definition of the coface coboundary gives $(\delta_{j+1,f}\lambda_f)_g(\beta_f) = \epsilon_{g,f}\lambda_{f,f}(\beta_f)$.
On the other hand, \eqref{eq:full.link.coordinate.system} gives $(\delta_{j+1,f}\lambda_f)_g(\beta_f) = (\eta_f)_g(\beta_f)$.
Using \eqref{eq:proper.exact.datum.vanishing}, we obtain $\epsilon_{g,f}\lambda_{f,f}(\beta_f) = (\eta_f)_g(\beta_f) = 0$.
Since $j<n$, the cell $f$ belongs to at least one $(j+1)$-cell, and hence $\lambda_{f,f}(\beta_f)=0$.
It follows that $\Psi_f^j(\DIFF_f\chi_f) = \Psi_f^j(\beta_f) = 0$, which proves \eqref{eq:relative.closedness}.

Finally, assume that $j+1<n$ and that \eqref{eq:full.link.cocycle} holds.
Applying the geometric coordinate map and using \eqref{eq:cochain.commutativity}, the cocycle condition becomes $\delta_{j+2,s}\eta_s=0$ for all $s\in\mathfrak S_h^j$.
For every cell $s$ with $\dim s\leq j$, the coface complex is exact in degree $j+1$, which is below $n$, by Lemma~\ref{lem:coface.exactness}.
Therefore, $\eta_s\in\operatorname{Im}\delta_{j+1,s}$, so the componentwise range conditions follow from exactness.
\end{proof}

\section{Stable conforming lifting}
\label{sec:poincare.proof}

Let $\underline v_h\in\uXh{k}$ and set
\begin{equation}\label{eq:prescribed.differential.data}
  \sigma_T
  \coloneqq
  \DIFF_{\ell,T}^{k}\underline v_T
  \in
  \Dpoly{T}{k},
  \qquad
  T\in\mathcal T_h.
\end{equation}
The proof of Theorem~\ref{thm:main.hybrid.poincare} proceeds in two steps.
The first step constructs a stable compatible polynomial face skeleton whose moments reproduce the action of the prescribed family $\sigma_h=(\sigma_T)_{T\in\mathcal T_h}$ on the relevant closed test forms.
This is the only stage of the argument that addresses global compatibility.
The second step completes the resulting face forms through independent local first-order problems in the mesh cells.
It is entirely cellwise and produces a conforming lifting whose projected exterior derivative equals $\sigma_h$ and whose $L^2$ norm and exterior derivative are controlled by $|\underline v_h|_{\DIFF,h}$.

\paragraph*{Prescribed and auxiliary differential data.}

The prescribed reconstruction $\sigma_T$ is characterised by the hybrid Stokes formula on the test space
\[
  \mathcal M_T^k
  \coloneqq
  \star_T\Dpoly{T}{k},
  \qquad
  \mathcal Z_T^k
  \coloneqq
  \mathcal M_T^k\cap\Ker\DIFF_T.
\]
The recursive skeleton construction is instead organised on full trimmed polynomial test spaces on the auxiliary simplicial submeshes.
These spaces provide the exact complexes, trace structure, and geometric decompositions needed to construct the intermediate-dimensional hierarchy that is absent from the hybrid space.
We therefore supplement the prescribed reconstruction with an auxiliary full trimmed reconstruction built from the same hybrid datum, so that the hybrid Stokes relation becomes available on this larger test complex.

All polynomial differential forms on a mesh cell $f$ are understood on its fixed auxiliary simplicial submesh $\Sh(f)$.
We choose the degree $m$ in \eqref{eq:def.auxiliary.trimmed.derivative.space} sufficiently large, with $m\geq\ell+1$, so that
\[
  \Dpoly{T}{k}
  \subseteq
  \widehat{\mathcal R}_T^{k+1}
  \qquad
  \forall T\in\mathcal T_h,
\]
and set
\[
  \widehat{\mathcal M}_T^k
  \coloneqq
  \star_T\widehat{\mathcal R}_T^{k+1}
  =
  \PLtrim{m}{n-k-1}(\Sh(T)),
  \qquad
  \widehat{\mathcal Z}_T^k
  \coloneqq
  \widehat{\mathcal M}_T^k\cap\Ker\DIFF_T.
\]
The superscript $k$ in these test-space symbols refers to the hybrid input degree; their elements are $(n-k-1)$-forms.
The auxiliary reconstruction is defined by
\begin{equation}\label{eq:auxiliary.differential.data}
  \widehat\sigma_T
  \coloneqq
  \DIFF_{\widehat{\mathcal R},T}^{k}\underline v_T
  \in
  \widehat{\mathcal R}_T^{k+1}.
\end{equation}
It is used only in the analysis and does not modify either the hybrid unknowns or the globally coupled discrete space.

\begin{proposition}[Projection and stability of the auxiliary reconstruction]
\label{prop:auxiliary.projection.control}
For every $T\in\mathcal T_h$,
\begin{equation}\label{eq:auxiliary.projection}
  \piD{T}{k}\widehat\sigma_T
  =
  \sigma_T.
\end{equation}
Moreover,
\begin{equation}\label{eq:auxiliary.graph.control}
  \sum_{T\in\mathcal T_h}
  \left(
    \|\widehat\sigma_T\|_T^2
    +
    \|\sigma_T\|_T^2
  \right)
  \lesssim
  |\underline v_h|_{\DIFF,h}^2.
\end{equation}
\end{proposition}

\begin{proof}
Identity \eqref{eq:auxiliary.projection} follows by restricting the defining equation of $\widehat\sigma_T$ to
\[
  \Dpoly{T}{k}
  \subseteq
  \widehat{\mathcal R}_T^{k+1},
\]
since on this subspace the defining equation coincides with the one for $\sigma_T$.
Estimate \eqref{eq:auxiliary.graph.control} follows from Proposition~\ref{prop:comparison.broken.reconstructed.energies}, applied first with the prescribed reconstruction space $\Dpoly{T}{k}$ and then with the auxiliary space $\widehat{\mathcal R}_T^{k+1}$.
For both spaces, the inverse trace estimate \eqref{eq:reconstruction.inverse.trace} follows from the fixed-degree polynomial inverse inequalities on the uniformly shape-regular auxiliary simplicial submeshes.
\end{proof}

By construction,
\[
  \mathcal M_T^k
  \subseteq
  \widehat{\mathcal M}_T^k,
  \qquad
  \mathcal Z_T^k
  \subseteq
  \widehat{\mathcal Z}_T^k.
\]
Identity \eqref{eq:auxiliary.projection} also shows that $\widehat\sigma_T$ and $\sigma_T$ have the same moments against every test form in $\mathcal M_T^k$.
The auxiliary datum therefore provides the corresponding Stokes relation on the full trimmed test complex while leaving unchanged the prescribed moments on the original test space.

\paragraph*{Target of the skeleton construction.}

The global step of the proof constructs polynomial forms
\[
  \lambda^d
  =
  (\lambda_f)_{f\in\FM{d}(\Mh)},
  \qquad
  d\in\{k,\ldots,n-1\},
\]
that are compatible across successive cell dimensions in the sense that
\begin{equation}\label{eq:stable.skeleton.conformity}
  \trc_e\lambda_f
  =
  \lambda_e
  \qquad
  \forall f\in\FM{d}(\Mh),
  \quad
  \forall e\in\FM{d-1}(\partial f).
\end{equation}
For each positive-dimensional cell $f$, we write $\lambda_{\partial f}\coloneqq(\lambda_e)_{e\in\FM{\dim f-1}(\partial f)}$.
The face components must reproduce the auxiliary reconstruction against every closed full trimmed test form,
\begin{equation}\label{eq:auxiliary.skeleton.identity.overview}
  \sum_{F\in\mathcal F_T}
  \epsTF
  \int_F
  \lambda_F\wedge\trc_F\widehat\mu_T
  =
  \int_T
  \widehat\sigma_T\wedge\widehat\mu_T
  \qquad
  \forall T\in\mathcal T_h,
  \quad
  \forall\widehat\mu_T\in\widehat{\mathcal Z}_T^k.
\end{equation}
Restricting this identity to $\mathcal Z_T^k$ and using \eqref{eq:auxiliary.projection} gives the prescribed moment relation
\begin{equation}\label{eq:prescribed.skeleton.identity.overview}
  \sum_{F\in\mathcal F_T}
  \epsTF
  \int_F
  \lambda_F\wedge\trc_F\mu_T
  =
  \int_T
  \sigma_T\wedge\mu_T
  \qquad
  \forall T\in\mathcal T_h,
  \quad
  \forall\mu_T\in\mathcal Z_T^k.
\end{equation}
The construction must also satisfy the uniform estimate
\begin{equation}\label{eq:stable.skeleton.global.stability}
\begin{aligned}
  &
  \sum_{d=k}^{n-1}
  \sum_{f\in\FM{d}(\Mh)}
  h_f^{n-d}
  \|\lambda_f\|_f^2
  \\
  &\qquad+
  \sum_{d=k+1}^{n-1}
  \sum_{f\in\FM{d}(\Mh)}
  h_f^{n-d}
  \|\DIFF_f\lambda_f\|_f^2
  \lesssim
  |\underline v_h|_{\DIFF,h}^2.
\end{aligned}
\end{equation}
The contributions with $d=n-1$ are the weighted face estimates required by the final conforming completion.

\paragraph*{Organisation of the construction.}

Subsection~\ref{subsec:auxiliary.tools} collects the analytical ingredients used throughout the proof.
These include the discrete Poincar\'e inequality for cellular cochains, stable local potential problems, and polynomial moment realisations on the auxiliary simplicial submeshes.

Section~\ref{sec:functional.framework} introduces the functional double complex that generates the missing intermediate-dimensional hierarchy.
Its horizontal differential is the functional trace coboundary and its vertical differential is the transpose of the local exterior derivative.
A trace-compatible geometric decomposition transforms each horizontal row into a direct sum of cellular coefficient complexes on coface links.
The topology of these links provides the required exactness, while their uniformly bounded combinatorial complexity yields stable local horizontal inverses.

Subsection~\ref{subsec:recursive.skeleton} uses this framework to construct the compatible family in \eqref{eq:stable.skeleton.conformity}.
An auxiliary full trimmed hybrid Stokes formula first shows that the top functional datum belongs to the range of the horizontal coboundary.
A descending zig-zag procedure then alternates the vertical differential with stable horizontal solves and produces relatively closed functionals
\[
  \Xi^d
  \in
  \mathbb X_d^d,
  \qquad
  d\in\{k+1,\ldots,n-1\}.
\]
At the terminal level, evaluation on constants defines a cellular $(k+1)$-cochain by
\[
  \bar\xi_f
  \coloneqq
  \Xi_f^{k+1}(1).
\]
Comparison with an explicit reference recursive family shows that $\bar\xi$ is a cellular coboundary.
A stable cellular Poincar\'e solve then determines a $k$-cochain potential and constitutes the only non-local solve in the proof.

The construction is subsequently reversed.
The terminal functionals are represented by top-degree polynomial forms on the $(k+1)$-cells, while the cellular potential is realised by normalised top-degree forms on the $k$-cells.
At each successive dimension, the functionals $\Xi_f^d$ are realised by closed polynomial forms $\xi_f$ with compatible traces.
Local zero-trace moment corrections enforce the prescribed moments, and local first-order problems produce forms $\lambda_f$ satisfying
\[
  \DIFF_f\lambda_f
  =
  \xi_f,
  \qquad
  \trc_{\partial f}\lambda_f
  =
  \lambda_{\partial f}.
\]
This ascending recursion transfers the terminal cochain estimate to the compatible polynomial hierarchy and yields the moment identities \eqref{eq:auxiliary.skeleton.identity.overview} and \eqref{eq:prescribed.skeleton.identity.overview}, together with the stability estimate \eqref{eq:stable.skeleton.global.stability}.
When $k=n-1$, the descending and ascending recursions are absent, and the face skeleton is obtained directly from the hybrid Stokes identity and the cellular cochain Poincar\'e inequality.

Finally, Subsection~\ref{subsec:hybrid.completion} completes the stable face forms in the mesh cells by local moment completions and local first-order problems with prescribed trace.
The resulting conforming form $\lambda_h$ has projected exterior derivative equal to the prescribed reconstructed exterior derivative of $\underline v_h$ and satisfies the stability estimate required in Theorem~\ref{thm:main.hybrid.poincare}.
Since the global compatibility information required by the completion is encoded in the stable skeleton, this final step is entirely cellwise and requires no further non-local solve.

\subsection{Auxiliary analytical tools}
\label{subsec:auxiliary.tools}

\subsubsection{Cochain Poincar\'e inequality and local finite element problems}
\label{subsubsec:cochain.and.local.problems}

We first recall two auxiliary results that will be used in the construction.
The first is a discrete Poincar\'e inequality for scalar cochains on the cellular mesh, while the second concerns local finite element problems for polynomial differential forms on the auxiliary simplicial submeshes.
Both results are quoted from the literature in the form needed below, with the notation adapted to the present setting.

Recall that, for any $k\in\{0,\ldots,n\}$, a \emph{$k$-cochain} is a family of real numbers associated with the $k$-cells of the mesh $\Mh$.

Throughout the proof, Greek letters are used for cochain data and potentials.
We use symbols such as $\xi$ and $\eta$ for scalar or coefficient-cochain data, and $\lambda$ and $\vartheta$ for their preimages under the cellular coboundary.
We slightly abuse notation by using the same Greek symbols for cellular cochains and local polynomial differential forms; the intended meaning is determined by the ambient space.

\begin{lemma}[Discrete Poincar\'e inequality on cochains]
\label{lem:cochain.poincare}
Let $k\in\{0,\ldots,n-1\}$ and let $\xi=(\xi_f)_{f\in\FM{k+1}(\Mh)}$ be a $(k+1)$-cochain.
Assume that there exists a $k$-cochain $\lambda=(\lambda_e)_{e\in\FM{k}(\Mh)}$ such that
\begin{equation}\label{eq:face.cell.coboundary.assumption}
  \sum_{e\in\FM{k}(\partial f)}\epsilon_{f,e}\lambda_e=\xi_f
  \qquad
  \forall f\in\FM{k+1}(\Mh).
\end{equation}
Then the set of $k$-cochains satisfying
\begin{equation}\label{eq:face.cell.coboundary.identity}
  \sum_{e\in\FM{k}(\partial f)}\epsilon_{f,e}\widehat\lambda_e=\xi_f
  \qquad
  \forall f\in\FM{k+1}(\Mh)
\end{equation}
is non-empty, and its unique element $\widehat\lambda=(\widehat\lambda_e)_{e\in\FM{k}(\Mh)}$ of minimum weighted norm
\[
  \|\widehat\lambda\|_{k,h}^2
  \coloneqq
  \sum_{e\in\FM{k}(\Mh)}h_e^{n-2k}|\widehat\lambda_e|^2
\]
depends linearly on $\xi$ and satisfies
\begin{equation}\label{eq:face.cell.poincare.estimate}
  \sum_{e\in\FM{k}(\Mh)}h_e^{n-2k}|\widehat\lambda_e|^2
  \leq
  (C_{\mathrm{coch}}^k)^2
  \sum_{f\in\FM{k+1}(\Mh)}h_f^{n-2k-2}|\xi_f|^2.
\end{equation}
Here $C_{\mathrm{coch}}^k>0$ may depend on the fixed domain $\Omega$, on $n$ and $k$, and on the standing mesh-regularity parameters, but is uniform along the regular mesh family.
\end{lemma}
\begin{proof}
Assumption \eqref{eq:face.cell.coboundary.assumption} is precisely the solvability condition for \eqref{eq:face.cell.coboundary.identity}.
The Poincar\'e inequality for cochains in \cite[Lemma~6, Eq.~(21)]{Di-Pietro.Droniou.ea:25}, applied to the cellular mesh $\Mh$, provides a solution $\lambda^*$ of \eqref{eq:face.cell.coboundary.identity} satisfying the corresponding estimate written as sums over the top-dimensional cell patches.
Since the local mesh sizes of incident cells are uniformly comparable and the number of cells in each local patch is uniformly bounded, these patchwise sums are uniformly equivalent to the global weighted sums appearing in \eqref{eq:face.cell.poincare.estimate}.
Notice that, for $k=0$, we use the positive vertex scales introduced in Subsection~\ref{sec:setting:mesh}, which are uniformly comparable with the local mesh sizes of the incident cells and therefore yield a non-degenerate weighted norm equivalent to the patchwise norm.
Consequently,
\[
\sum_{e\in\FM{k}(\Mh)}h_e^{n-2k}|\lambda^*_e|^2
\lesssim
\sum_{f\in\FM{k+1}(\Mh)}h_f^{n-2k-2}|\xi_f|^2.
\]

Since $\widehat\lambda$ has minimum weighted norm among all the solutions of \eqref{eq:face.cell.coboundary.identity}, \eqref{eq:face.cell.poincare.estimate} follows.
Finally, $\widehat\lambda$ is the unique solution lying in the weighted orthogonal complement of the kernel of the cellular coboundary, and therefore depends linearly on $\xi$.
\end{proof}

\begin{lemma}[Local boundary value problem in finite element spaces]
\label{lem:first.local.problem}
Let $f\in\FM{d}(\Mh)$ be a $d$-cell, with $d\ge1$, let $k\in\{0,\ldots,d-1\}$ be a form degree, and let $r\ge1$ be a polynomial degree.
Let
\[
  \xi\in\PLtrim{r}{k+1}(\Sh(f))
  \qquad\text{and}\qquad
  \theta\in\PLtrim{r+1}{k}(\Sh(\partial f))
\]
be such that $\DIFF\xi=0$ and the following compatibility condition holds.
  \begin{align*}
    \text{if $d=k+1$},&
    \qquad
    \int_f \xi = \int_{\partial f}\theta,\\
    \text{if $d\ge k+2$},&
    \qquad
    \trc_{\partial f}\xi = \DIFF_{\partial f}\theta.
  \end{align*}
Then there exists $\lambda\in\PLtrim{r+1}{k}(\Sh(f))$ such that
\[
  \DIFF\lambda=\xi,
  \qquad
  \trc_{\partial f}\lambda=\theta,
\]
and
\[
  \norm{f}{\lambda}^2
  \lesssim
  h_f\norm{\partial f}{\theta}^2
  +
  h_f^2\norm{f}{\xi}^2.
\]
\end{lemma}
\begin{proof}
This is \cite[Lemma~11]{Di-Pietro.Droniou.Pitassi:25}, written in the notation of the present paper.
\end{proof}

\begin{lemma}[Local potential problem without prescribed trace]
\label{lem:local.potential.problem}
Let $f\in\FM{d}(\Mh)$, let $k\in\{0,\ldots,d-1\}$, and let $r\geq1$.
For every $\xi\in\DIFF_f\PLtrim{r}{k}(\Sh(f))$, let $\lambda\in\PLtrim{r}{k}(\Sh(f))$ be the unique solution of $\DIFF_f\lambda=\xi$ of minimum $L^2(f)$ norm.
Then $\lambda$ depends linearly on $\xi$ and satisfies
\[
  \|\lambda\|_f
  \lesssim
  h_f\|\xi\|_f.
\]
\end{lemma}
\begin{proof}
The range assumption ensures solvability.
The auxiliary simplicial submeshes have uniformly bounded cardinality, so only finitely many combinatorial types occur.
Fix a reference realisation of each type.
The simplexwise affine pullback identifies the conforming trimmed polynomial spaces with fixed finite-dimensional reference spaces and commutes with the exterior derivative.
On each reference space, the exterior derivative admits a bounded right inverse on its image.
Uniform shape regularity, comparability of the simplex diameters with $h_f$, and scaling therefore give a solution $\lambda^*\in\PLtrim{r}{k}(\Sh(f))$ such that
\[
  \DIFF_f\lambda^*=\xi,
  \qquad
  \|\lambda^*\|_f\lesssim h_f\|\xi\|_f.
\]
The minimum-norm solution has no larger norm.
It is the unique solution in the $L^2(f)$-orthogonal complement of $\Ker(\DIFF_f|_{\PLtrim{r}{k}(\Sh(f))})$, and hence depends linearly on $\xi$.
\end{proof}

\subsubsection{Polynomial bubble moment realisations}
\label{subsubsec:trace.and.bubble}

The following result relies on the full--trimmed polynomial duality pairing on a simplex; see \cite[Corollary~3.3]{Berchenko-Kogan:21}.
\begin{lemma}[Polynomial bubble moment realisation]
\label{lem:polynomial.bubble.moment.realisation}
Let $f\in\FM{d}(\Mh)$ be a $d$-cell, let $k\in\{0,\ldots,d\}$, and let
\[
\mathcal U_f
\subset
\bigtimes_{S\in\Sh(f)}\PL{r}{d-k}(S)
\]
be a finite-dimensional space of piecewise polynomial forms, endowed with the $L^2(f)$ norm.
For a polynomial degree $r_{\mathrm b}$ depending only on $r,d,k$ and on the fixed mesh-regularity parameters, every $\Phi_f\in\mathcal U_f'$ admits a representative $\rho_f\in\PLtrimz{r_{\mathrm b}}{k}(\Sh(f))$ such that
\[
\trc_{\partial f}\rho_f=0,
\qquad
\int_f\rho_f\wedge\mu
=
\Phi_f(\mu)
\qquad
\forall\mu\in\mathcal U_f,
\]
and
\[
\|\rho_f\|_{L^2\Lambda^k(f)}
\lesssim
\|\Phi_f\|_{\mathcal U_f'}.
\]
The representative can be selected linearly in $\Phi_f$.
\end{lemma}

\begin{proof}
Set $\mathcal W_f \coloneqq \bigtimes_{S\in\Sh(f)} \PL{r}{d-k}(S)$, with its broken $L^2$ norm.
On every $d$-simplex $S$, the full--trimmed duality \cite[Corollaries~3.3--3.4]{Berchenko-Kogan:21} shows that, for a sufficiently large degree $r_{\mathrm b}$ depending only on $r$, $d$, and $k$, the wedge pairing
\[
\PLtrimz{r_{\mathrm b}}{k}(S)\times\PL{r}{d-k}(S)
\longrightarrow
\Real,
\qquad
(\alpha,\eta)
\mapsto
\int_S\alpha\wedge\eta,
\]
is non-degenerate in the second factor.
On a reference simplex, the representative of minimum $L^2$ norm therefore defines a bounded right inverse of
\[
\PLtrimz{r_{\mathrm b}}{k}(S)\ni\alpha
\mapsto
\left(
    \eta\mapsto\int_S\alpha\wedge\eta
\right)
\in
\left(\PL{r}{d-k}(S)\right)'.
\]
Transport to the uniformly shape-regular simplices and scaling yield a bound independent of $S$.
Taking the direct sum over the simplices yields a uniformly bounded realisation map from $\mathcal W_f'$ to the direct sum of simplex bubble spaces.

Extend $\Phi_f$ to $\widehat\Phi_f\in\mathcal W_f'$ with minimum dual norm.
The Hilbert-space extension theorem gives $\|\widehat\Phi_f\|_{\mathcal W_f'} = \|\Phi_f\|_{\mathcal U_f'}$.
Applying the simplexwise realisation map produces forms $\rho_S\in\PLtrimz{r_{\mathrm b}}{k}(S)$.
Define $\rho_f$ by $(\rho_f)_{|S}=\rho_S$.
Since every $\rho_S$ has zero trace on $\partial S$, the pieces are conforming across the internal interfaces and vanish on $\partial f$.
Thus $\rho_f\in\PLtrimz{r_{\mathrm b}}{k}(\Sh(f))$.

For every $\mu\in\mathcal U_f$, $\int_f\rho_f\wedge\mu = \widehat\Phi_f(\mu) = \Phi_f(\mu)$.
The stability of the simplexwise right inverses, together with the uniformly bounded number of simplices in $\Sh(f)$, then yields
\[
\|\rho_f\|_{L^2\Lambda^k(f)}
\lesssim
\|\widehat\Phi_f\|_{\mathcal W_f'}
=
\|\Phi_f\|_{\mathcal U_f'}.
\]
Choosing both the minimum-norm extension and the simplexwise minimum-norm representatives makes the construction linear.
\end{proof}

\subsection{Recursive construction of the stable compatible skeleton}
\label{subsec:recursive.skeleton}

\subsubsection{Descending construction of moment functionals}
\label{subsubsec:downward}

For the remainder of the recursive skeleton construction, we assume $k\in\{0,\ldots,n-2\}$.
The top-degree case $k=n-1$ does not require the downward recursive construction and is treated directly in Lemma~\ref{lem:top.degree.face.skeleton}.
We choose a polynomial degree $m_{\mathcal S}\geq m+2$, independent of $h$, sufficiently large that $m_{\mathcal S}-1$ dominates all polynomial degrees required by Lemma~\ref{lem:stable.closed.moment.realisation} over the finitely many cell dimensions up to $d=n$.
Throughout the ascending construction, all closed polynomial $(k+1)$-forms $\xi_f$ are constructed in $\PLtrim{m_{\mathcal S}-1}{k+1}(\Sh(f))$, whereas all polynomial $k$-form potentials $\lambda_f$ belong to $\PLtrim{m_{\mathcal S}}{k}(\Sh(f))$.

The vertical differential $\mathbb D$ and the commuting-diagram identities used below are defined in Proposition~\ref{prop:functional.complex}.
The descending construction follows the standard zig-zag through this double complex.
At each level a vertical differential produces a horizontal cocycle, which is then inverted by a stable coface-link solve.

\paragraph{Base case.}

Set $q_* \coloneqq n-k-2 = q_{n-1}$.

For every $T\in\mathcal T_h$, define
\[
  \Gamma_T^n(\zeta_T)
  \coloneqq
  (-1)^{k+1}
  \int_T
  \widehat\sigma_T\wedge\DIFF_T\zeta_T,
  \qquad
  \zeta_T\in
  \PLtrim{m}{q_*}(\Sh(T)),
\]
and set $\Gamma^n \coloneqq (\Gamma_T^n)_{T\in\mathcal T_h} \in \mathbb X_n^{n-1}$.

We also introduce the datum $\Upsilon^{n-1}\in\mathbb X_{n-1}^{n-1}$ defined by
\[
  \Upsilon_F^{n-1}(\nu_F)
  \coloneqq
  (-1)^{k+1}
  \int_F
  v_F\wedge\DIFF_F\nu_F,
  \qquad
  \nu_F\in
  \PLtrim{m}{q_*}(\Sh(F)).
\]

\begin{lemma}[Stable top link solve]
\label{lem:top.link.solve}
There exists a relatively closed family $\Xi^{n-1}\in\mathbb X_{n-1}^{n-1}$ such that
\begin{equation}\label{eq:top.full.system}
  \mathbb B_n^{n-1}\Xi^{n-1}
  =
  \Gamma^n.
\end{equation}
The family can be selected linearly in $\Gamma^n$ and satisfies
\begin{equation}\label{eq:top.link.stability}
  \|\Xi^{n-1}\|_{\mathbb X_{n-1}^{n-1}}^2
  \lesssim
  \sum_{T\in\mathcal T_h}
  \|\widehat\sigma_T\|_T^2
  \lesssim
  |\underline v_h|_{\DIFF,h}^2.
\end{equation}

Moreover, for every $T\in\mathcal T_h$ and every $\zeta_T\in\PLtrim{m}{q_*}(\Sh(T))$ such that $\DIFF_T\zeta_T\in\star_T\Dpoly{T}{k}$, one has
\begin{equation}\label{eq:top.reconstruction.restriction}
  \sum_{F\in\mathcal F_T}
  \epsTF
  \Xi_F^{n-1}(\trc_F\zeta_T)
  =
  (-1)^{k+1}
  \int_T
  \sigma_T\wedge\DIFF_T\zeta_T.
\end{equation}

In addition, if $q_*>0$, then
\begin{equation}\label{eq:top.vertical.compatibility}
  \mathbb B_n^{n-2}
  \mathbb D_{n-1}^{n-2}\Xi^{n-1}
  =
  0,
\end{equation}
whereas, if $q_*=0$,
\begin{equation}\label{eq:top.augmentation}
  \left(
    \mathbb B_n^{n-1}\Xi^{n-1}
  \right)_T(1)
  =
  0
  \qquad
  \forall T\in\mathcal T_h.
\end{equation}
\end{lemma}

\begin{proof}
Fix $T\in\mathcal T_h$ and let $\zeta_T\in\PLtrim{m}{q_*}(\Sh(T))$.
Since $\DIFF_T\zeta_T \in \PLtrim{m}{q_*+1}(\Sh(T)) = \star_T\widehat{\mathcal R}_T^{k+1}$, and $\DIFF_T^2\zeta_T=0$, the defining equation of the auxiliary reconstruction gives
\[
  \int_T
  \widehat\sigma_T\wedge\DIFF_T\zeta_T
  =
  \sum_{F\in\mathcal F_T}
  \epsTF
  \int_F
  v_F\wedge\DIFF_F\trc_F\zeta_T.
\]
Multiplying by $(-1)^{k+1}$ yields
\begin{equation}\label{eq:auxiliary.top.range.identity}
  \Gamma^n
  =
  \mathbb B_n^{n-1}\Upsilon^{n-1}.
\end{equation}

Applying the geometric coordinate map to \eqref{eq:auxiliary.top.range.identity} shows that every coordinate of $\mathcal G_n^{n-1}\Gamma^n$ belongs to the corresponding range of $\delta_{n,s}$.

If $\zeta_T\in\PLtrimz{m}{q_*}(\Sh(T))$, all its traces vanish.
Hence $\Gamma_T^n(\zeta_T)=0$, and $\Gamma^n$ is trace-supported.

Assume that $q_*>0$.
Let $F\in\mathcal F_T$ and set $\eta_F \coloneqq \DIFF_F\chi_F$, $\chi_F \in \PLtrimz{m}{q_*-1}(\Sh(F))$.
The trace property of the geometric extension and \eqref{eq:auxiliary.top.range.identity} give
\[
\begin{aligned}
  \Gamma_T^n
  \left(
    \mathcal E_{F,T}^{n-1}\eta_F
  \right)
  &=
  \epsTF
  \Upsilon_F^{n-1}(\eta_F)
  \\
  &=
  (-1)^{k+1}
  \epsTF
  \int_F
  v_F\wedge\DIFF_F^2\chi_F
  =
  0.
\end{aligned}
\]
This proves \eqref{eq:proper.exact.datum.vanishing}.

For every $\chi_F\in\PLtrim{m}{q_*-1}(\Sh(F))$, one also has $\Upsilon_F^{n-1}(\DIFF_F\chi_F) = (-1)^{k+1} \int_F v_F\wedge\DIFF_F^2\chi_F = 0$.
Therefore,
\begin{equation}\label{eq:auxiliary.certificate.vertical.closedness}
  \mathbb D_{n-1}^{n-2}
  \Upsilon^{n-1}
  =
  0.
\end{equation}

If $q_*=0$, condition \eqref{eq:proper.exact.datum.vanishing} is void, while $\DIFF_F1=0$ gives
\begin{equation}\label{eq:auxiliary.certificate.augmentation}
  \Upsilon_F^{n-1}(1)
  =
  0
  \qquad
  \forall F\in\mathcal F_h.
\end{equation}

Lemma~\ref{lem:full.link.solve} now gives a relatively closed solution of \eqref{eq:top.full.system}, selected linearly in $\Gamma^n$.

It remains to prove the stability estimate.
Set $\eta \coloneqq \mathcal G_n^{n-1}\Gamma^n$ and $\lambda \coloneqq \mathcal G_{n-1}^{n-1}\Xi^{n-1}$.
For every $s\in\mathfrak S_h^{n-1}$ with $\dim s\leq n-1$, let $\eta_s$ and $\lambda_s$ denote the corresponding geometric coordinates and set $a\coloneqq\dim s$.
By the selection made in Lemma~\ref{lem:full.link.solve}, $\lambda_s$ is the unique solution of $\delta_{n,s}\lambda_s=\eta_s$ of minimum norm $\|\cdot\|_{\mathbb C_{n-1,s}^{n-1}}$.
From the definition of $\Gamma^n$, the stability of the geometric extensions, and the inverse estimate in the fixed-degree trimmed space, we obtain
\[
\|(\eta_s)_T\|_{
\left(
\PLtrimz{m}{q_*}(\Sh(s))
\right)'
}
\lesssim
h_T^{(n-a)/2-1}
\|\widehat\sigma_T\|_T
\qquad
\forall
T\in\cof{n}(s).
\]
The weighted coface Poincar\'e estimate with $d=n$ yields
\[
\|\lambda_s\|_{\mathbb C_{n-1,s}^{n-1}}^2
\lesssim
\sum_{T\in\cof{n}(s)}
  \|\widehat\sigma_T\|_T^2.
\]
Summing over the cells $s$, using the uniformly bounded number of subcells of each mesh cell, and applying the coordinate norm equivalence gives
\[
\|\Xi^{n-1}\|_{\mathbb X_{n-1}^{n-1}}^2
\lesssim
\sum_{T\in\mathcal T_h}
\|\widehat\sigma_T\|_T^2.
\]
The second estimate in \eqref{eq:top.link.stability} follows from Proposition~\ref{prop:auxiliary.projection.control}.

Let $T\in\mathcal T_h$ and let
\[
\zeta_T
\in
\PLtrim{m}{q_*}(\Sh(T)),
\qquad
\mu_T
\coloneqq
\DIFF_T\zeta_T
\in
\star_T\Dpoly{T}{k}.
\]
Write $\mu_T=\star_Tq_T$ with $q_T\in\Dpoly{T}{k}$.
By \eqref{eq:auxiliary.projection}, $\int_T \widehat\sigma_T\wedge\mu_T = \int_T \sigma_T\wedge\mu_T$.
Using \eqref{eq:top.full.system} and the definition of $\Gamma^n$, we therefore obtain
\[
\begin{aligned}
  \sum_{F\in\mathcal F_T}
  \epsTF
  \Xi_F^{n-1}(\trc_F\zeta_T)
  &=
  \Gamma_T^n(\zeta_T)
  \\
  &=
  (-1)^{k+1}
  \int_T
  \widehat\sigma_T\wedge\DIFF_T\zeta_T
  \\
  &=
  (-1)^{k+1}
  \int_T
  \sigma_T\wedge\DIFF_T\zeta_T,
\end{aligned}
\]
which proves \eqref{eq:top.reconstruction.restriction}.

If $q_*>0$, the identity of the double complex gives
\[
\begin{aligned}
  \mathbb B_n^{n-2}
  \mathbb D_{n-1}^{n-2}\Xi^{n-1}
  &=
  \mathbb D_n^{n-2}
  \mathbb B_n^{n-1}\Xi^{n-1}
  \\
  &=
  \mathbb D_n^{n-2}\Gamma^n.
\end{aligned}
\]
For every admissible $\chi_T$,
\[
  \left(
    \mathbb D_n^{n-2}\Gamma^n
  \right)_T(\chi_T)
  =
  \Gamma_T^n(\DIFF_T\chi_T)
  =
  0,
\]
since $\DIFF_T^2\chi_T=0$.
This proves \eqref{eq:top.vertical.compatibility}.

If $q_*=0$, then $\Gamma_T^n(1) = (-1)^{k+1} \int_T \widehat\sigma_T\wedge\DIFF_T1 = 0$.
Together with \eqref{eq:top.full.system}, this proves \eqref{eq:top.augmentation}.
\end{proof}

\paragraph{Recursive link descent.}
Let $d\in\{n-1,\ldots,k+2\}$ and let $\Xi^d\in\mathbb X_d^d$ be the relatively closed family obtained at the preceding step of the recursion, satisfying $\mathbb B_{d+1}^{d-1}\mathbb D_d^{d-1}\Xi^d=0$.
Define $\Gamma^d \coloneqq (-1)^{k+1} \mathbb D_d^{d-1}\Xi^d \in\mathbb X_d^{d-1}$.

\begin{lemma}[Stable local link descent]
\label{lem:local.link.descent}
The datum $\Gamma^d$ is trace-supported and satisfies \eqref{eq:proper.exact.datum.vanishing} with $j=d-1$, together with
\[
  \mathbb B_{d+1}^{d-1}\Gamma^d=0.
\]
Consequently, there exists a relatively closed $\Xi^{d-1}\in\mathbb X_{d-1}^{d-1}$ such that
\[
  \mathbb B_d^{d-1}\Xi^{d-1}=\Gamma^d.
\]
It can be selected linearly in $\Gamma^d$ and satisfies
\begin{equation}\label{eq:local.link.descent.stability}
  \|\Xi^{d-1}\|_{\mathbb X_{d-1}^{d-1}}^2
  \lesssim
  \|\Xi^d\|_{\mathbb X_d^d}^2.
\end{equation}
\end{lemma}
\begin{proof}
If $\nu_f\in\PLtrimz{m}{d-k-2}(\Sh(f))$, relative closedness of $\Xi_f^d$ gives $\Gamma_f^d(\nu_f) = (-1)^{k+1}\Xi_f^d(\DIFF_f\nu_f) =0$.
Hence $\Gamma^d$ is trace-supported.

If $q_{d-1}>0$, let $e\in\FM{d-1}(\partial f)$ and let $\eta_e=\DIFF_e\chi_e\in\DIFF_e\PLtrimz{m}{d-k-3}(\Sh(e))$.
Set $\nu_f=\mathcal E_{e,f}^{d-1}\eta_e$.
The form $\DIFF_f\nu_f$ has zero trace.
Its trace on $e$ is $\DIFF_e\eta_e=0$, and its trace on every other facet is zero.
Since $q_d=d-k-1<d$, the relative trimmed complex on the ball $f$ is exact in degree $q_d$ \cite{Arnold.Falk.ea:06,Arnold.Falk.ea:09}, and therefore there exists $\nu_f^\circ\in\PLtrimz{m}{d-k-2}(\Sh(f))$ such that $\DIFF_f\nu_f^\circ=\DIFF_f\nu_f$.
Therefore $\Gamma_f^d(\nu_f) = (-1)^{k+1} \Xi_f^d(\DIFF_f\nu_f^\circ) =0$.
This is precisely the proper-exact datum vanishing condition at level $d-1$.
When $q_{d-1}=0$, this condition is void.

For the first step, the cocycle identity is \eqref{eq:top.vertical.compatibility}.
At later steps, the commutation identity of the double complex and the equation already satisfied by $\Xi^d$ give
\[
\begin{aligned}
\mathbb B_{d+1}^{d-1}\Gamma^d
&={}
(-1)^{k+1}
\mathbb D_{d+1}^{d-1}
\mathbb B_{d+1}^{d}\Xi^d
=0,
\end{aligned}
\]
where the last equality follows from the vertical complex identity $\mathbb D\mathbb D=0$.
Since $d<n$, local link exactness gives the required componentwise range conditions.
Lemma~\ref{lem:full.link.solve} therefore gives the relatively closed solution.

Moreover, the polynomial inverse inequality gives
\[
  \|\Gamma_f^d\|_{
    \left(
      \PLtrim{m}{d-k-2}(\Sh(f))
    \right)'
  }
  \lesssim
  h_f^{-1}
  \|\Xi_f^d\|_{
    \left(
      \PLtrim{m}{d-k-1}(\Sh(f))
    \right)'
  }
  \qquad
  \forall f\in\FM{d}(\Mh).
\]
Hence \eqref{eq:full.link.stability}, applied with $j=d-1$, yields
\[
\begin{aligned}
  \|\Xi^{d-1}\|_{\mathbb X_{d-1}^{d-1}}^2
  &\lesssim
  \sum_{f\in\FM{d}(\Mh)}
  h_f^{n-d+2}
  \|\Gamma_f^d\|_{
    \left(
      \PLtrim{m}{d-k-2}(\Sh(f))
    \right)'
  }^2
  \\
  &\lesssim
  \sum_{f\in\FM{d}(\Mh)}
  h_f^{n-d}
  \|\Xi_f^d\|_{
    \left(
      \PLtrim{m}{d-k-1}(\Sh(f))
    \right)'
  }^2
  \\
  &=
  \|\Xi^d\|_{\mathbb X_d^d}^2.
\end{aligned}
\]
This proves \eqref{eq:local.link.descent.stability}.
\end{proof}

Starting from \eqref{eq:top.link.stability} and iterating \eqref{eq:local.link.descent.stability} over the finitely many levels, we obtain relatively closed functionals $\Xi^d\in\mathbb X_d^d$ for $d\in \{k+1,\ldots,n-1\}$, satisfying
\begin{equation}\label{eq:full.downward.identity}
  \mathbb B_d^{d-1}\Xi^{d-1}
  =
  (-1)^{k+1}\mathbb D_d^{d-1}\Xi^d,
  \qquad
  d \in \{k+2,\ldots,n-1\},
\end{equation}
and
\[
  \sum_{d=k+1}^{n-1}
  \|\Xi^d\|_{\mathbb X_d^d}^2
  \lesssim
  \sum_{T\in\mathcal T_h}
  \|\widehat\sigma_T\|_T^2
  \lesssim
  |\underline v_h|_{\DIFF,h}^2.
\]

\subsubsection{Terminal cochain and global cochain solve}
\label{subsubsec:terminal.cochain}

Since $q_{k+1}=0$, evaluation on the constant function defines the terminal $(k+1)$-cochain
\[
  \bar\xi
  \coloneqq
  (\bar\xi_f)_{f\in\FM{k+1}(\Mh)},
  \qquad
  \bar\xi_f
  \coloneqq
  \Xi_f^{k+1}(1).
\]
For $k+2<n$, evaluating \eqref{eq:full.downward.identity} on the constant function gives, for every $g\in\FM{k+2}(\Mh)$,
\[
\begin{aligned}
  \sum_{f\in\FM{k+1}(\partial g)}
  \epsilon_{g,f}\bar\xi_f
  &=
  \left(
    \mathbb B_{k+2}^{k+1}\Xi^{k+1}
  \right)_g(1)
  \\
  &=
  (-1)^{k+1}
  \left(
    \mathbb D_{k+2}^{k+1}\Xi^{k+2}
  \right)_g(1)
  \\
  &=
  (-1)^{k+1}\Xi_g^{k+2}(\DIFF_g1)
  =
  0.
\end{aligned}
\]
When $k+2=n$, the same cocycle identity follows from \eqref{eq:top.augmentation}.

For every $f\in\FM{k+1}(\Mh)$, choose the minimum-norm top-degree form $\xi_f\in\PLtrim{m_{\mathcal S}-1}{k+1}(\Sh(f))$ such that
\[
  \int_f\xi_f\wedge v_f
  =
  \Xi_f^{k+1}(v_f)
  \qquad
  \forall v_f\in\PLtrim{m}{0}(\Sh(f)).
\]
This is a finite-dimensional local mass-matrix problem.
Fixed-degree norm equivalence shows that the minimum-norm selection is linear and satisfies
\begin{equation}\label{eq:terminal.realisation.stability}
  \|\xi_f\|_f
  \lesssim
  \|\Xi_f^{k+1}\|_{\left(\PLtrim{m}{0}(\Sh(f))\right)'}.
\end{equation}
Since $\xi_f$ has top degree, it is closed, and $\int_f\xi_f=\bar\xi_f$.
Moreover, Cauchy--Schwarz and scaling give
\begin{equation}\label{eq:terminal.cochain.datum.stability}
  |\bar\xi_f|
  =
  |\Xi_f^{k+1}(1)|
  \leq
  \|\Xi_f^{k+1}\|_{\left(\PLtrim{m}{0}(\Sh(f))\right)'}
  \vvvert1\vvvert_f
  \lesssim
  h_f^{(k+1)/2}
  \|\Xi_f^{k+1}\|_{\left(\PLtrim{m}{0}(\Sh(f))\right)'}.
\end{equation}

\begin{lemma}[Invariance of the terminal cochain modulo coboundaries]
\label{lem:terminal.class.invariance}
Let $\Xi^d,\widetilde\Xi^d\in\mathbb X_d^d$ for $d\in\{k+1,\ldots,n-1\}$ be two recursive families associated with the same top datum, in the sense that
\begin{equation}\label{eq:same.top.datum}
  \mathbb B_n^{n-1}\Xi^{n-1}
  =
  \mathbb B_n^{n-1}\widetilde\Xi^{n-1},
\end{equation}
and
\begin{equation}\label{eq:two.recursive.families}
\begin{aligned}
  \mathbb B_d^{d-1}\Xi^{d-1}
  &=
  (-1)^{k+1}
  \mathbb D_d^{d-1}\Xi^d,
  \\
  \mathbb B_d^{d-1}\widetilde\Xi^{d-1}
  &=
  (-1)^{k+1}
  \mathbb D_d^{d-1}\widetilde\Xi^d
\end{aligned}
\end{equation}
for every $d\in\{k+2,\ldots,n-1\}$.

Define the associated terminal cochains by
\[
  \bar\xi_f
  \coloneqq
  \Xi_f^{k+1}(1),
  \qquad
  \widetilde{\bar\xi}_f
  \coloneqq
  \widetilde\Xi_f^{k+1}(1).
\]
Then there exists a $k$-cochain $\vartheta=(\vartheta_e)_{e\in\FM{k}(\Mh)}$ such that
\[
  \bar\xi_f-\widetilde{\bar\xi}_f
  =
  \sum_{e\in\FM{k}(\partial f)}
  \epsilon_{f,e}\vartheta_e
  \qquad
  \forall f\in\FM{k+1}(\Mh).
\]
\end{lemma}
\begin{proof}
Set $\Delta\Xi^d\coloneqq\Xi^d-\widetilde\Xi^d$.
By \eqref{eq:same.top.datum}, $\mathbb B_n^{n-1}\Delta\Xi^{n-1}=0$.
Horizontal exactness below the top dimension, given by Corollary~\ref{cor:functional.horizontal.exactness}, therefore provides $\Phi^{n-2} \in \mathbb X_{n-2}^{n-1}$ such that
\begin{equation}\label{eq:recursion.top.difference}
  \Delta\Xi^{n-1}
  =
  \mathbb B_{n-1}^{n-1}\Phi^{n-2}.
\end{equation}

If $k\leq n-3$, we now descend recursively.
Combining \eqref{eq:two.recursive.families}, \eqref{eq:recursion.top.difference}, and the commuting identity between $\mathbb B$ and $\mathbb D$, we obtain
\[
  \mathbb B_{n-1}^{n-2}
  \left(
    \Delta\Xi^{n-2}
    -
    (-1)^{k+1}
    \mathbb D_{n-2}^{n-2}\Phi^{n-2}
  \right)
  =
  0.
\]
Horizontal exactness therefore provides $\Phi^{n-3}\in\mathbb X_{n-3}^{n-2}$ such that
\[
  \Delta\Xi^{n-2}
  =
  (-1)^{k+1}\mathbb D_{n-2}^{n-2}\Phi^{n-2}
  +
  \mathbb B_{n-2}^{n-2}\Phi^{n-3}.
\]

Iterating this argument yields
\begin{equation}\label{eq:recursion.terminal.difference}
  \Delta\Xi^{k+1}
  =
  (-1)^{k+1}
  \mathbb D_{k+1}^{k+1}\Phi^{k+1}
  +
  \mathbb B_{k+1}^{k+1}\Phi^k
\end{equation}
when $k\leq n-3$.
When $k=n-2$, the same argument directly gives $\Delta\Xi^{k+1}=\mathbb B_{k+1}^{k+1}\Phi^k$.

When $k\leq n-3$, evaluating \eqref{eq:recursion.terminal.difference} on the constant function and using $\DIFF_f1=0$ gives
\[
  \bar\xi_f-\widetilde{\bar\xi}_f
  =
  \sum_{e\in\FM{k}(\partial f)}
  \epsilon_{f,e}\Phi_e^k(1)
  \qquad
  \forall f\in\FM{k+1}(\Mh).
\]
When $k=n-2$, the same identity follows directly from $\Delta\Xi^{k+1}=\mathbb B_{k+1}^{k+1}\Phi^k$.
The result follows by setting $\vartheta_e\coloneqq\Phi_e^k(1)$ for every $e\in\FM{k}(\Mh)$.
\end{proof}

\begin{corollary}[Exactness of the terminal cochain]
\label{cor:terminal.cochain.exactness}
For $f\in\FM{k+1}(\Mh)$, set
\[
  \bar\xi_f
  \coloneqq
  \Xi_f^{k+1}(1).
\]
Then there exists a $k$-cochain $\lambda=(\lambda_e)_{e\in\FM{k}(\Mh)}$ such that
\[
\sum_{e\in\FM{k}(\partial f)}
\epsilon_{f,e}\lambda_e
=\bar\xi_f
\qquad
  \forall f\in\FM{k+1}(\Mh).
\]
\end{corollary}
\begin{proof}
Consider the reference family
\[
  \widetilde\Xi^{n-1}
  \coloneqq
  \Upsilon^{n-1},
  \qquad
  \widetilde\Xi^d
  \coloneqq
  0
  \quad
  \forall d\in\{k+1,\ldots,n-2\}.
\]
By \eqref{eq:auxiliary.top.range.identity}, $\mathbb B_n^{n-1}\widetilde\Xi^{n-1} = \mathbb B_n^{n-1}\Upsilon^{n-1} = \Gamma^n$.
Moreover, if $k\leq n-3$, then $q_*>0$, and \eqref{eq:auxiliary.certificate.vertical.closedness} gives $\mathbb D_{n-1}^{n-2} \widetilde\Xi^{n-1} = \mathbb D_{n-1}^{n-2} \Upsilon^{n-1} = 0$.
Hence the family $\widetilde\Xi$ satisfies the required recursion identities.

Its terminal cochain is zero.
Indeed, if $k\leq n-3$, then $\widetilde\Xi^{k+1}=0$, whereas, if $k=n-2$, \eqref{eq:auxiliary.certificate.augmentation} gives $\widetilde\Xi_F^{k+1}(1) = \Upsilon_F^{n-1}(1) = 0$.
Lemma~\ref{lem:terminal.class.invariance} therefore shows that the terminal cochain associated with the selected recursive family is a cellular coboundary.
\end{proof}

Corollary~\ref{cor:terminal.cochain.exactness} shows that $\bar\xi$ is a cellular coboundary.
The comparison with the reference family proves this exactness without assuming that $\Omega$ is contractible.
Lemma~\ref{lem:cochain.poincare} therefore gives a stable minimum-norm $k$-cochain $\widehat\lambda=(\widehat\lambda_e)_{e\in\FM{k}(\Mh)}$ satisfying
\begin{equation}\label{eq:global.cochain.problem}
  \bar\xi_f
  =
  \sum_{e\in\FM{k}(\partial f)}
  \epsilon_{f,e}\widehat\lambda_e
  \qquad
  \forall f\in\FM{k+1}(\Mh).
\end{equation}
By \eqref{eq:face.cell.poincare.estimate} and \eqref{eq:terminal.cochain.datum.stability},
\begin{equation*}
\begin{aligned}
  \sum_{e\in\FM{k}(\Mh)}
  h_e^{n-2k}|\widehat\lambda_e|^2
  &\lesssim
  \sum_{f\in\FM{k+1}(\Mh)}
  h_f^{n-2k-2}|\bar\xi_f|^2
  \\
  &\lesssim
  \sum_{f\in\FM{k+1}(\Mh)}
  h_f^{n-k-1}
  \|\Xi_f^{k+1}\|_{\left(\PLtrim{m}{0}(\Sh(f))\right)'}^2
  =
  \|\Xi^{k+1}\|_{\mathbb X_{k+1}^{k+1}}^2.
\end{aligned}
\end{equation*}

The solution of the global cochain problem \eqref{eq:global.cochain.problem} is the only global solve in the construction.
Choose on every $k$-cell a fixed normalised top form
\[
  \omega_e\in\PLtrim{m_{\mathcal S}}{k}(\Sh(e)),
  \qquad
  \int_e\omega_e=1,
  \qquad
  \|\omega_e\|_e^2\lesssim h_e^{-k},
\]
and set $\lambda_e\coloneqq\widehat\lambda_e\omega_e$.
Then
\begin{equation}\label{eq:terminal.form.stability}
  \sum_{e\in\FM{k}(\Mh)}
  h_e^{n-k}\|\lambda_e\|_e^2
  \lesssim
  \sum_{e\in\FM{k}(\Mh)}
  h_e^{n-2k}|\widehat\lambda_e|^2
  \lesssim
  \|\Xi^{k+1}\|_{\mathbb X_{k+1}^{k+1}}^2.
\end{equation}
Moreover,
\begin{equation}\label{eq:terminal.stokes.compatibility}
  \int_f\xi_f
  =
  \sum_{e\in\FM{k}(\partial f)}
  \epsilon_{f,e}\int_e\lambda_e.
\end{equation}
For every $(k+1)$-cell $f$, define $\lambda_f$ as the minimum-norm solution in $\PLtrim{m_{\mathcal S}}{k}(\Sh(f))$ of
\[
  \DIFF_f\lambda_f=\xi_f,
  \qquad
  \trc_e\lambda_f=\lambda_e
  \quad
  \forall e\in\FM{k}(\partial f).
\]
The compatibility condition required by Lemma~\ref{lem:first.local.problem} is \eqref{eq:terminal.stokes.compatibility}.
Applying that lemma with polynomial parameter $m_{\mathcal S}-1$ and form degree $k$ therefore shows that this local problem is solvable.
Since $\DIFF_f\lambda_f=\xi_f$, its stability estimate gives
\begin{equation}\label{eq:first.upward.stability}
  \|\lambda_f\|_f^2
  +
  h_f^2\|\DIFF_f\lambda_f\|_f^2
  \lesssim
  h_f\|\lambda_{\partial f}\|_{\partial f}^2
  +
  h_f^2\|\xi_f\|_f^2.
\end{equation}

\subsubsection{Local closed-moment realisation}
\label{subsubsec:closed.moment.realisation}

The following result combines the polynomial bubble moment realisation in Lemma~\ref{lem:polynomial.bubble.moment.realisation} with the exactness of the trimmed polynomial complex and a stable local right inverse of the exterior derivative.

\begin{lemma}[Stable closed zero-trace moment realisation]
\label{lem:stable.closed.moment.realisation}
Let $f\in\FM{d}(\Mh)$ and let $q\in\{1,\ldots,d-1\}$.
Let
\[
  \Lambda_f
  \in
  \left(
    \PLtrim{m}{q}(\Sh(f))
  \right)'
\]
satisfy
\begin{equation}\label{eq:closed.moment.absolute.compatibility}
  \Lambda_f(\DIFF_f\nu_f)=0
  \qquad
  \forall
  \nu_f\in
  \PLtrim{m}{q-1}(\Sh(f)).
\end{equation}
For a sufficiently large but fixed polynomial degree $m_{\mathrm{cl}}$, there exists
\[
  \xi_f
  \in
  \PLtrimz{m_{\mathrm{cl}}}{d-q}(\Sh(f))
  \cap
  \Ker\DIFF_f
\]
such that
\begin{equation}\label{eq:closed.moment.zero.trace.realisation}
  \int_f
  \xi_f\wedge\alpha_f
  =
  \Lambda_f(\alpha_f)
  \qquad
  \forall
  \alpha_f\in
  \PLtrim{m}{q}(\Sh(f)).
\end{equation}
The unique representative of minimum $L^2(f)$ norm depends linearly on $\Lambda_f$ and satisfies
\[
  \|\xi_f\|_f
  \lesssim
  \|\Lambda_f\|_{
    \left(
      \PLtrim{m}{q}(\Sh(f))
    \right)'
  }.
\]

The degree $m_{\mathrm{cl}}$ depends only on $d$, $q$, $m$, and the standing regularity parameters.
\end{lemma}
\begin{proof}
Set $p\coloneqq d-q$, so that $1\leq p\leq d-1$.

Standard exactness of the absolute trimmed complex on the contractible simplicial complex $\Sh(f)$ \cite{Arnold.Falk.ea:06,Arnold.Falk.ea:09} identifies the kernel of $\DIFF_f$ with the image from the preceding degree.
Condition \eqref{eq:closed.moment.absolute.compatibility} therefore shows that $\Lambda_f$ vanishes on this kernel and hence factors through the image of $\DIFF_f$.
More precisely, define a functional on $\DIFF_f\PLtrim{m}{q}(\Sh(f))$ by
\[
\Phi_f(\beta_f)
\coloneqq
(-1)^p\Lambda_f(\alpha_f),
\]
where $\alpha_f\in\PLtrim{m}{q}(\Sh(f))$ is any form satisfying $\DIFF_f\alpha_f=\beta_f$.
The definition is independent of the selected preimage.
By Lemma~\ref{lem:local.potential.problem}, a preimage can be chosen so that $\|\alpha_f\|_f\lesssim h_f\|\beta_f\|_f$, and consequently
\[
|\Phi_f(\beta_f)|
\lesssim
h_f
\|\Lambda_f\|_{
\left(
\PLtrim{m}{q}(\Sh(f))
\right)'
}
\|\beta_f\|_f.
\]

The range $\DIFF_f\PLtrim{m}{q}(\Sh(f))$ consists of piecewise polynomial $(q+1)$-forms.
Lemma~\ref{lem:polynomial.bubble.moment.realisation}, applied with form degree $p-1$, therefore provides, for a sufficiently large but fixed degree $m_{\mathrm b}$, a form $\rho_f\in\PLtrimz{m_{\mathrm b}}{p-1}(\Sh(f))$ such that
\[
\int_f\rho_f\wedge\beta_f
=
\Phi_f(\beta_f)
\qquad
\forall
\beta_f\in
\DIFF_f\PLtrim{m}{q}(\Sh(f)),
\]
and
\begin{equation}\label{eq:closed.moment.bubble.stability}
\|\rho_f\|_f
\lesssim
h_f
\|\Lambda_f\|_{
\left(
\PLtrim{m}{q}(\Sh(f))
\right)'
}.
\end{equation}

Set $\xi_f\coloneqq\DIFF_f\rho_f$.
Then $\DIFF_f\xi_f=0$, and the zero trace of $\rho_f$ gives $\trc_{\partial f}\xi_f = \DIFF_{\partial f}\trc_{\partial f}\rho_f = 0$.
After increasing $m_{\mathrm{cl}}$ if necessary, it follows that $\xi_f \in \PLtrimz{m_{\mathrm{cl}}}{p}(\Sh(f)) \cap \Ker\DIFF_f$.
The polynomial inverse inequality and \eqref{eq:closed.moment.bubble.stability} give
\[
  \|\xi_f\|_f
  =
  \|\DIFF_f\rho_f\|_f
  \lesssim
  h_f^{-1}\|\rho_f\|_f
  \lesssim
  \|\Lambda_f\|_{
    \left(
      \PLtrim{m}{q}(\Sh(f))
    \right)'
  },
\]
which proves the stability estimate.

It remains to verify the prescribed moments.
Let $\alpha_f\in\PLtrim{m}{q}(\Sh(f))$.
Since $\rho_f$ has zero trace, Stokes' formula and the definition of the induced functional give
\[
\begin{aligned}
  \int_f\xi_f\wedge\alpha_f
  &=
  \int_f\DIFF_f\rho_f\wedge\alpha_f
  \\
  &=
  (-1)^p
  \int_f\rho_f\wedge\DIFF_f\alpha_f
  \\
  &=
  (-1)^p
  \Phi_f\left(\DIFF_f\alpha_f\right)
  \\
  &=
  \Lambda_f(\alpha_f).
\end{aligned}
\]
This proves \eqref{eq:closed.moment.zero.trace.realisation}.

Finally, among the forms in $\PLtrimz{m_{\mathrm{cl}}}{p}(\Sh(f))\cap\Ker\DIFF_f$ satisfying the prescribed moments, choose the unique one of minimum $L^2(f)$ norm.
Its norm is no larger than that of the representative constructed above.
Moreover, it is the unique solution lying in the $L^2(f)$-orthogonal complement of the kernel of the corresponding moment map and therefore depends linearly on $\Lambda_f$.
\end{proof}

\subsubsection{Ascending construction of polynomial forms}
\label{subsubsec:upward}

\begin{lemma}[Completion with prescribed trace]
\label{lem:closed.moment.completion}
Let $d\in\{k+2,\ldots,n\}$ and $f\in\FM{d}(\Mh)$.
Suppose that a conforming family of closed forms $\xi_{\partial f}=(\xi_e)_{e\in\FM{d-1}(\partial f)}$ has already been constructed, with $\xi_e\in\PLtrim{m_{\mathcal S}-1}{k+1}(\Sh(e))$ for every $e\in\FM{d-1}(\partial f)$.
Let $\Xi_f^d\in(\PLtrim{m}{d-k-1}(\Sh(f)))'$ satisfy
\begin{equation}\label{eq:closed.moment.compatibility}
\Xi_f^d(\DIFF_f\chi_f)
=(-1)^{k+1}
\sum_{e\in\FM{d-1}(\partial f)}
\epsilon_{f,e}
\int_e\xi_e\wedge\trc_e\chi_f
\end{equation}
for every $\chi_f\in\PLtrim{m}{d-k-2}(\Sh(f))$.
Then the set of forms $\xi_f\in \PLtrim{m_{\mathcal S}-1}{k+1}(\Sh(f))\cap\Ker\DIFF_f$ such that
\[
  \trc_{\partial f}\xi_f=\xi_{\partial f}
\]
and
\[
  \int_f\xi_f\wedge\nu_f
  =
  \Xi_f^d(\nu_f)
  \qquad
  \forall\nu_f\in\PLtrim{m}{d-k-1}(\Sh(f))
\]
is non-empty.
Its unique element of minimum $L^2(f)$ norm depends linearly on the data and satisfies
\begin{equation}\label{eq:closed.moment.completion.stability}
  \|\xi_f\|_f^2
  \lesssim
  \|\Xi_f^d\|_{\left(\PLtrim{m}{d-k-1}(\Sh(f))\right)'}^2
  +
  h_f\|\xi_{\partial f}\|_{\partial f}^2.
\end{equation}
\end{lemma}
\begin{proof}
The boundary datum is closed and conforming.
If $d=k+2$, its terminal compatibility follows by taking $\chi_f=1$ in \eqref{eq:closed.moment.compatibility}.
For $d\geq k+3$, the compatibility condition required by Lemma~\ref{lem:first.local.problem} follows from the closedness and conformity of the already constructed boundary forms.
Lemma~\ref{lem:first.local.problem}, applied with polynomial parameter $m_{\mathcal S}-2$, form degree $k+1$, and zero derivative datum, therefore gives a stable closed extension $\widetilde\xi_f \in \PLtrim{m_{\mathcal S}-1}{k+1}(\Sh(f))$ of $\xi_{\partial f}$.
We select the solution of minimum $L^2(f)$ norm, so that it depends linearly on the boundary datum, and
\[
  \|\widetilde\xi_f\|_f
  \lesssim
  h_f^{1/2}\|\xi_{\partial f}\|_{\partial f}.
\]
Define $\mathfrak R_f(\nu_f) \coloneqq \Xi_f^d(\nu_f) - \int_f\widetilde\xi_f\wedge\nu_f$.
Then
\[
  \|\mathfrak R_f\|_{\left(\PLtrim{m}{d-k-1}(\Sh(f))\right)'}
  \lesssim
  \|\Xi_f^d\|_{\left(\PLtrim{m}{d-k-1}(\Sh(f))\right)'}
  +
  h_f^{1/2}\|\xi_{\partial f}\|_{\partial f}.
\]
Stokes' formula and \eqref{eq:closed.moment.compatibility} give $\mathfrak R_f(\DIFF_f\chi_f)=0$ for every $\chi_f\in\PLtrim{m}{d-k-2}(\Sh(f))$.
Lemma~\ref{lem:stable.closed.moment.realisation} provides a closed zero-trace correction $\eta_f$ realising $\mathfrak R_f$ and satisfying
\[
  \|\eta_f\|_f
  \lesssim
  \|\Xi_f^d\|_{\left(\PLtrim{m}{d-k-1}(\Sh(f))\right)'}
  +
  h_f^{1/2}\|\xi_{\partial f}\|_{\partial f}.
\]
By the choice of $m_{\mathcal S}$, $m_{\mathrm{cl}}\leq m_{\mathcal S}-1$, and hence $\eta_f \in \PLtrimz{m_{\mathrm{cl}}}{k+1}(\Sh(f)) \subset \PLtrimz{m_{\mathcal S}-1}{k+1}(\Sh(f))$.
Thus $\widetilde\xi_f+\eta_f$ has all the required properties and satisfies the estimate in \eqref{eq:closed.moment.completion.stability}.
Among the admissible forms, choose the unique one of minimum $L^2(f)$ norm.
It depends linearly on the data and has norm no larger than that of $\widetilde\xi_f+\eta_f$, which completes the proof.
\end{proof}

\paragraph{Inductive upward step.}
Let $d\in\{k+2,\ldots,n-1\}$ and suppose that the forms $\lambda_e$ have been constructed on every $(d-1)$-cell.
Set
\begin{equation}\label{eq:boundary.xi.definition}
  \xi_e\coloneqq\DIFF_e\lambda_e.
\end{equation}
The forms $\xi_e$ are single-valued on common subcells and realise the functionals $\Xi_e^{d-1}$.
From \eqref{eq:full.downward.identity}, for every $\chi_f\in\PLtrim{m}{d-k-2}(\Sh(f))$,
\[
\begin{aligned}
  \Xi_f^d(\DIFF_f\chi_f)
  &={}
  (-1)^{k+1}
  (\mathbb B_d^{d-1}\Xi^{d-1})_f(\chi_f)
  \\
  &={}
  (-1)^{k+1}
  \sum_{e\in\FM{d-1}(\partial f)}
  \epsilon_{f,e}
  \int_e\xi_e\wedge\trc_e\chi_f.
\end{aligned}
\]
Lemma~\ref{lem:closed.moment.completion} constructs a closed form $\xi_f$ satisfying
\begin{equation}\label{eq:upward.xi.properties}
  \trc_{\partial f}\xi_f=\xi_{\partial f},
  \qquad
  \int_f\xi_f\wedge\nu_f
  =
  \Xi_f^d(\nu_f)
  \quad
  \forall\nu_f\in\PLtrim{m}{d-k-1}(\Sh(f)).
\end{equation}
We then define $\lambda_f$ as the minimum-norm solution in $\PLtrim{m_{\mathcal S}}{k}(\Sh(f))$ of
\[
  \DIFF_f\lambda_f=\xi_f,
  \qquad
  \trc_e\lambda_f=\lambda_e
  \quad
  \forall e\in\FM{d-1}(\partial f).
\]
The local compatibility follows from \eqref{eq:upward.xi.properties} and \eqref{eq:boundary.xi.definition}.
Hence Lemma~\ref{lem:first.local.problem}, applied with polynomial parameter $m_{\mathcal S}-1$ and form degree $k$, shows that this local problem is solvable independently on every cell.
Since $\DIFF_f\lambda_f=\xi_f$, its stability estimate gives
\begin{equation}\label{eq:upward.lambda.stability}
  \|\lambda_f\|_f^2
  +
  h_f^2\|\DIFF_f\lambda_f\|_f^2
  \lesssim
  h_f\|\lambda_{\partial f}\|_{\partial f}^2
  +
  h_f^2\|\xi_f\|_f^2.
\end{equation}

After the final step, the face forms satisfy
\begin{equation}\label{eq:final.face.moments}
  \DIFF_F\lambda_F=\xi_F,
  \qquad
  \int_F\xi_F\wedge\nu_F
  =
  \Xi_F^{n-1}(\nu_F)
  \quad
  \forall\nu_F\in\PLtrim{m}{n-k-2}(\Sh(F)).
\end{equation}

For every $F\in\mathcal F_h$ and every $\zeta_F\in\PLtrim{m}{n-k-2}(\Sh(F))$, Stokes' formula gives
\begin{equation}\label{eq:face.stokes}
\begin{aligned}
  \int_F
  \lambda_F\wedge\DIFF_F\zeta_F
  ={}&
  (-1)^k
  \sum_{E\in\FM{n-2}(\partial F)}
  \epsilon_{F,E}
  \int_E
  \lambda_E\wedge\trc_E\zeta_F
  \\
  &+
  (-1)^{k+1}
  \int_F
  \xi_F\wedge\zeta_F.
\end{aligned}
\end{equation}

\subsubsection{Construction of the stable compatible face skeleton}
\label{subsubsec:skeleton.verification}

\begin{lemma}[Stable top-degree face skeleton]
\label{lem:top.degree.face.skeleton}
Let $k=n-1$, and let $\sigma_T$ and $\widehat\sigma_T$ be the prescribed and auxiliary reconstructed differentials defined in \eqref{eq:prescribed.differential.data} and \eqref{eq:auxiliary.differential.data}, respectively.
For a sufficiently large but fixed polynomial degree $m_{\mathcal S}\geq m+2$, there exists a family depending linearly on $\underline v_h$,
\[
  \lambda^{n-1}
  =
  (\lambda_F)_{F\in\mathcal F_h},
  \qquad
  \lambda_F
  \in
  \PLtrim{m_{\mathcal S}}{n-1}(\Sh(F)),
\]
such that
\[
  \sum_{F\in\mathcal F_T}
  \epsTF
  \int_F\lambda_F
  =
  \int_T\widehat\sigma_T
  =
  \int_T\sigma_T
  \qquad
  \forall T\in\mathcal T_h,
\]
and
\[
  \sum_{F\in\mathcal F_h}
  h_F\|\lambda_F\|_F^2
  \lesssim
  |\underline v_h|_{\DIFF,h}^2.
\]
Moreover, the family is compatible on the mesh face skeleton.
\end{lemma}

\begin{proof}
Set $\bar\xi_T \coloneqq \int_T\widehat\sigma_T$.
Testing the defining equation of the auxiliary reconstruction against $\star_T^{-1}1\in\widehat{\mathcal R}_T^n$ gives
\[
  \bar\xi_T
  =
  \sum_{F\in\mathcal F_T}
  \epsTF
  \int_Fv_F.
\]
Thus $\bar\xi=(\bar\xi_T)_{T\in\mathcal T_h}$ belongs to the image of the cellular coboundary.
Lemma~\ref{lem:cochain.poincare}, applied at cochain degree $n-1$, provides a linear $(n-1)$-cochain $\widehat\lambda=(\widehat\lambda_F)_{F\in\mathcal F_h}$ such that
\[
  \bar\xi_T
  =
  \sum_{F\in\mathcal F_T}
  \epsTF\widehat\lambda_F
\]
and
\[
  \sum_{F\in\mathcal F_h}
  h_F^{2-n}|\widehat\lambda_F|^2
  \lesssim
  \sum_{T\in\mathcal T_h}
  h_T^{-n}|\bar\xi_T|^2
  \lesssim
  \sum_{T\in\mathcal T_h}
  \|\widehat\sigma_T\|_T^2
  \lesssim
  |\underline v_h|_{\DIFF,h}^2.
\]

For every $F\in\mathcal F_h$, choose a fixed top-degree form $\omega_F \in \PLtrim{m_{\mathcal S}}{n-1}(\Sh(F))$ such that $\int_F\omega_F=1$, $\|\omega_F\|_F^2 \lesssim h_F^{1-n}$, and set $\lambda_F \coloneqq \widehat\lambda_F\omega_F$.
Since $\lambda_F$ has top degree on $F$, its trace on $\partial F$ vanishes, so the resulting skeleton family is compatible.
Moreover,
\[
  \sum_{F\in\mathcal F_h}
  h_F\|\lambda_F\|_F^2
  \lesssim
  \sum_{F\in\mathcal F_h}
  h_F^{2-n}|\widehat\lambda_F|^2
  \lesssim
  |\underline v_h|_{\DIFF,h}^2.
\]
Finally,
\[
  \sum_{F\in\mathcal F_T}
  \epsTF
  \int_F\lambda_F
  =
  \sum_{F\in\mathcal F_T}
  \epsTF\widehat\lambda_F
  =
  \bar\xi_T
  =
  \int_T\widehat\sigma_T.
\]
Since the top-degree reconstruction space contains $\star_T^{-1}1\in\Dpoly{T}{n-1}$, \eqref{eq:auxiliary.projection} also gives $\int_T\widehat\sigma_T = \int_T\sigma_T$, which completes the proof.
\end{proof}

The following lemma is the central structural result of the proof.
It collects the output of the descent through the double complex, the terminal cellular cochain solve, and the upward polynomial realisation into a stable compatible polynomial face skeleton.
The final conforming lifting will be obtained from this skeleton by cellwise operations only.

\begin{lemma}[Stable compatible polynomial face skeleton]
\label{lem:stable.face.skeleton}
For every $\underline v_h\in\uXh{k}$, let $\sigma_T$ and $\widehat\sigma_T$ be defined by \eqref{eq:prescribed.differential.data} and \eqref{eq:auxiliary.differential.data}, respectively.
There exists a family depending linearly on $\underline v_h$,
\[
  \lambda^{n-1}
  =
  (\lambda_F)_{F\in\mathcal F_h},
  \qquad
  \lambda_F
  \in
  \PLtrim{m_{\mathcal S}}{k}(\Sh(F)),
\]
for a sufficiently large but fixed polynomial degree $m_{\mathcal S}\geq m+2$, such that the boundary family $\lambda_{\partial T}\coloneqq(\lambda_F)_{F\in\mathcal F_T}$ is compatible on $\partial T$ for every $T\in\mathcal T_h$, and
\begin{equation}\label{eq:auxiliary.skeleton.identity}
  \sum_{F\in\mathcal F_T}
  \epsTF
  \int_F
  \lambda_F\wedge\trc_F\widehat\mu_T
  =
  \int_T
  \widehat\sigma_T\wedge\widehat\mu_T
  \qquad
  \forall T\in\mathcal T_h,\quad
  \forall\widehat\mu_T\in\widehat{\mathcal Z}_T^k.
\end{equation}
Consequently,
\begin{equation}\label{eq:prescribed.skeleton.identity}
  \sum_{F\in\mathcal F_T}
  \epsTF
  \int_F
  \lambda_F\wedge\trc_F\mu_T
  =
  \int_T
  \sigma_T\wedge\mu_T
  \qquad
  \forall T\in\mathcal T_h,\quad
  \forall\mu_T\in\mathcal Z_T^k.
\end{equation}
For $k\in\{0,\ldots,n-2\}$, the face differentials satisfy $\DIFF_F\lambda_F\in\PLtrim{m_{\mathcal S}-1}{k+1}(\Sh(F))$ for every $F\in\mathcal F_h$.
Moreover,
\begin{equation}\label{eq:face.skeleton.stability}
  \sum_{F\in\mathcal F_h}
  h_F\bigl(\|\lambda_F\|_F^2+\|\DIFF_F\lambda_F\|_F^2\bigr)
  \lesssim
  |\underline v_h|_{\DIFF,h}^2.
\end{equation}
\end{lemma}

\begin{proof}
If $k=n-1$, then both $\widehat{\mathcal Z}_T^{n-1}$ and $\mathcal Z_T^{n-1}$ consist of constant $0$-forms.
The two identities and the stability estimate therefore follow from Lemma~\ref{lem:top.degree.face.skeleton}, since $\DIFF_F\lambda_F=0$ on every face.
We henceforth assume $k\in\{0,\ldots,n-2\}$.
The upward construction provides polynomial forms $\lambda^d=(\lambda_f)_{f\in\FM{d}(\Mh)}$ for $d\in\{k,\ldots,n-1\}$ and $\xi^d=(\xi_f)_{f\in\FM{d}(\Mh)}$ for $d\in\{k+1,\ldots,n-1\}$, with $\xi_f=\DIFF_f\lambda_f\in\PLtrim{m_{\mathcal S}-1}{k+1}(\Sh(f))$, such that
\[
  \trc_e\lambda_f
  =
  \lambda_e
  \qquad
  \forall
  f\in\FM{d}(\Mh),\quad
  \forall
  e\in\FM{d-1}(\partial f).
\]
In particular, the restriction $\lambda^{n-1} = (\lambda_F)_{F\in\mathcal F_h}$ is a compatible face-skeleton family.

Let $T\in\mathcal T_h$ and let $\widehat\mu_T\in\widehat{\mathcal Z}_T^k$.
Since $\widehat\mu_T \in \PLtrim{m}{n-k-1}(\Sh(T)) \cap \Ker\DIFF_T$, the exactness of the trimmed complex on the contractible simplicial submesh $\Sh(T)$ \cite{Arnold.Falk.ea:06,Arnold.Falk.ea:09} gives $\zeta_T \in \PLtrim{m}{n-k-2}(\Sh(T))$ such that $\DIFF_T\zeta_T = \widehat\mu_T$.
This potential is used only as a test in the argument below, so no particular linear or stable selection is required.

Summing \eqref{eq:face.stokes} over the faces of $T$, the terms on the $(n-2)$-skeleton cancel by the conformity of the skeleton family and the boundary identity $\partial^2=0$.
Using \eqref{eq:final.face.moments}, \eqref{eq:top.full.system}, and the definition of $\Gamma^n$, we obtain
\[
\begin{aligned}
  \sum_{F\in\mathcal F_T}
  \epsTF
  \int_F
  \lambda_F\wedge\trc_F\widehat\mu_T
  &=
  (-1)^{k+1}
  \sum_{F\in\mathcal F_T}
  \epsTF
  \int_F
  \xi_F\wedge\trc_F\zeta_T
  \\
  &=
  (-1)^{k+1}
  \sum_{F\in\mathcal F_T}
  \epsTF
  \Xi_F^{n-1}(\trc_F\zeta_T)
  \\
  &=
  (-1)^{k+1}\Gamma_T^n(\zeta_T)
  \\
  &=
  \int_T
  \widehat\sigma_T\wedge\DIFF_T\zeta_T
  \\
  &=
  \int_T
  \widehat\sigma_T\wedge\widehat\mu_T.
\end{aligned}
\]
This proves \eqref{eq:auxiliary.skeleton.identity}.

Now let $\mu_T\in\mathcal Z_T^k$.
Since $\mathcal Z_T^k\subseteq\widehat{\mathcal Z}_T^k$ and $\mu_T=\star_Tq_T$ for some $q_T\in\Dpoly{T}{k}$, \eqref{eq:auxiliary.projection} gives $\int_T \widehat\sigma_T\wedge\mu_T = \int_T \sigma_T\wedge\mu_T$.
Restricting \eqref{eq:auxiliary.skeleton.identity} to these tests proves \eqref{eq:prescribed.skeleton.identity}.

All selections entering the construction of $\lambda^{n-1}$ are linear.
The top and intermediate systems use uniformly stable link inverses, the terminal step uses one global cochain Poincar\'e solve, and the upward sweep uses local moment completions and local first-order problems.
More precisely, \eqref{eq:terminal.realisation.stability}, \eqref{eq:terminal.form.stability}, \eqref{eq:first.upward.stability}, \eqref{eq:closed.moment.completion.stability}, and \eqref{eq:upward.lambda.stability}, after multiplication by the corresponding cell weights, give, for every $d\in\{k+2,\ldots,n-1\}$,
\[
\begin{aligned}
  \sum_{f\in\FM{d}(\Mh)}
  h_f^{n-d}\|\xi_f\|_f^2
  &\lesssim
  \sum_{f\in\FM{d}(\Mh)}
  h_f^{n-d}
  \|\Xi_f^d\|_{\left(\PLtrim{m}{d-k-1}(\Sh(f))\right)'}^2
  \\
  &\qquad+
  \sum_{e\in\FM{d-1}(\Mh)}
  h_e^{n-d+1}\|\xi_e\|_e^2,
\end{aligned}
\]
and, for every $d\in\{k+1,\ldots,n-1\}$,
\[
  \sum_{f\in\FM{d}(\Mh)}
  h_f^{n-d}\|\lambda_f\|_f^2
  \lesssim
  \sum_{e\in\FM{d-1}(\Mh)}
  h_e^{n-d+1}\|\lambda_e\|_e^2
  +
  \sum_{f\in\FM{d}(\Mh)}
  h_f^{n-d+2}\|\xi_f\|_f^2.
\]
Here, we used the uniformly bounded number of cells incident on each subcell and the uniform comparability of incident mesh sizes.
The estimate \eqref{eq:terminal.realisation.stability} starts the first recursion at $d=k+1$ and gives
\[
  \sum_{d=k+1}^{n-1}
  \sum_{f\in\FM{d}(\Mh)}
  h_f^{n-d}\|\xi_f\|_f^2
  \lesssim
  \sum_{d=k+1}^{n-1}\|\Xi^d\|_{\mathbb X_d^d}^2
  \lesssim
  \sum_{T\in\mathcal T_h}\|\widehat\sigma_T\|_T^2.
\]
Iterating the estimates for $\lambda_f$ from \eqref{eq:terminal.form.stability}, using $h_f^2\leq\operatorname{diam}(\Omega)^2$ to bound the additional factors multiplying $\|\xi_f\|_f^2$, therefore gives
\[
\begin{aligned}
  &
  \sum_{d=k}^{n-1}
  \sum_{f\in\FM{d}(\Mh)}
  h_f^{n-d}\|\lambda_f\|_f^2
  \\
  &\qquad+
  \sum_{d=k+1}^{n-1}
  \sum_{f\in\FM{d}(\Mh)}
  h_f^{n-d}\|\xi_f\|_f^2
  \\
  &\qquad\lesssim
  \sum_{T\in\mathcal T_h}
  \|\widehat\sigma_T\|_T^2
  \lesssim
  |\underline v_h|_{\DIFF,h}^2.
\end{aligned}
\]
Taking the contributions corresponding to $d=n-1$ in both sums gives
\[
  \sum_{F\in\mathcal F_h}
  h_F\bigl(\|\lambda_F\|_F^2+\|\DIFF_F\lambda_F\|_F^2\bigr)
  \lesssim
  |\underline v_h|_{\DIFF,h}^2,
\]
which proves \eqref{eq:face.skeleton.stability}.
\end{proof}

In particular, \eqref{eq:global.cochain.problem} is the only non-local solve in the construction.
The face-skeleton family depends linearly on the full hybrid datum through the auxiliary reconstruction $(\widehat\sigma_T)_{T\in\mathcal T_h}$.
It satisfies the full trimmed identity \eqref{eq:auxiliary.skeleton.identity}, whose restriction to the prescribed test space gives \eqref{eq:prescribed.skeleton.identity} with $\sigma_h$.

\subsection{Conforming completion and proof of the main theorem}
\label{subsec:hybrid.completion}

The stable face skeleton is now completed in the mesh cells.
The cell forms are defined independently by local moment completions and local first-order problems with prescribed trace.

\stableconforminglifting*

\begin{proof}[Proof of Theorem~\ref{thm:main.hybrid.poincare}]
For the given $\underline v_h$, let $\sigma_T$ and $\widehat\sigma_T$ be defined by \eqref{eq:prescribed.differential.data} and \eqref{eq:auxiliary.differential.data}, respectively.
Let $(\lambda_F)_{F\in\mathcal F_h}$ be the family given by Lemma~\ref{lem:stable.face.skeleton}.

Assume first that $k\in\{0,\ldots,n-2\}$ and set $\xi_F\coloneqq\DIFF_F\lambda_F$ on every face.
The forms $\xi_F\in\PLtrim{m_{\mathcal S}-1}{k+1}(\Sh(F))$ are closed and have compatible traces because the family $(\lambda_F)_{F\in\mathcal F_h}$ is compatible.
For every $T\in\mathcal T_h$, define the local functional
\[
  \Xi_T^n(\nu_T)
  \coloneqq
  \int_T\widehat\sigma_T\wedge\nu_T
  \qquad
  \forall\nu_T\in\widehat{\mathcal M}_T^k.
\]
Let $\chi_T\in\PLtrim{m}{n-k-2}(\Sh(T))$.
Since $\DIFF_T\chi_T\in\widehat{\mathcal Z}_T^k$, identity \eqref{eq:auxiliary.skeleton.identity} and Stokes' formula on the faces give
\begin{equation}\label{eq:conforming.completion.compatibility}
\begin{aligned}
  \Xi_T^n(\DIFF_T\chi_T)
  &=
  \sum_{F\in\mathcal F_T}
  \epsTF
  \int_F\lambda_F\wedge\DIFF_F\trc_F\chi_T
  \\
  &=
  (-1)^{k+1}
  \sum_{F\in\mathcal F_T}
  \epsTF
  \int_F\xi_F\wedge\trc_F\chi_T.
\end{aligned}
\end{equation}
The contributions on the $(n-2)$-skeleton cancel by conformity and the boundary identity $\partial^2=0$.
Thus \eqref{eq:conforming.completion.compatibility} is the compatibility condition of Lemma~\ref{lem:closed.moment.completion} at $d=n$.
For $k=n-2$, taking $\chi_T=1$ also gives the required zero integral of the boundary datum.

Applying that lemma on every cell gives a closed form $\xi_T\in\PLtrim{m_{\mathcal S}-1}{k+1}(\Sh(T))$ such that
\begin{equation}\label{eq:conforming.completion.differential}
  \trc_F\xi_T=\xi_F
  \quad\forall F\in\mathcal F_T,
  \qquad
  \int_T\xi_T\wedge\nu_T
  =
  \int_T\widehat\sigma_T\wedge\nu_T
  \quad\forall\nu_T\in\widehat{\mathcal M}_T^k.
\end{equation}
Since $\|\Xi_T^n\|_{(\widehat{\mathcal M}_T^k)'}\leq\|\widehat\sigma_T\|_T$, its stability estimate yields
\begin{equation}\label{eq:conforming.completion.differential.bound}
  \|\xi_T\|_T^2
  \lesssim
  \|\widehat\sigma_T\|_T^2
  +
  h_T\|\DIFF_{\partial T}\lambda_{\partial T}\|_{\partial T}^2.
\end{equation}
The choice of minimum norm makes $\xi_T$ linear in the data.

If $k=n-1$, set $\xi_T\coloneqq\widehat\sigma_T$.
This form has top degree and is therefore closed.
The choice $m_{\mathcal S}\geq m+2$ ensures that $\xi_T\in\PLtrim{m_{\mathcal S}-1}{n}(\Sh(T))$.
The moment identity and the bound in \eqref{eq:conforming.completion.differential}--\eqref{eq:conforming.completion.differential.bound} hold directly, with zero boundary differential.
Moreover, Lemma~\ref{lem:top.degree.face.skeleton} gives
\[
  \int_T\xi_T
  =
  \sum_{F\in\mathcal F_T}\epsTF\int_F\lambda_F.
\]

For every $k\in\{0,\ldots,n-1\}$, the compatibility conditions of Lemma~\ref{lem:first.local.problem} are therefore satisfied with derivative datum $\xi_T$ and boundary datum $\lambda_{\partial T}$.
We define $\lambda_T$ as the minimum-norm solution in $\PLtrim{m_{\mathcal S}}{k}(\Sh(T))$ of
\begin{equation}\label{eq:conforming.completion.local.problem}
  \DIFF_T\lambda_T=\xi_T,
  \qquad
  \trc_F\lambda_T=\lambda_F
  \quad\forall F\in\mathcal F_T.
\end{equation}
That lemma, applied with polynomial parameter $m_{\mathcal S}-1$, gives
\[
  \|\lambda_T\|_T^2
  \lesssim
  h_T\|\lambda_{\partial T}\|_{\partial T}^2
  +
  h_T^2\|\xi_T\|_T^2.
\]
The selection depends linearly on the data.

Define $\lambda_h$ by $\left.\lambda_h\right|_T=\lambda_T$ for every $T\in\mathcal T_h$.
The trace identities in \eqref{eq:conforming.completion.local.problem} show that $\lambda_h\in \PLtrim{m_{\mathcal S}}{k}(\Sh)$, and $(\DIFF\lambda_h)_{|T}=\xi_T$.
For every $q_T\in\Dpoly{T}{k}$, the moment identity in \eqref{eq:conforming.completion.differential} and the projection identity \eqref{eq:auxiliary.projection} give
\[
  \int_T\xi_T\wedge\star_Tq_T
  =
  \int_T\widehat\sigma_T\wedge\star_Tq_T
  =
  \int_T\sigma_T\wedge\star_Tq_T.
\]
Hence $\piD{T}{k}\xi_T=\sigma_T$, which proves \eqref{eq:main.lifting.projected.derivative}.

Summing the local estimates, using the uniform comparability of incident mesh sizes and the uniformly bounded number of cells incident on each face, gives
\[
\begin{aligned}
  \|\lambda_h\|_\Omega^2+\|\DIFF\lambda_h\|_\Omega^2
  &\lesssim
  \sum_{T\in\mathcal T_h}\|\widehat\sigma_T\|_T^2
  \\
  &\quad+
  \sum_{F\in\mathcal F_h}
  h_F\bigl(\|\lambda_F\|_F^2+\|\DIFF_F\lambda_F\|_F^2\bigr)
  \\
  &\lesssim
  |\underline v_h|_{\DIFF,h}^2.
\end{aligned}
\]
Here the factors $h_T^2$ are bounded by $\operatorname{diam}(\Omega)^2$ and absorbed in the hidden constant, and the last inequality follows from \eqref{eq:auxiliary.graph.control} and \eqref{eq:face.skeleton.stability}.
This proves \eqref{eq:main.lifting.stability}.
The construction is linear because the stable skeleton and all local minimum-norm solutions are selected linearly.
\end{proof}

\begin{remark}[Scalar moment preservation and conforming smoothers]
\label{rem:scalar.moments.ern.zanotti}
For $k=0$, the lifting constructed above also preserves the hybrid moments up to a constant on each connected component of $\Omega$.
Assume first that $\Omega$ is connected, set $\lambda_h=\mathcal L_h^0\underline v_h$, and define $c_h\coloneqq |\Omega|^{-1}\int_\Omega(\lambda_h-v_h)$ and $E_h\underline v_h\coloneqq\lambda_h-c_h$.
Then
\[
  \uI{0}{h}E_h\underline v_h=\underline v_h,
  \qquad
  \|\DIFF E_h\underline v_h\|_\Omega
  \lesssim |\underline v_h|_{\DIFF,h}.
\]
Indeed, the full moment identity \eqref{eq:conforming.completion.differential}, the auxiliary Stokes formula, and integration by parts give
\[
  \int_T(\lambda_h-v_T)\DIFF_T\nu_T
  =
  \sum_{F\in\mathcal F_T}\epsTF
  \int_F(\trc_F\lambda_h-v_F)\trc_F\nu_T
\]
for every $\nu_T\in\PLtrim{m}{n-1}(\Sh(T))$.
Recall that $m\geq\ell+1$.
Testing first with forms of zero trace and using relative exactness shows that the projection of $\lambda_h-v_T$ onto the broken polynomials of degree $m-1$ on $\Sh(T)$ is a constant $c_T$.
Surjectivity of the trace on the auxiliary boundary simplices then gives $\piQ{F}{0}(\trc_F\lambda_h-v_F)=c_T$ on every face of $T$.
Since constants belong to every admissible face space, conformity across interior faces implies that the constants $c_T$ coincide.
This proves the interpolation identity, and the stability estimate follows because subtracting $c_h$ does not change the exterior derivative.
For a disconnected domain, the same argument applies on each connected component.

For the usual scalar HHO choice $\Qtrace{F}{0}=\PL{\ell}{0}(F)$, these identities give the moment preservation used by Ern and Zanotti in their analysis of HHO methods with $H^{-1}$ loads \cite[Lemma~4.4 and Proposition~4.5]{Ern.Zanotti:20:smoother}.
Their smoother takes values in $H_0^1(\Omega)$, whereas the present construction gives an $H^1(\Omega)$ function.
Its use in that Dirichlet formulation would therefore require an additional completion that enforces zero boundary trace while preserving the moments and stability.
\end{remark}

\subsection{Hybrid Poincar\'e inequality and uniformity}
\label{subsec:skeleton.interpretation.uniformity}

\begin{corollary}[Uniform discrete Poincar\'e inequality in the canonical hybrid seminorm]
\label{cor:hybrid.poincare}
Let $k\in\{0,\ldots,n-1\}$ and $\ell\geq0$, and assume that the selected face spaces satisfy the stability-admissibility condition \eqref{eq:def.Qtrace.admissible}.
For every $\underline v_h\in\uXh{k}$, there exists $\underline w_h\in\uXh{k}$, depending linearly on $\underline v_h$, such that
\begin{equation}\label{eq:main.poincare.same.derivative}
  \DIFF_{\ell,h}^{k}\underline w_h
  =
  \DIFF_{\ell,h}^{k}\underline v_h,
\end{equation}
and
\begin{equation}\label{eq:main.poincare.representative.estimate}
  \|\underline w_h\|_{0,h}
  \leq
  C_{\mathrm{hyb}}^{k,\ell}
  |\underline v_h|_{\DIFF,h}.
\end{equation}
Consequently,
\begin{equation}\label{eq:main.poincare.quotient}
  \inf_{\underline\phi_h\in\Ker\DIFF_{\ell,h}^{k}}
  \|\underline v_h+\underline\phi_h\|_{0,h}
  \leq
  C_{\mathrm{hyb}}^{k,\ell}
  |\underline v_h|_{\DIFF,h}.
\end{equation}
The constant $C_{\mathrm{hyb}}^{k,\ell}$ has the same dependencies as $C_{\mathrm{lift}}^{k,\ell}$ and is likewise independent of the mesh size and of the admissible choice of the face spaces.
\end{corollary}

\begin{proof}
Let $\lambda_h=\mathcal L_h^k\underline v_h$ be the form given by Theorem~\ref{thm:main.hybrid.poincare}, and set $\underline w_h\coloneqq\uI{k}{h}\lambda_h$.
Since $\lambda_h$ is conforming and piecewise polynomial, its face traces are well defined in $L^2$ and Stokes' formula applies against every prescribed reconstruction test.
Corollary~\ref{cor:prescribed.projected.commuting.forms} and \eqref{eq:main.lifting.projected.derivative} therefore give
\[
  \DIFF_{\ell,T}^{k}\underline w_T
  =
  \piD{T}{k}\bigl((\DIFF\lambda_h)_{|T}\bigr)
  =
  \DIFF_{\ell,T}^{k}\underline v_T
  \qquad\forall T\in\mathcal T_h,
\]
which proves \eqref{eq:main.poincare.same.derivative}.

The $L^2$ stability of the cell and face projectors, the fixed-degree polynomial inverse trace inequality on $\Sh(T)$, and \eqref{eq:main.lifting.stability} yield
\[
\begin{aligned}
  \|\underline w_h\|_{0,h}^2
  &\leq
  \sum_{T\in\mathcal T_h}
  \left(
    \|\lambda_h\|_T^2
    +
    \sum_{F\in\mathcal F_T}
    h_F\|\trc_F\lambda_h\|_F^2
  \right)
  \\
  &\lesssim
  \|\lambda_h\|_\Omega^2
  \lesssim
  |\underline v_h|_{\DIFF,h}^2.
\end{aligned}
\]
This proves \eqref{eq:main.poincare.representative.estimate}.
The selection is linear because both the lifting and the interpolator are linear.
Finally, \eqref{eq:main.poincare.same.derivative} gives $\underline w_h-\underline v_h\in\Ker\DIFF_{\ell,h}^{k}$, so that
\[
  \inf_{\underline\phi_h\in\Ker\DIFF_{\ell,h}^{k}}
  \|\underline v_h+\underline\phi_h\|_{0,h}
  \leq
  \|\underline w_h\|_{0,h},
\]
and \eqref{eq:main.poincare.quotient} follows from \eqref{eq:main.poincare.representative.estimate}.
\end{proof}

\begin{remark}[Uniformity of the constants and auxiliary degrees]
\label{rem:uniformity.constants.degrees}
The auxiliary polynomial degrees are fixed before considering the mesh size.
The degree $m$ is chosen sufficiently large that $\Dpoly{T}{k}\subseteq\widehat{\mathcal R}_T^{k+1}=\star_T^{-1}\PLtrim{m}{n-k-1}(\Sh(T))$; the choice $m\geq\ell+1$ is sufficient.
The degree $m_{\mathcal S}$ is then chosen large enough for all terminal and closed-moment realisations to lie in $\PLtrim{m_{\mathcal S}-1}{k+1}(\Sh(f))$, while the corresponding local potentials lie in $\PLtrim{m_{\mathcal S}}{k}(\Sh(f))$.
These choices include the volume level and are possible because only finitely many cell dimensions and fixed-degree local polynomial problems occur.
Both degrees are independent of the mesh size, the number of cells, and the hybrid unknown.
Thus the lifting is computable through finite-dimensional cochain and polynomial problems with fixed local polynomial degrees.

For these fixed degrees and local polynomial spaces, the only non-local analytical contribution to $C_{\mathrm{lift}}^{k,\ell}$ is the cellular cochain Poincar\'e constant $C_{\mathrm{coch}}^k$ from Lemma~\ref{lem:cochain.poincare}.
The constant $C_{\mathrm{hyb}}^{k,\ell}$ in Corollary~\ref{cor:hybrid.poincare} can be chosen as a uniformly bounded interpolation factor times $C_{\mathrm{lift}}^{k,\ell}$.
All remaining factors arise from uniformly controlled coface-link inverses and local finite-dimensional polynomial operations.
The link constants are controlled by the uniformly bounded combinatorial complexity of the coface links, while the polynomial operations are controlled by scaling and reference-element norm equivalences.
These local factors are not tracked individually, and no attempt is made to optimise either the auxiliary degrees or $C_{\mathrm{lift}}^{k,\ell}$.
\end{remark}

\begin{remark}[Relation with DDR conforming liftings]
\label{rem:relation.with.ddr.lifting}
The recursive nature of the construction is closely related to the conforming liftings developed for Discrete de Rham complexes in \cite{Di-Pietro.Droniou.Pitassi:25}.
In the DDR setting, a discrete $k$-form has polynomial components on all cells of dimensions $d\in\{k,\ldots,n\}$, and the discrete exterior derivative and potential reconstruction are defined recursively over the cell dimension.
The conforming DDR lifting uses this complete cellular datum to build a conforming representative with projection properties on cells of every dimension and is a right inverse of the DDR interpolator; see \cite[Theorem~3 and Remark~5]{Di-Pietro.Droniou.Pitassi:25}.
The local boundary value problem in Lemma~\ref{lem:first.local.problem} is one of the shared ingredients of the two constructions.

The hybrid space $\uXh{k}$ contains only cell polynomial forms and face trace unknowns.
The intermediate cellular data are generated from $\underline v_h$ through the functional recursion and the terminal cellular cochain solve, and are then realised by a compatible polynomial hierarchy.
The lifting of Theorem~\ref{thm:main.hybrid.poincare} preserves the projected exterior derivative in \eqref{eq:main.lifting.projected.derivative} and is controlled by $|\underline v_h|_{\DIFF,h}$, but it does not in general reproduce the original hybrid unknown under interpolation.
Indeed, the stability estimate forces the lifting to vanish on every hybrid datum with zero canonical seminorm.

The relation with DDR is thus at the level of the hierarchical construction and the local polynomial problems.
An explicit identification with an auxiliary DDR unknown would require specifying the corresponding moment data and verifying their compatibility with the DDR reconstructions and differential.
\end{remark}

\section{A Poincar\'e inequality modulo conforming closed forms}
\label{sec:poincare.conforming.closed}

The lifting of Theorem~\ref{thm:main.hybrid.poincare} preserves the prescribed reconstructed exterior derivative after projection and is controlled by the canonical hybrid seminorm.
We now prove a complementary estimate for the broken cell polynomial field itself, with quotient space given by the continuous conforming kernel $\mathcal Z_{\mathrm c}^k(\Omega)$ defined in \eqref{eq:def.conforming.closed.space}.

This is a closed-data realisation of the cellular-to-hybrid transfer principle announced in the Introduction.
It retains the coface-link geometry, terminal cellular Poincar\'e correction, and local polynomial realisation mechanisms of the stable skeleton construction.
In Section~\ref{sec:poincare.proof}, our construction uses the full functional double complex to propagate a general reconstructed exterior-derivative datum.
Here the reference cell forms are already closed, so the construction reduces to a descent of their trace incompatibility followed by closed local extensions.
In particular, the vertical propagation becomes trivial, while the horizontal coface-link descent, the terminal cellular correction, and the upward polynomial completion remain.

We first extract the cellwise closed polynomial component of the hybrid unknown.
Its incompatibility is then propagated from the cell traces to the $k$-skeleton, where a single cellular cochain Poincar\'e problem removes the terminal obstruction.
The resulting data are finally completed upward by closed local trace extensions.
Thus the only non-local operation remains the cochain solve of Lemma~\ref{lem:cochain.poincare}.

\subsection{Statement of the conforming-kernel estimate}
\label{subsec:poincare.conforming.closed.statement}

\poincareconformingclosed*

The proof is split into two ingredients.
The first one follows from the local polynomial analysis of Section~\ref{sec:hho.forms}.
The second one is a closed-data version of the descent--ascent construction of Section~\ref{sec:poincare.proof}.

\subsection{Stable reduction to cellwise closed forms}
\label{subsec:stable.cellwise.closed.reduction}

Let $\mathcal F_h^\circ\coloneqq\{F\in\mathcal F_h\mid\card(\cof{n}(F))=2\}$ be the set of interior mesh faces.
For every $F\in\mathcal F_h^\circ$, fix an arbitrary ordering $\cof{n}(F)=\{T_F^+,T_F^-\}$.
For a family of cellwise closed forms $z_h=(z_T)_{T\in\mathcal T_h}$ with $z_T\in\PL{\ell}{k}(T)\cap\Ker\DIFF_T$, define the seminorm by
\[
  J_{\mathrm{cl},h}^k(z_h)^2
  \coloneqq
  \sum_{F\in\mathcal F_h^\circ}
  h_F^{-1}
  \left\|
    \trc_Fz_{T_F^+}
    -
    \trc_Fz_{T_F^-}
  \right\|_F^2.
\]
The chosen ordering of the adjacent cells is irrelevant to this seminorm.

\begin{proposition}[Stable reduction to cellwise closed forms]
\label{prop:stable.cellwise.closed.reduction}

For every $\underline v_h\in\uXh{k}$, let $z_T\in\PL{\ell}{k}(T)\cap\Ker\DIFF_T$ be the linear selection given by Proposition~\ref{prop:stable.polynomial.kernel.decomposition}.
Then
\begin{equation}
\label{eq:cellwise.closed.volume.error}
  \sum_{T\in\mathcal T_h}
  \|v_T-z_T\|_T^2
  \lesssim
  |\underline v_h|_{\DIFF,h}^2,
\end{equation}
and
\begin{equation}
\label{eq:cellwise.closed.face.residual}
  \sum_{T\in\mathcal T_h}
  \sum_{F\in\mathcal F_T}
  h_F^{-1}
  \|\trc_Fz_T-v_F\|_F^2
  \lesssim
  |\underline v_h|_{\DIFF,h}^2.
\end{equation}
In particular,
\begin{equation}
\label{eq:closed.jump.controlled.by.hybrid.energy}
  J_{\mathrm{cl},h}^k(z_h)^2
  \lesssim
  |\underline v_h|_{\DIFF,h}^2.
\end{equation}
\end{proposition}

\begin{proof}

Estimate \eqref{eq:cellwise.closed.volume.error} follows immediately from \eqref{eq:stable.polynomial.kernel.decomposition}, after summing over the mesh and using $h_T\leq\operatorname{diam}(\Omega)$.

For every $T\in\mathcal T_h$ and $F\in\mathcal F_T$, $\trc_Fz_T-v_F=\trc_F(z_T-v_T)+(\trc_Fv_T-v_F)$.
The first term is controlled by Proposition~\ref{prop:stable.polynomial.kernel.decomposition} and the uniform comparability of incident mesh sizes, whereas the second one is controlled by Proposition~\ref{prop:full.jump.control.forms}.
Squaring, multiplying by $h_F^{-1}$, and summing proves \eqref{eq:cellwise.closed.face.residual}.

Finally, on every interior face $F=T_F^+\cap T_F^-$, $\trc_Fz_{T_F^+}-\trc_Fz_{T_F^-}=(\trc_Fz_{T_F^+}-v_F)-(\trc_Fz_{T_F^-}-v_F)$.
Estimate \eqref{eq:closed.jump.controlled.by.hybrid.energy} therefore follows from \eqref{eq:cellwise.closed.face.residual}.
\end{proof}

\subsection{Stable conforming closure of cellwise closed forms}
\label{subsec:stable.conforming.closure}
Starting from a family that is closed on every cell, the next proposition constructs a globally conforming closed form by correcting only the incompatibility of the traces.

\begin{proposition}[Stable conforming closure of cellwise closed forms]
\label{prop:stable.conforming.closure}

Let $z_h=(z_T)_{T\in\mathcal T_h}$ with $z_T\in\PL{\ell}{k}(T)\cap\Ker\DIFF_T$.
There exist a fixed integer $m_{\mathrm c}\geq\ell+2$ and a form
\[
  z_h^{\mathrm c}
  \in\mathcal Z_{\mathrm c}^k(\Omega)\cap\PLtrim{m_{\mathrm c}}{k}(\Sh),
\]
depending linearly on $z_h$, such that
\begin{equation}
\label{eq:stable.conforming.closure}
  \sum_{T\in\mathcal T_h}
  \left\|
    z_T-\left.z_h^{\mathrm c}\right|_T
  \right\|_T^2
  \lesssim
  J_{\mathrm{cl},h}^k(z_h)^2.
\end{equation}
\end{proposition}
\begin{proof}
The proof consists of a reference descent, one terminal cochain correction, and a closed upward completion.

\medskip
\noindent\emph{Step 1: reference descent.}

Set $\overline z_T\coloneqq z_T$ for every $T\in\mathcal T_h$.
Then, recursively for $d \in \{n-1,\ldots,k\}$ and $f\in\FM{d}(\Mh)$, define
\begin{equation}
\label{eq:reference.closed.descent}
  \overline z_f
  \coloneqq
  \frac{1}{
    \card(\cof{d+1}(f))
  }
  \sum_{g\in\cof{d+1}(f)}
  \trc_f\overline z_g.
\end{equation}
Trace commutation with the exterior derivative gives $\DIFF_f\overline z_f=0$ for all $f\in\FM{d}(\Mh)$, $d\in\{k,\ldots,n\}$.
All these forms belong to fixed-degree polynomial spaces.
After choosing $m_{\mathrm c}\geq\ell+2$, we regard them as elements of $\PLtrim{m_{\mathrm c}}{k}(\Sh(f))$.

For every $d\in\{k,\ldots,n-1\}$, introduce the quantity
\[
  E_d^{\mathrm{ref}}
  \coloneqq
  \sum_{f\in\FM{d}(\Mh)}
  \sum_{g\in\cof{d+1}(f)}
  h_f^{n-d-2}
  \|
    \trc_f\overline z_g-\overline z_f
  \|_f^2.
\]
When $d=0$, the norm of a $0$-form on a vertex is understood as its absolute value.

At the top level, a boundary face has a single $n$-coface and gives no contribution.
If $F\in\mathcal F_h^\circ$, each difference $\trc_F\overline z_T-\overline z_F$ is, up to sign, one half of the jump between the traces from the adjacent cells.

Consequently,
\begin{equation}
\label{eq:reference.top.defect}
  E_{n-1}^{\mathrm{ref}}
  \lesssim
  J_{\mathrm{cl},h}^k(z_h)^2.
\end{equation}

We next prove recursively that
\begin{equation}
\label{eq:reference.defect.descent}
  E_d^{\mathrm{ref}}
  \lesssim
  E_{d+1}^{\mathrm{ref}}
  \qquad
  \forall d\in\{k,\ldots,n-2\}.
\end{equation}
Fix $f\in\FM{d}(\Mh)$ and set $u_g\coloneqq\trc_f\overline z_g$ for every $g\in\cof{d+1}(f)$.
The definition \eqref{eq:reference.closed.descent} says that $\overline z_f$ is the arithmetic mean of the family $(u_g)_g$.
The vertices of the coface link $\operatorname{Lk}_{\Mh}(f)$ correspond to the $(d+1)$-cofaces of $f$, whereas its edges correspond to the $(d+2)$-cofaces.
The link has positive dimension and vanishing reduced cohomology in degree zero, so its $1$-skeleton is connected.
Its cardinality is uniformly bounded; see Subsection~\ref{subsubsec:coface.link.complexes} and Lemma~\ref{lem:coface.exactness}.
The standard degree-zero Poincar\'e inequality on this $1$-skeleton then gives
\[
  \sum_{g\in\cof{d+1}(f)}
  \|u_g-\overline z_f\|_f^2
  \lesssim
  \sum_{G\in\cof{d+2}(f)}
  \|u_{g_G^+}-u_{g_G^-}\|_f^2,
\]
where $g_G^+$ and $g_G^-$ are the two $(d+1)$-cofaces of $f$
contained in $\partial G$.
For each such $G$, adding and subtracting $\trc_f\overline z_G$ and using the triangle inequality gives
\[
  \|u_{g_G^+}-u_{g_G^-}\|_f^2
  \lesssim
  \sum_{\substack{
    g\in\FM{d+1}(\partial G)\\
    f\subset g
  }}
  \|u_g-\trc_f\overline z_G\|_f^2.
\]
Consequently,
\begin{equation}
\label{eq:reference.graph.poincare}
  \sum_{g\in\cof{d+1}(f)}
  \|u_g-\overline z_f\|_f^2
  \lesssim
  \sum_{G\in\cof{d+2}(f)}
  \sum_{\substack{
    g\in\FM{d+1}(\partial G)\\
    f\subset g
  }}
  \left\|
    u_g-\trc_f\overline z_G
  \right\|_f^2.
\end{equation}
By transitivity of traces, $u_g-\trc_f\overline z_G=\trc_f(\overline z_g-\trc_g\overline z_G)$.
The fixed-degree polynomial inverse trace inequality, together with the mesh-regularity relation $h_f\simeq h_g$, therefore gives
\[
  h_f^{n-d-2}
  \left\|
    u_g-\trc_f\overline z_G
  \right\|_f^2
  \lesssim
  h_g^{n-d-3}
  \left\|
    \overline z_g-\trc_g\overline z_G
  \right\|_g^2.
\]
Multiplying \eqref{eq:reference.graph.poincare} by $h_f^{n-d-2}$, summing over $f$, and using the uniformly bounded number of subcells of each mesh cell proves \eqref{eq:reference.defect.descent}.
Combining \eqref{eq:reference.top.defect} and \eqref{eq:reference.defect.descent}, we obtain
\begin{equation}
\label{eq:reference.all.defects}
  E_d^{\mathrm{ref}}
  \lesssim
  J_{\mathrm{cl},h}^k(z_h)^2
  \qquad
  \forall d\in\{k,\ldots,n-1\}.
\end{equation}

\medskip
\noindent\emph{Step 2: terminal cochain correction.}

Define the terminal $(k+1)$-cochain
\[
  \eta_f
  \coloneqq
  \sum_{e\in\FM{k}(\partial f)}
  \epsilon_{f,e}
  \int_e\overline z_e,
  \qquad
  f\in\FM{k+1}(\Mh).
\]
For $k=0$, integration over a vertex is evaluation of the $0$-form, consistently with the positive orientation fixed in Subsection~\ref{sec:setting:mesh}.
Since $\DIFF_f\overline z_f=0$, Stokes' formula yields
\begin{equation}
\label{eq:terminal.cochain.from.reference.defect}
  \eta_f
  =
  \sum_{e\in\FM{k}(\partial f)}
  \epsilon_{f,e}
  \int_e
  \left(
    \overline z_e-\trc_e\overline z_f
  \right).
\end{equation}
Cauchy--Schwarz, scaling on the $k$-cells, and the uniform comparability of incident mesh sizes give
\[
  \sum_{f\in\FM{k+1}(\Mh)}
  h_f^{n-2k-2}|\eta_f|^2
  \lesssim
  E_k^{\mathrm{ref}}.
\]

By definition, $\eta$ is a cellular coboundary.
Hence Lemma~\ref{lem:cochain.poincare} provides the minimum-norm $k$-cochain $\widehat\lambda=(\widehat\lambda_e)_{e\in\FM{k}(\Mh)}$ such that
\begin{equation}
\label{eq:terminal.cochain.correction}
  \sum_{e\in\FM{k}(\partial f)}
  \epsilon_{f,e}\widehat\lambda_e
  =
  \eta_f
  \qquad
  \forall f\in\FM{k+1}(\Mh),
\end{equation}
and
\[
  \sum_{e\in\FM{k}(\Mh)}
  h_e^{n-2k}|\widehat\lambda_e|^2
  \lesssim
  E_k^{\mathrm{ref}}.
\]

As in the terminal realisation leading to \eqref{eq:terminal.form.stability}, fix on every $k$-cell a normalised top form
\[
  \omega_e
  \in
  \PLtrim{m_{\mathrm c}}{k}(\Sh(e)),
  \qquad
  \int_e\omega_e=1,
  \qquad
  \|\omega_e\|_e^2
  \lesssim
  h_e^{-k}.
\]
For $k=0$, this means $\omega_e=1$.
Set $\lambda_e\coloneqq-\widehat\lambda_e\omega_e$ and $\tau_e \coloneqq\overline z_e+\lambda_e$.
The forms $\tau_e$ are closed, and \eqref{eq:terminal.cochain.from.reference.defect}-- \eqref{eq:terminal.cochain.correction} give
\begin{equation}
\label{eq:terminal.closed.compatibility}
  \sum_{e\in\FM{k}(\partial f)}
  \epsilon_{f,e}
  \int_e\tau_e
  =
  0
  \qquad
  \forall f\in\FM{k+1}(\Mh).
\end{equation}
Moreover,
\begin{equation}
\label{eq:terminal.closed.correction.stability}
  \sum_{e\in\FM{k}(\Mh)}
  h_e^{n-k}\|\lambda_e\|_e^2
  \lesssim
  E_k^{\mathrm{ref}}.
\end{equation}
For $k=0$, the left-hand side is understood as $\sum_{e\in\FM{0}(\Mh)}h_e^n|\lambda_e|^2$.

\medskip
\noindent\emph{Step 3: closed upward completion.}

We now construct $\lambda_f$ and $\tau_f=\overline z_f+\lambda_f$ successively for $d\in \{k+1,\ldots,n\}$.
Suppose that the compatible closed forms $\tau_e$ have already been constructed on every cell of dimension $d-1$.
For $f\in\FM{d}(\Mh)$, define the boundary family
\begin{equation}
\label{eq:def.closed.boundary.correction}
  \left.\theta_{\partial f}\right|_e
  \coloneqq
  \tau_e-\trc_e\overline z_f
  \qquad
  \forall e\in\FM{d-1}(\partial f).
\end{equation}
The induction hypothesis and transitivity of traces show that this family is compatible on $\partial f$.
Both terms in \eqref{eq:def.closed.boundary.correction} are closed, and therefore $\DIFF_{\partial f}\theta_{\partial f}=0$.
If $d=k+1$, the terminal compatibility condition follows from \eqref{eq:terminal.closed.compatibility} and Stokes' formula, which gives
\[
\begin{aligned}
  \int_{\partial f}\theta_{\partial f}
  &=
  \sum_{e\in\FM{k}(\partial f)}
  \epsilon_{f,e}\int_e\tau_e
  -
  \int_{\partial f}\trc_{\partial f}\overline z_f
  \\
  &=
  -\int_f\DIFF_f\overline z_f
  =
  0.
\end{aligned}
\]
Thus the hypotheses of Lemma~\ref{lem:first.local.problem}, applied with polynomial parameter $m_{\mathrm c}-1$ and zero differential datum, are satisfied.
We select the minimum-norm solution
\begin{equation}
\label{eq:closed.upward.local.problem}
  \lambda_f
  \in
  \PLtrim{m_{\mathrm c}}{k}(\Sh(f)),
  \qquad
  \DIFF_f\lambda_f=0,
  \qquad
  \trc_{\partial f}\lambda_f
  =
  \theta_{\partial f}.
\end{equation}
It depends linearly on the boundary datum and satisfies
\begin{equation}
\label{eq:closed.upward.local.stability}
  \|\lambda_f\|_f^2
  \lesssim
  h_f
  \|\theta_{\partial f}\|_{\partial f}^2.
\end{equation}
Set $\tau_f \coloneqq \overline z_f+\lambda_f$.
Then $\DIFF_f\tau_f=0$ and $\trc_e\tau_f=\tau_e$ for all $e\in\FM{d-1}(\partial f)$.

It remains to prove stability.
Define
\[
  E_d^{\mathrm{corr}}
  \coloneqq
  \sum_{f\in\FM{d}(\Mh)}
  h_f^{n-d}\|\lambda_f\|_f^2,
  \qquad
  d\in\{k,\ldots,n\},
\]
with the convention described after \eqref{eq:terminal.closed.correction.stability} when $k=d=0$.
Equation \eqref{eq:terminal.closed.correction.stability} gives
\begin{equation}
\label{eq:closed.correction.base.stability}
  E_k^{\mathrm{corr}}
  \lesssim
  E_k^{\mathrm{ref}}.
\end{equation}
For $d\geq k+1$, using \eqref{eq:def.closed.boundary.correction} and \eqref{eq:closed.upward.local.stability}, we obtain
\[
\begin{aligned}
  E_d^{\mathrm{corr}}
  \lesssim{}&
  \sum_{f\in\FM{d}(\Mh)}
  \sum_{e\in\FM{d-1}(\partial f)}
  h_f^{n-d+1}\|\lambda_e\|_e^2
  \\
  &+
  \sum_{f\in\FM{d}(\Mh)}
  \sum_{e\in\FM{d-1}(\partial f)}
  h_f^{n-d+1}
  \|
    \overline z_e-\trc_e\overline z_f
  \|_e^2.
\end{aligned}
\]
The uniformly bounded number of cells incident on each subcell and the mesh-regularity relation $h_f\simeq h_e$ control the first sum by $E_{d-1}^{\mathrm{corr}}$.
For the second one, using additionally $h_e^2\leq\operatorname{diam}(\Omega)^2$, we obtain
\begin{equation}
\label{eq:closed.correction.recursive.stability}
  E_d^{\mathrm{corr}}
  \lesssim
  E_{d-1}^{\mathrm{corr}}
  +
  E_{d-1}^{\mathrm{ref}}.
\end{equation}
Iterating \eqref{eq:closed.correction.recursive.stability}, and using \eqref{eq:closed.correction.base.stability} and \eqref{eq:reference.all.defects}, gives
\begin{equation}
\label{eq:closed.correction.all.levels.stability}
  E_d^{\mathrm{corr}}
  \lesssim
  J_{\mathrm{cl},h}^k(z_h)^2
  \qquad
  \forall d\in\{k,\ldots,n\}.
\end{equation}

At the top level, $\overline z_T=z_T$ and $z_T-\tau_T=-\lambda_T$.
Hence the case $d=n$ of \eqref{eq:closed.correction.all.levels.stability} gives
\[
  \sum_{T\in\mathcal T_h}
  \|z_T-\tau_T\|_T^2
  \lesssim
  J_{\mathrm{cl},h}^k(z_h)^2.
\]
The traces of the top-dimensional forms $\tau_T$ agree on every interior face, so the piecewise form $z_h^{\mathrm c}$ defined by $\left.z_h^{\mathrm c}\right|_T\coloneqq\tau_T$ for every $T\in\mathcal T_h$ belongs to $H\Lambda^k(\Omega)$.
Its exterior derivative vanishes cellwise and hence globally.
All selections are linear. The coface averages are linear, the cochain solution in Lemma~\ref{lem:cochain.poincare} is the minimum-norm one, and the local solutions in \eqref{eq:closed.upward.local.problem} are also selected with minimum norm.
This proves \eqref{eq:stable.conforming.closure}.
\end{proof}

\subsection{Proof of the conforming-kernel Poincar\'e inequality}
\label{subsec:proof.poincare.conforming.closed}

\begin{proof}[Proof of Theorem~\ref{thm:poincare.conforming.closed}]

Let $(z_T)_{T\in\mathcal T_h}$ be given by Proposition~\ref{prop:stable.cellwise.closed.reduction}.
Apply Proposition~\ref{prop:stable.conforming.closure} to the resulting family $z_h=(z_T)_{T\in\mathcal T_h}$, and denote the selected conforming closed form by $z_h^{\mathrm c}$.
Using \eqref{eq:cellwise.closed.volume.error}, \eqref{eq:closed.jump.controlled.by.hybrid.energy}, and \eqref{eq:stable.conforming.closure}, we obtain
\[
\begin{aligned}
  \sum_{T\in\mathcal T_h}
  \left\|
    v_T-\left.z_h^{\mathrm c}\right|_T
  \right\|_T^2
  &\lesssim
  \sum_{T\in\mathcal T_h}
  \|v_T-z_T\|_T^2
  +
  \sum_{T\in\mathcal T_h}
  \left\|
    z_T-\left.z_h^{\mathrm c}\right|_T
  \right\|_T^2
  \\
  &\lesssim
  |\underline v_h|_{\DIFF,h}^2.
\end{aligned}
\]
This proves \eqref{eq:poincare.conforming.closed.representative}.
Since the selected $z_h^{\mathrm c}$ is admissible in the infimum in \eqref{eq:poincare.modulo.conforming.closed.space}, the quotient estimate follows.
\end{proof}

\begin{remark}[Relation between the two Poincar\'e inequalities]
\label{rem:poincare.closed.implies.hybrid}
Theorem~\ref{thm:poincare.conforming.closed} also implies Corollary~\ref{cor:hybrid.poincare}.
Indeed, let $z_h^{\mathrm c}$ be the conforming closed form supplied by that theorem and set
$\underline w_h\coloneqq\underline v_h-\uI{k}{h}z_h^{\mathrm c}$.
Corollary~\ref{cor:prescribed.projected.commuting.forms} gives
\[
  \DIFF_{\ell,h}^{k}\underline w_h
  =
  \DIFF_{\ell,h}^{k}\underline v_h,
\]
since $\DIFF z_h^{\mathrm c}=0$.
The fixed polynomial degree of $z_h^{\mathrm c}$ ensures that its interpolator is well defined.

Set $e_T\coloneqq v_T-\left.z_h^{\mathrm c}\right|_T$.
For every $T\in\mathcal T_h$ and $F\in\mathcal F_T$, the face component satisfies
\[
  w_F
  =
  v_F-\piQ{F}{k}\trc_Fv_T
  +
  \piQ{F}{k}\trc_Fe_T.
\]
The $L^2$ stability of the projectors and the fixed-degree inverse trace inequality on $\Sh(T)$ therefore yield
\[
\begin{aligned}
  \|\underline w_h\|_{0,h}^2
  &\lesssim
  \sum_{T\in\mathcal T_h}\|e_T\|_T^2
  +
  \sum_{T\in\mathcal T_h}
  \sum_{F\in\mathcal F_T}
  h_F\|v_F-\piQ{F}{k}\trc_Fv_T\|_F^2
  \\
  &\lesssim
  |\underline v_h|_{\DIFF,h}^2.
\end{aligned}
\]
The last inequality follows from
\eqref{eq:poincare.conforming.closed.representative}
and $h_F^2\leq\operatorname{diam}(\Omega)^2$.
The construction is linear and proves Corollary~\ref{cor:hybrid.poincare}, with a possibly different uniform constant.
\end{remark}

\begin{remark}[An alternative proof using distributional forms]
\label{rem:distributional.poincare.proof}
Theorem~\ref{thm:poincare.conforming.closed} also follows from the Poincar\'e--Friedrichs inequalities for discrete distributional forms in \cite[Section~6]{Christiansen.Licht:20}.
Consider the distributional complex on $\Sh$ with local spaces
$\PLtrim{\ell+1}{j}(S)$ for $j\in\{0,\ldots,n\}$.
Taking the full boundary as the relative subcomplex removes boundary jump terms and imposes no boundary condition on the conforming kernel.

Apply the inequality at the broken $k$-form space to $v_h$, and let
$z_h^{\mathrm c}\in\mathcal Z_{\mathrm c}^k(\Omega)\cap\PLtrim{\ell+1}{k}(\Sh)$
be its $L^2$ projection onto the kernel.
The distributional norm contains cell derivatives and interior trace jumps weighted by the inverse face size.
Since jumps inside each polytopal cell vanish, mesh regularity gives
\[
\begin{aligned}
  \|v_h-z_h^{\mathrm c}\|_\Omega^2
  &\lesssim
  \sum_{T\in\mathcal T_h}\|\DIFF_Tv_T\|_T^2+
  \sum_{F\in\mathcal F_h^\circ}
  h_F^{-1}
  \|\trc_Fv_{T_F^+}-\trc_Fv_{T_F^-}\|_F^2.
\end{aligned}
\]
Proposition~\ref{prop:full.jump.control.forms} bounds the right-hand side by
$|\underline v_h|_{\DIFF,h}^2$.

The difference between the two arguments is methodological.
The alternative proof uses the Poincar\'e constant of the distributional complex on $\Sh$ as an intermediate analytical input.
In \cite{Christiansen.Licht:20}, this constant is controlled through simplicial chain and Whitney-form Poincar\'e constants, together with local stability factors.
The present proof starts directly from the cellular estimate of Lemma~\ref{lem:cochain.poincare} on $\Mh$ and propagates its control through coface-link problems and local polynomial extensions.
The connection of this cellular estimate with Whitney forms is established in \cite[Section~3.4]{Di-Pietro.Droniou.ea:25}.
Thus the dependence on $C_{\mathrm{coch}}^k$ is obtained without invoking a global Poincar\'e inequality for distributional forms.
\end{remark}

\subsection{Poincar\'e inequalities in vector proxies in three space dimensions}
\label{sec:poincare.vector.proxies}

The vector proxies of Subsection~\ref{sec:vector.proxies} identify the three-dimensional cases with the hybrid gradient, curl, and divergence.
We now use this correspondence to state the consequences of Theorem~\ref{thm:poincare.conforming.closed} for zero-, one-, and two-forms.

We write $|\cdot|_{\operatorname{grad},h}$, $|\cdot|_{\operatorname{curl},h}$, and $|\cdot|_{\operatorname{div},h}$ for the respective vector-proxy expressions of $|\cdot|_{\DIFF,h}$ described in Subsection~\ref{sec:vector.proxies}.

We shall repeatedly use the elementary observation that, if $V$ is a subspace of $L^2\Lambda^k(\Omega)$ and $v\perp V$ in the $L^2$ product, then
\[
  \|v\|_{L^2(\Omega)}
  \leq
  \inf_{z\in V}
  \|v-z\|_{L^2(\Omega)}.
\]
Indeed, for every $z\in V$, $\|v\|_{L^2(\Omega)}^2=(v,v-z)_\Omega \leq \|v\|_{L^2(\Omega)}\|v-z\|_{L^2(\Omega)}$.

\subsubsection{Zero-forms and hybrid gradient Poincar\'e inequalities}
\label{subsec:poincare.proxy.gradient}

For $k=0$, the continuous conforming closed space is
\[
  \mathcal Z_{\mathrm c}^0(\Omega)
  =
  \left\{
    z\in H^1(\Omega):
    \operatorname{grad}z=0
  \right\}.
\]
If $\Omega$ is connected, then $\mathcal Z_{\mathrm c}^0(\Omega)=\mathbb R$.
For $k=0$, the broken cell polynomial $v_h$ introduced in Section~\ref{sec:hho.forms} is scalar, and Theorem~\ref{thm:poincare.conforming.closed} gives
\begin{equation}
\label{eq:hybrid.gradient.poincare.wirtinger}
  \inf_{c\in\mathbb R}
  \|v_h-c\|_{L^2(\Omega)}
  \leq
  C_{\mathrm{conf}}^{0,\ell}
  |\underline v_h|_{\operatorname{grad},h}.
\end{equation}

\begin{corollary}[Hybrid gradient Poincar\'e--Wirtinger inequality]
\label{cor:hybrid.gradient.poincare.wirtinger}
Assume that $\Omega$ is connected and that $(v_h,1)_\Omega=0$.
Then
\[
  \|v_h\|_{L^2(\Omega)}
  \leq
  C_{\mathrm{conf}}^{0,\ell}
  |\underline v_h|_{\operatorname{grad},h}.
\]
\end{corollary}

The zero-mean condition is precisely the $L^2$-orthogonality of $v_h$ to $\mathcal Z_{\mathrm c}^0(\Omega)=\mathbb R$.
The corollary therefore follows directly from \eqref{eq:hybrid.gradient.poincare.wirtinger} and the observation above.

The same estimate yields the Poincar\'e inequality associated with the mixed-order, $L^2$-scaled hybrid setting used in \cite[Eqs.~(31), (34), and~(50)]{Dalphin.Ducreux.Lemaire.Pitassi:25}.
Indeed, taking $r\geq1$ and $\ell=r$, restricting the cell component to $\mathcal P^{r-1}(T)$, choosing $\mathcal Q_F^{0,r}=\mathcal P^r(F)$, and setting $h_\flat\coloneqq\min_{T\in\mathcal T_h}h_T$, uniform equivalence of incident mesh sizes gives
\[
  h_\flat\|v_h\|_{L^2(\Omega)}
  \lesssim
  \left(
    \sum_{T\in\mathcal T_h}
    \left(
      h_T^2\|\operatorname{grad}v_T\|_T^2
      +
      \sum_{F\in\mathcal F_T}
      h_T\|v_{T|F}-v_F\|_F^2
    \right)
  \right)^{1/2}
\]
under the same zero-mean condition.

\subsubsection{One-forms and hybrid curl Poincar\'e inequalities}
\label{subsec:poincare.proxy.curl}

For $k=1$, the continuous conforming closed space has vector proxy
\[
  \boldsymbol{\mathcal Z}_{\mathrm c,\operatorname{curl}}(\Omega)
  \coloneqq
  \boldsymbol H(\operatorname{curl};\Omega)
  \cap
  \Ker(\operatorname{curl}).
\]
Under the one-form vector proxy, the broken cell polynomial is denoted by $\boldsymbol v_h$.
Theorem~\ref{thm:poincare.conforming.closed} gives
\begin{equation}
\label{eq:hybrid.curl.poincare.modulo.closed}
  \inf_{\boldsymbol z\in
    \boldsymbol{\mathcal Z}_{\mathrm c,\operatorname{curl}}(\Omega)}
  \|
    \boldsymbol v_h-\boldsymbol z
  \|_{L^2(\Omega)}
  \leq
  C_{\mathrm{conf}}^{1,\ell}
  |\underline{\boldsymbol v}_h|_{\operatorname{curl},h}.
\end{equation}

For the second Weber setting, introduce the normal harmonic space
\[
  \boldsymbol{\mathcal H}_{\mathrm n}(\Omega)
  \coloneqq
  \left\{
    \boldsymbol z
    \in
    \boldsymbol H(\operatorname{curl};\Omega)
    \cap
    \boldsymbol H_0(\operatorname{div};\Omega):
    \operatorname{curl}\boldsymbol z=0,
    \quad
    \operatorname{div}\boldsymbol z=0
  \right\}.
\]
The topological description recalled in \cite[Eq.~(12)]{Lemaire.Pitassi:25} gives
\begin{equation}
\label{eq:continuous.curl.kernel.decomposition}
  \boldsymbol{\mathcal Z}_{\mathrm c,\operatorname{curl}}(\Omega)
  =
  \operatorname{grad}
  \left(
    H^1(\Omega)\cap L_0^2(\Omega)
  \right)
  \oplus
  \boldsymbol{\mathcal H}_{\mathrm n}(\Omega).
\end{equation}

\begin{corollary}[Second hybrid Weber inequality in the curl gauge]
\label{cor:second.hybrid.weber.closed.gauge}
Let $n=3$, $k=1$, and assume that
\[
  \left(
    \boldsymbol v_h,
    \boldsymbol z
  \right)_\Omega
  =
  0
  \qquad
  \forall
  \boldsymbol z
  \in
  \operatorname{grad}
  \left(
    H^1(\Omega)\cap L_0^2(\Omega)
  \right)
  \oplus
  \boldsymbol{\mathcal H}_{\mathrm n}(\Omega).
\]
Then
\begin{equation}
\label{eq:second.hybrid.weber.closed.gauge}
  \|\boldsymbol v_h\|_{L^2(\Omega)}
  \leq
  C_{\mathrm{conf}}^{1,\ell}
  |\underline{\boldsymbol v}_h|_{\operatorname{curl},h}.
\end{equation}
\end{corollary}

By \eqref{eq:continuous.curl.kernel.decomposition}, the gauge condition in the corollary is exactly the $L^2$-orthogonality of $\boldsymbol v_h$ to $\boldsymbol{\mathcal Z}_{\mathrm c,\operatorname{curl}}(\Omega)$.
Hence \eqref{eq:second.hybrid.weber.closed.gauge} follows directly from \eqref{eq:hybrid.curl.poincare.modulo.closed} and the observation above.

The same argument applies in the weighted $L^2$ product.
If $0<\mu_{\flat}\leq\mu\leq\mu_{\sharp}$ almost everywhere in $\Omega$ and
\[
  (\mu\boldsymbol v_h,\boldsymbol z)_\Omega=0
  \qquad
  \forall\boldsymbol z\in
  \boldsymbol{\mathcal Z}_{\mathrm c,\operatorname{curl}}(\Omega),
\]
then
\[
  \|\mu^{1/2}\boldsymbol v_h\|_{L^2(\Omega)}
  \leq
  \mu_{\sharp}^{1/2}
  C_{\mathrm{conf}}^{1,\ell}
  |\underline{\boldsymbol v}_h|_{\operatorname{curl},h}.
\]
Thus the gauge argument introduces no additional stability constant.
The unit-coefficient estimate uses exactly $C_{\mathrm{conf}}^{1,\ell}$, while the weighted variant only adds the explicit factor $\mu_{\sharp}^{1/2}$.
For piecewise constant $\mu$, this is the weighted useful variant used in the HHO stability analysis of magnetostatic and div--curl systems in \cite{Lemaire.Pitassi:25,Dalphin.Ducreux.Lemaire.Pitassi:25}.

Estimate \eqref{eq:second.hybrid.weber.closed.gauge} is the unit-coefficient form of the gauge-fixed second hybrid Weber estimate in \cite[Remark~16]{Lemaire.Pitassi:25}.
The construction gives a different proof of the discrete stability step, with the global information transferred through the polynomial face skeleton rather than through an operator-specific Helmholtz--Hodge decomposition.

\subsubsection{Two-forms and hybrid divergence Poincar\'e inequalities}
\label{subsec:poincare.proxy.divergence}

For $k=2$, the continuous conforming closed space has vector proxy
\[
  \boldsymbol{\mathcal Z}_{\mathrm c,\operatorname{div}}(\Omega)
  \coloneqq
  \boldsymbol H(\operatorname{div};\Omega)
  \cap
  \Ker(\operatorname{div}).
\]
Under the two-form vector proxy, the broken cell polynomial is again denoted by $\boldsymbol v_h$.
Theorem~\ref{thm:poincare.conforming.closed} gives
\begin{equation}
\label{eq:hybrid.divergence.poincare.modulo.closed}
  \inf_{\boldsymbol z\in
    \boldsymbol{\mathcal Z}_{\mathrm c,\operatorname{div}}(\Omega)}
  \|
    \boldsymbol v_h-\boldsymbol z
  \|_{L^2(\Omega)}
  \leq
  C_{\mathrm{conf}}^{2,\ell}
  |\underline{\boldsymbol v}_h|_{\operatorname{div},h}.
\end{equation}

\begin{corollary}[Hybrid divergence Poincar\'e inequality]
\label{cor:hybrid.divergence.poincare}
Assume that
\[
  \left(
    \boldsymbol v_h,
    \boldsymbol z
  \right)_\Omega
  =
  0
  \qquad
  \forall
  \boldsymbol z
  \in
  \boldsymbol{\mathcal Z}_{\mathrm c,\operatorname{div}}(\Omega).
\]
Then
\[
  \|\boldsymbol v_h\|_{L^2(\Omega)}
  \leq
  C_{\mathrm{conf}}^{2,\ell}
  |\underline{\boldsymbol v}_h|_{\operatorname{div},h}.
\]
\end{corollary}

The assumed orthogonality is exactly the $L^2$-orthogonality of $\boldsymbol v_h$ to $\boldsymbol{\mathcal Z}_{\mathrm c,\operatorname{div}}(\Omega)$.
The corollary therefore follows directly from \eqref{eq:hybrid.divergence.poincare.modulo.closed} and the observation above.

\section{Conclusions and perspectives}
\label{sec:conclusions}

We have established a uniform stable conforming lifting and a uniform discrete Poincar\'e inequality for Hybrid High-Order differential forms on general polyhedral meshes, at every form degree, in arbitrary space dimension, and on domains with arbitrary topology.
The first controls both the conforming form and its exterior derivative in $L^2$, while the projection of the exterior derivative agrees with the prescribed reconstruction.
The second controls the broken cell polynomial form modulo the continuous conforming kernel of the exterior derivative.
Both estimates are uniform with respect to the mesh size and to every stability-admissible choice of the face spaces, whose essential local requirement is the inclusion of the traces of the polynomial kernel.
Interpolating the lifting also gives the hybrid Poincar\'e inequality of Corollary~\ref{cor:hybrid.poincare}.

The two results are complementary realisations of the same cellular-to-hybrid transfer principle.
For general reconstructed differential data, the missing intermediate-dimensional structure is generated through the functional double complex, transferred to a stable compatible polynomial skeleton, and completed in a conforming finite element space.
For cellwise closed data, the vertical part of this construction is no longer needed. The trace defect is descended to the terminal cellular cochain, corrected by the same Poincar\'e solve, and completed upward through closed local extensions.
In both cases, $C_{\mathrm{coch}}^k$ is the only non-local analytical stability constant, while the remaining factors entering $C_{\mathrm{lift}}^{k,\ell}$ and $C_{\mathrm{conf}}^{k,\ell}$ arise from uniformly controlled local coface-link and finite element operations.

In three dimensions, the conforming-kernel estimate recovers the hybrid Poincar\'e--Wirtinger inequality for the gradient, the second hybrid Weber estimate under the standard gauge, and the corresponding divergence-kernel estimate.
Its weighted curl consequence is the form entering the stability analysis of HHO magnetostatic and div--curl discretisations on domains with general topology.
These consequences show that stability results previously expressed through operator-specific arguments can be organised within a single exterior-calculus mechanism without identifying the reconstructed kernel with the continuous conforming kernel.

The transfer viewpoint also suggests a natural treatment of boundary conditions.
Replacing the absolute cellular cochain Poincar\'e problem by its relative counterpart should lead, through the same local-to-global architecture, to conforming-kernel estimates with vanishing traces.
For zero-forms, this would recover the standard HHO Poincar\'e inequality with homogeneous Dirichlet boundary conditions, whereas in the three-dimensional one-form case it would yield the analogous estimate with homogeneous tangential boundary conditions, corresponding to the setting of the first Weber inequality.
Further directions include mixed boundary conditions, extensions to cell- and face-dependent polynomial degrees, and the transfer of Sobolev-type or other functional inequalities to hybrid spaces.
The relation with discrete distributional differential form complexes also suggests studying analogous constructions for broader classes of non-conforming and skeletal discretisations.


\section*{Acknowledgements}
The author acknowledges the funding of the European Union via the MSCA EffECT, project number 101146324.


\bibliographystyle{unsrturl}
\bibliography{references}

@article{Ern.Zanotti:20:smoother,
  author  = {Ern, Alexandre and Zanotti, Pietro},
  title   = {A quasi-optimal variant of the hybrid high-order method for elliptic partial differential equations with {$H^{-1}$} loads},
  journal = {IMA Journal of Numerical Analysis},
  year    = {2020},
  volume  = {40},
  number  = {4},
  pages   = {2163--2188},
  doi     = {10.1093/imanum/drz057}
}

@article{Dong.Ern:26,
  author  = {Dong, Zhaonan and Ern, Alexandre},
  title   = {$\boldsymbol{H}(\operatorname{curl})$-reconstruction of piecewise polynomial fields with application to {$hp$} a posteriori nonconforming error analysis for {Maxwell}'s equations},
  journal = {SIAM Journal on Numerical Analysis},
  volume  = {64},
  number  = {4},
  pages   = {1390--1413},
  year    = {2026},
  doi     = {10.1137/25M1789433}
}

@misc{Di-Pietro.Droniou:25b,
  author        = {Di Pietro, Daniele A. and Droniou, J{\'e}r{\^o}me},
  title         = {From Finite Elements to Hybrid High-Order methods},
  year          = {2025},
  eprint        = {2503.00425},
  archivePrefix = {arXiv},
  primaryClass  = {math.NA},
  doi           = {10.48550/arXiv.2503.00425}
}

@book{Hatcher:02,
  author    = {Hatcher, Allen},
  title     = {Algebraic Topology},
  publisher = {Cambridge University Press},
  address   = {Cambridge},
  year      = {2002},
  isbn      = {978-0-521-79540-1}
}

@book{Kozlov:08,
  author    = {Kozlov, Dmitry N.},
  title     = {Combinatorial Algebraic Topology},
  series    = {Algorithms and Computation in Mathematics},
  volume    = {21},
  publisher = {Springer},
  address   = {Berlin, Heidelberg},
  year      = {2008},
  isbn      = {978-3-540-71961-8}
}

@article{Arnold.Falk.ea:06,
  author  = {Arnold, Douglas N. and Falk, Richard S. and Winther, Ragnar},
  title   = {Finite Element Exterior Calculus, Homological Techniques, and Applications},
  journal = {Acta Numerica},
  volume  = {15},
  pages   = {1--155},
  year    = {2006},
  doi     = {10.1017/S0962492906210018}
}

@article{Arnold.Falk.ea:09,
  author  = {Arnold, Douglas N. and Falk, Richard S. and Winther, Ragnar},
  title   = {Geometric Decompositions and Local Bases for Spaces of Finite Element Differential Forms},
  journal = {Computer Methods in Applied Mechanics and Engineering},
  volume  = {198},
  number  = {21--26},
  pages   = {1660--1672},
  year    = {2009},
  doi     = {10.1016/j.cma.2008.12.017}
}

@article{Arnold.Falk.ea:10,
  author  = {Arnold, Douglas N. and Falk, Richard S. and Winther, Ragnar},
  title   = {Finite Element Exterior Calculus: From {Hodge} Theory to Numerical Stability},
  journal = {Bulletin of the American Mathematical Society},
  volume  = {47},
  number  = {2},
  pages   = {281--354},
  year    = {2010},
  doi     = {10.1090/S0273-0979-10-01278-4}
}

@book{Arnold:18,
  author    = {Arnold, Douglas N.},
  title     = {Finite Element Exterior Calculus},
  series    = {CBMS-NSF Regional Conference Series in Applied Mathematics},
  volume    = {93},
  publisher = {Society for Industrial and Applied Mathematics},
  address   = {Philadelphia},
  year      = {2018},
  isbn      = {978-1-61197-553-6},
  doi       = {10.1137/1.9781611975543}
}

@article{Berchenko-Kogan:21,
  author        = {Berchenko-Kogan, Yakov},
  title         = {Duality in Finite Element Exterior Calculus and {Hodge} Duality on the Sphere},
  journal       = {Foundations of Computational Mathematics},
  volume        = {21},
  number        = {5},
  pages         = {1153--1180},
  year          = {2021},
  doi           = {10.1007/s10208-020-09478-5},
  eprint        = {1906.06354},
  archivePrefix = {arXiv},
  primaryClass  = {math.NA}
}

@article{Bonaldi.Di-Pietro.ea:24,
  author        = {Bonaldi, Francesco and Di Pietro, Daniele A. and Droniou, J{\'e}r{\^o}me and Hu, Kaibo},
  title         = {An Exterior Calculus Framework for Polytopal Methods},
  journal       = {Journal of the European Mathematical Society},
  year          = {2025},
  note          = {Published online first},
  doi           = {10.4171/JEMS/1602},
  eprint        = {2303.11093},
  archivePrefix = {arXiv},
  primaryClass  = {math.NA}
}

@article{Chave.Di-Pietro.Lemaire:22,
  author  = {Chave, Florent and Di Pietro, Daniele A. and Lemaire, Simon},
  title   = {A Discrete {Weber} Inequality on Three-Dimensional Hybrid Spaces with Application to the {HHO} Approximation of Magnetostatics},
  journal = {Mathematical Models and Methods in Applied Sciences},
  volume  = {32},
  number  = {1},
  pages   = {175--207},
  year    = {2022},
  doi     = {10.1142/S0218202522500051}
}

@article{Christiansen.Licht:20,
  author  = {Christiansen, Snorre H. and Licht, Martin W.},
  title   = {{Poincar{\'e}--Friedrichs} Inequalities of Complexes of Discrete Distributional Differential Forms},
  journal = {BIT Numerical Mathematics},
  volume  = {60},
  number  = {2},
  pages   = {345--371},
  year    = {2020},
  doi     = {10.1007/s10543-019-00784-1}
}

@book{Cicuttin.Ern.Pignet:21,
  author    = {Cicuttin, Matteo and Ern, Alexandre and Pignet, Nicolas},
  title     = {Hybrid High-Order Methods: A Primer with Applications to Solid Mechanics},
  series    = {SpringerBriefs in Mathematics},
  publisher = {Springer},
  address   = {Cham},
  year      = {2021},
  doi       = {10.1007/978-3-030-81477-9}
}

@article{Dalphin.Ducreux.Lemaire.Pitassi:25,
  author  = {Dalphin, J{\'e}r{\'e}my and Ducreux, Jean-Pierre and Lemaire, Simon and Pitassi, Silvano},
  title   = {{Hybrid High-Order} Approximations of Div--Curl Systems on Domains with General Topology},
  journal = {SIAM Journal on Scientific Computing},
  volume  = {47},
  number  = {5},
  pages   = {A2988--A3016},
  year    = {2025},
  doi     = {10.1137/25M1744356}
}

@misc{Di-Pietro.Droniou.Pitassi:25,
  author        = {Di Pietro, Daniele A. and Droniou, J{\'e}r{\^o}me and Pitassi, Silvano},
  title         = {Conforming Lifting and Adjoint Consistency for the {Discrete de Rham} Complex of Differential Forms},
  year          = {2025},
  eprint        = {2509.21449},
  archivePrefix = {arXiv},
  primaryClass  = {math.NA},
  doi           = {10.48550/arXiv.2509.21449}
}

@misc{Di-Pietro.Droniou.ea:25,
  author        = {Di Pietro, Daniele A. and Droniou, J{\'e}r{\^o}me and Hanot, Marien-Lorenzo and Pitassi, Silvano},
  title         = {Uniform {Poincar{\'e}} Inequalities for the Discrete {de Rham} Complex of Differential Forms},
  year          = {2025},
  eprint        = {2501.16116},
  archivePrefix = {arXiv},
  primaryClass  = {math.NA},
  doi           = {10.48550/arXiv.2501.16116}
}

@book{Di-Pietro.Droniou:20,
  author    = {Di Pietro, Daniele A. and Droniou, J{\'e}r{\^o}me},
  title     = {The {Hybrid High-Order} Method for Polytopal Meshes: Design, Analysis, and Applications},
  series    = {Modeling, Simulation and Applications},
  volume    = {19},
  publisher = {Springer},
  address   = {Cham},
  year      = {2020},
  doi       = {10.1007/978-3-030-37203-3}
}

@article{Di-Pietro.Ern.Lemaire:14,
  author  = {Di Pietro, Daniele A. and Ern, Alexandre and Lemaire, Simon},
  title   = {An Arbitrary-Order and Compact-Stencil Discretization of Diffusion on General Meshes Based on Local Reconstruction Operators},
  journal = {Computational Methods in Applied Mathematics},
  volume  = {14},
  number  = {4},
  pages   = {461--472},
  year    = {2014},
  doi     = {10.1515/cmam-2014-0018}
}

@article{Di-Pietro.Ern:15,
  author  = {Di Pietro, Daniele A. and Ern, Alexandre},
  title   = {A Hybrid High-Order Locking-Free Method for Linear Elasticity on General Meshes},
  journal = {Computer Methods in Applied Mechanics and Engineering},
  volume  = {283},
  pages   = {1--21},
  year    = {2015},
  doi     = {10.1016/j.cma.2014.09.009}
}

@article{Di-Pietro.Hanot:24,
  author  = {Di Pietro, Daniele A. and Hanot, Marien-Lorenzo},
  title   = {Uniform {Poincar{\'e}} Inequalities for the Discrete {de Rham} Complex on General Domains},
  journal = {Results in Applied Mathematics},
  volume  = {23},
  pages   = {100496},
  year    = {2024},
  doi     = {10.1016/j.rinam.2024.100496}
}

@article{Falk.Winther:14,
  author  = {Falk, Richard S. and Winther, Ragnar},
  title   = {Local Bounded Cochain Projections},
  journal = {Mathematics of Computation},
  volume  = {83},
  number  = {290},
  pages   = {2631--2656},
  year    = {2014},
  doi     = {10.1090/S0025-5718-2014-02827-5}
}

@article{Falk.Winther:15,
  author  = {Falk, Richard S. and Winther, Ragnar},
  title   = {Double Complexes and Local Cochain Projections},
  journal = {Numerical Methods for Partial Differential Equations},
  volume  = {31},
  number  = {2},
  pages   = {541--551},
  year    = {2015},
  doi     = {10.1002/num.21922}
}

@article{Lemaire.Pitassi:25,
  author  = {Lemaire, Simon and Pitassi, Silvano},
  title   = {Discrete {Weber} Inequalities and Related {Maxwell} Compactness for Hybrid Spaces over Polyhedral Partitions of Domains with General Topology},
  journal = {Foundations of Computational Mathematics},
  volume  = {25},
  number  = {3},
  pages   = {725--763},
  year    = {2025},
  doi     = {10.1007/s10208-024-09648-9}
}

@article{Licht:17,
  author  = {Licht, Martin W.},
  title   = {Complexes of Discrete Distributional Differential Forms and Their Homology Theory},
  journal = {Foundations of Computational Mathematics},
  volume  = {17},
  number  = {4},
  pages   = {1085--1122},
  year    = {2017},
  doi     = {10.1007/s10208-016-9315-y}
}

\end{document}